\documentclass[12pt]{article}
\usepackage{amsmath}
\usepackage{graphicx}
\usepackage{enumerate}
\usepackage{url} 

\newcommand{\blind}{0}

\usepackage{amsmath, amsthm, amssymb}
\usepackage{graphicx}
\graphicspath{{Images/}{Figures/Mean/Stationary/}{Figures/Mean/2D/}{Figures/Mean/1D/}{Figures/tstat/2D/}{Figures/}{Figures/Tiffs/}{Figures/fMRI/}}
\usepackage[round]{natbib}
\usepackage{float}

\usepackage{subcaption}
\newtheorem{thm}{Theorem}[section]
\newtheorem{lem}[thm]{Lemma}
\newtheorem{cor}[thm]{Corollary}
\newtheorem{ass}{Assumption}
\newtheorem{prop}[thm]{Proposition}

\theoremstyle{definition}
\newtheorem{rmk}{Remark}

\newtheorem{defn}[thm]{Definition}

\usepackage{hyperref}

\usepackage{xr}
\title{Confidence regions for the location of peaks of a smooth random field}
\author{Samuel Davenport, Thomas E. Nichols, Armin Schwartzman\\}
\date{\vspace{-1cm}}

\usepackage{./mymacros}
\usepackage{./packagemacros}

\def\jasa{0}

\pdfoutput=1

\begin{document}
\maketitle

\def\spacingset#1{\renewcommand{\baselinestretch}%
{#1}\small\normalsize} \spacingset{1}

\begin{abstract}
	Local maxima of random processes are useful for finding important regions and are routinely used, for summarising features of interest (e.g. in neuroimaging).  In this work we provide confidence regions for the location of local maxima of the mean and standardized effect size (i.e. Cohen's $ d $) given multiple realisations of a random process. We prove central limit theorems for the location of the maximum of mean and $ t $-statistic random fields and use these to provide asymptotic confidence regions for the location of peaks of the mean and Cohen's $ d $. Under the assumption of stationarity we develop Monte Carlo confidence regions for the location of peaks of the mean that have better finite sample coverage than regions derived based on classical asymptotic normality. We illustrate our methods on 1D MEG data and 2D fMRI data from the UK Biobank.\\
\noindent%
{\it Keywords: Random fields, spatial statistics, monte carlo, local maxima.}

\if\jasa1
\newpage
\fi
\end{abstract}
\spacingset{1.2} 
\section{Introduction}
Detecting the presence of local maxima, i.e. peaks, in a random field in order to determine areas where the mean is non-zero, is an approach used in a number of research domains including neuroimaging (\cite{Chumbley2009b}, \cite{Chumbley2010}) and astrophysics, \citep{Cheng2017a}. Inference is typically performed using theoretical results for zero mean random fields where the goal is to determine whether or not a particular peak is high enough to reject the null hypothesis that the mean is zero. These peak detection procedures were formalized in \cite{Schwartzman2011}, \cite{Cheng2017a}: taking advantage of the peak height distributions derived in \cite{Cheng2015a} for stationary zero mean Gaussian random fields. 

There has been a large amount of work investigating properties of zero mean stationary random fields (see e.g. the seminal work of \cite{Adler1981} or alternatively \cite{Adler2010} for a more recent overview), however when the mean is non-zero it becomes more difficult to prove useful results. Simultaneous confidence bands for the mean and Cohen's $ d $ have been derived in 1D using resampling approaches and the Gaussian Kinematic Formula (\cite{TelschowSCB}, \cite{Telschow2020Delta}, \cite{Liebl2019}). \cite{Sommerfield2018} derived asymptotic sets which provide confidence regions for sets of locations of the signal that lie above a specified threshold (with application to climate data). This framework been applied in neuroimaging to characterize areas of activation (\cite{Bowring2019}, \cite{Bowring2020}). 


In this work we derive asymptotic confidence regions for the location of peaks of the signal in non-zero mean random fields. In the neuroimaging setting, this can be used to determine where the highest peaks of activation are most likely to lie. This type of inference can enable studies to be combined correctly using meta-analysis (\cite{Eickhoff2012}, \cite{Radua2012}) and so that peak locations can be compared across different studies. Neuroimaging meta analyses require confidence regions for peak location and unbiased estimates of the effect sizes at peaks (we provided a selective inference resampling based approach to estimating unbiased peak effect sizes in \cite{Davenport2020}).


Our strategy for obtaining confidence regions is based on the observation that the problem of estimating the location of a peak of a random field is mathematically very similar to that of finding the location of the maximum of the likelihood function. Standard asymptotic theory for the latter gives a central limit theorem (CLT) for the distribution of the observed maximum about the true maximum. \cite{Amemiya1985} extended these results to the more general framework of extremum estimators, see also \cite{Hayashi2000}. We will take advantage of \cite{Amemiya1985}'s framework to derive CLTs for the local maxima of mean and $ t $-statistic fields. To do so we derive the asymptotic distribution of the derivative for mean and $ t $-statistic fields (giving an exact form when the underlying fields are Gaussian) and show that the scaled second derivative converges almost surely. Combining these results yields asymptotic confidence regions for the true peak locations of the underlying mean and Cohen's $ d $. 

We find that the asymptotic results require a larger number of samples than is available in typical datasets. To deal with this problem, under an assumption of stationarity, we use Monte Carlo simulation of the joint distribution between the first and second derivatives to obtain confidence regions for the location of peaks of the mean. We shall show that these Monte Carlo confidence regions have improved finite sample coverage relative to the asymptotic confidence regions that are obtained from the CLT. The failure of asymptotic results in the finite sample has been discussed a little in the literature (e.g. see \cite{Braunstein1992}) however most papers assume that a sufficient sample size is reached such that the CLT can be assumed to hold. As we shall demonstrate, this is not necessarily reasonable in practice; for instance in the context of neuroimaging.

\if\blind0
Random Field and peak location inference were conducted using the RFTtoolbox \url{https://github.com/sjdavenport/RFTtoolbox}. Code used to perform the simulations and real data analyses, is available here: \url{https://github.com/sjdavenport/peakcrs}.
\else
Code used to perform the simulations and real data analyses is attached in the submission. Random field and peak location inference was conducted using the RFTtoolbox \citep{Davenport2022Ravi}.
\fi
\if\jasa0
Proofs and further theoretical and simulation results are contained in the appendices.
\else
The supplementary material, sections of which will be denoted with an S suffix, contains proofs and further theoretical and simulation results.
\fi

\section{Model Set-up and Assumptions}\label{S:MSUA}

\subsection{Notation}\label{SS:setup}
Throughout we will take $ (Y_n)_{n \in \mathbb{N}} $ to be i.i.d almost surely differentiable random fields (\cite{Adler2007}) on some bounded open domain $ S \subset \mathbb{R}^D $, $ D \in \mathbb{N}$ (here $\mathbb{N} $ denotes the set of positive integers). We shall assume that $  Y_n = \mu + \sigma\epsilon_n  $ for some bounded functions $\mu, \sigma: S \longrightarrow \mathbb{R} $, $ \inf_{s \in S}\sigma(s) > 0 $ and zero mean, variance 1 i.i.d random fields $ (\epsilon_n)_{n \in \mathbb{N}}$ on $ S $. 

Throughout we will perform operations on random fields (such as addition, multiplication, division) pointwise. Given $ N \in \mathbb{N} $ samples we define the sample mean field as $ \hat{\mu}_N = \frac{1}{N} \sum_{n = 1}^N Y_n $ and the sample variance field to be $ \hat{\sigma}^2_N = \frac{1}{N-1} \sum_{n = 1}^N (Y_n- \hat{\mu}_N)^2. $
We will refer to the fields $ (Y_n)_{n \in \mathbb{N}} $ as the \textbf{component fields}. We can then define the $ t $-statistic field to be
\begin{equation}\label{eq:Tdef}
T_N = \frac{\sqrt{N}\hat{\mu}_N}{\hat{\sigma}_N}
= \frac{\sqrt{N}\mu + \sigma Z_N}{\sqrt{\frac{\sigma^2}{N-1}V_N}}
\end{equation}
where $ Z_N := \frac{1}{\sqrt{N}}\sum_{n = 1}^N \epsilon_n$ and $ V_N := \frac{1}{\sigma^2}\sum_{n = 1}^N (Y_n - \hat{\mu}_N)^2 $. If the component fields are Gaussian then this field is a non-central $ t $-field with $ N-1 $ degrees of freedom. We define the Cohen's $ d $ field to be $ d_N = T_N/\sqrt{N}. $ As $ N $ tends to infinity the Cohen's $ d $ field converges uniformly almost surely to $ d = \mu/\sigma $ (see Lemma \ref{lem:CDconv}). 

Given a differentiable function $ f:S \rightarrow \mathbb{R}^{D'} $, for $ s \in S $, we shall write $ \nabla f(s) \in \mathbb{R}^{D' \times D}$ to denote the gradient of $ f $ at $ s $ and use $ \nabla^T f(s) $ to denote $ (\nabla f(s))^T $. If $ f $ is twice differentiable and $ D' = 1 $ we will write $ \nabla^2 f(s)$ denote the Hessian of $ f $ at $ s. $ 
\noindent When they are defined, let $ \Lambda := \cov(\nabla^T Y_1) \text{ and } \Gamma := \mathbb{E}\left[ (Y_1 - \mu)(\nabla Y_1)^T \right] $. 

We will use $ \convd $ and $ \convp $ to denote convergence in distribution/probability respectively and use $ \convud, \convup $ and $ \convuas $ to denote uniform convergence over $ S $ in distribution/ probability/almost surely respectively.  When using this notation if the convergence occurs as the parameter $ N $ converges to infinity we will, for ease of notation, often omit the statement $ N \rightarrow \infty $ when this is otherwise clear. We shall in general write $ a.s. $ to mean almost sufrely. Moreover we will write $ \left\lVert \cdot \right\rVert $ to denote the euclidean norm.

Our main objects of interest will be peaks of random fields which are defined rigorously in the following definition. Here, given $ r> 0$ and $ s \in S$, we take $ B_r(s) $ be the $D$-dimensional ball of radius $ r $ that is centred at $ s$.

\begin{defn}
	Given a twice differentiable $ f: S \rightarrow \mathbb{R} $, we say that $ s \in S $ is a \textbf{critical point} of $ f $ if $ \nabla f(s) = 0 $. Given a critical point $ s $, we define $ s $ to be a \textbf{local maximum} of $ f $ if there is some $ r > 0 $ such that $ f(s) = \sup_{t \in B_r(s)}f(t) $ and call a local maximum $ s $ \textbf{non-degenerate} if $\nabla^2 f(s) \prec 0 $ (here we write $ A \prec 0$ to denote that $ A $ is a negative definite matrix). Local minima (and their non-degeneracy) can be defined similarly. We will use the word \textbf{peaks} to refer to local maxima and minima.
\end{defn}

\subsection{Derivative Exchangeability}
In what follows we will need to be able to exchange expectation and differentiation. This can be done when the random fields have the following property.
\begin{defn}
	We say that a random field $ f: S \longrightarrow \mathbb{R}^{D'} $, some $ D' \in \mathbb{N} $, is $ L_1- $\textbf{Lipschitz} \textbf{at} $ \mathbf{s} \in S$ if there exists an integrable real random variable $ L $ and some ball $ B(s) \subset S $ centred at $ s $ such that $ \left\lVert f(t) - f(s) \right\rVert  \leq L \left\lVert t-s \right\rVert \text{ for all } t\in B(s). $ We define $ f $ to be $ L_1- $\textbf{Lipschitz} on a subset $ S' \subset S $ if it is $ L_1- $Lipschitz at $ s $ for all $ s \in S'. $ If $ S' = S $ then we will not specify the subset.
\end{defn}
More generally we will say that a differentiable random field $ f $ on $ S $ satisfies the \textbf{DE (derivative exchangeability) condition} at $ s\in S $ if $ \mathbb{E}\left[ f(t) \right] $ is differentiable at $ t = s $ and $ \mathbb{E}\left[ \nabla f(t) \right] = \nabla \mathbb{E}\left[ f(t) \right] $, i.e. such that we can exchange the integral and derivative. We say that $ f $ satisfies the \textbf{DE condition} on $ S $ if this holds for all $ s \in S $. Arguing as in the proof of \cite{TelschowSCB}'s Lemma 4, in the following lemma we show that it is sufficient for $ f $ to be $ L_1 $-Lipschitz for this to hold. 
\begin{lem}\label{lem:DEcond}
	Let $ f:S \rightarrow \mathbb{R}^{D'} $ be an a.s. differentiable random field that is $ L_1 $-Lipschitz at $ s \in S$. Then $ f $ satisfies the DE condition at $ s $.
\end{lem}
Because of the mean value inequality, finiteness of the expected value of the supremum of the local derivative is a sufficient condition for $ L_1 $-Lipschitzness and so the following lemma follows immediately.
\begin{lem}\label{lem:supbound}
	Let $f$ be a random field on $ S $ which is a.s. differentiable on some ball $ B(s) \subset S $, centred at $ s \in S $. If $ \mathbb{E} \sup_{t \in B(s)} \left\lVert \nabla f(t) \right\rVert  < \infty $ then $ f $ is $ L_1 $-Lipschitz at $ s. $
\end{lem}

Applying the above lemma, using the fact that moments of the maximum of the absolute value of an a.s. continuous Gaussian field on a separable space are finite \citep{Landau1970}, we obtain the following result.

\begin{prop}\label{prop:Gaussfin}
	Suppose that $ f: S \rightarrow \mathbb{R} $ is an a.s. $ C^1 $ Gaussian random field. Then, for all $ k \in \mathbb{N} $, $ \mathbb{E} \sup_{t \in B(s)} \left\lVert \nabla f(t)^k \right\rVert < \infty. $ Thus $ f^k $ is $ L_1 $-Lipschitz on $ S $ and therefore satisfies the DE condition on $ S. $
\end{prop}
The $ L_1 $-Lipschitz condition is also satisfied by the broad class of convolution fields. These fields can be used to control the familywise error rate using the Gaussian Kinematic Formula (\cite{TelschowFWER}, \cite{DavenportThesis}) and help to bridge the gap between data on a lattice and theory describing continuous random fields. As we shall see, in Sections \ref{SS:fmri} and \ref{S:MEGspectra}, these arise naturally in applications. They are defined as follows.

\begin{defn}\label{defn:cfield}
	Given observations $ X $ on a lattice $ \mathcal{V} \in \mathbb{R}^{D}$ and some continuous kernel function $ K: \mathbb{R}^D \rightarrow \mathbb{R}$, we define the \textbf{convolution field} $ Y:S \rightarrow \mathbb{R}$ as the random field sending $ s \in S $ to $ 	Y(s) = \sum_{l \in \mathcal{V}} K(s-l) X(l). $
\end{defn}
The following proposition shows that convolution fields are Lipschitz continuous under reasonable conditions.
\begin{prop}\label{prop:convfield}
	Let $ Y $ be a $ D $-dimensional convolution field on $ S $ generated from observations $ X $ on a finite lattice $ \mathcal{V} $ and using a kernel $ K $ which is Lipschitz. If $ \mathbb{E}\left[ \left| X(l) \right| \right] < \infty$ for all $ l \in \mathcal{V} $, then $ Y $ is $ L_1 $-Lipschitz. In particular, if $ K $ is differentiable then $ Y $ satisfies DE condition on $ S. $
\end{prop}
\subsection{Assumptions on the Noise}
In what follows we will need to be able to guarantee that the DE condition holds for the first and second derivatives and apply the functional strong law of large numbers (fSLLN). To ensure that we can do this we will want to impose some or all of the following conditions on a random field $ f:S \rightarrow \mathbb{R} $ in order to control the behaviour of the noise.
\begin{ass}\label{ass:derivexch}
	\phantom{x}
	\begin{enumerate}[label=\alph*.]
		\item $ f $ is a.s. twice continuously differentiable and for all $ s \in S, \cov(\nabla^T f(s))$ is finite.
		\item $(i) \,\mathbb{E}\left[ \sup_{s\in S}\left| f(s) \right| \right], (ii) \,\mathbb{E}\left[ \sup_{j,s\in S}\left|  f_j(s) \right| \right]$ and $ (iii) \,\mathbb{E}\left[ \sup_{j,k,s\in S}\left|  f_{jk}(s) \right| \right]$ are finite.
		\item  $(i) \, \mathbb{E}\left[ \sup_{s \in S} f(s)^2 \right] $, $(ii) \, \mathbb{E}\left[ \sup_{j,s \in S} f_j(s)^2 \right]$ and $(iii)\, \mathbb{E}\left[ \sup_{j,k,s \in S} f_{jk}(s)^2 \right] $ are finite.
	\end{enumerate}
\end{ass}
In \ref{ass:derivexch}b and \ref{ass:derivexch}c the suprema are taken over all $ 1 \leq  j, k \leq D$. If $ f $ satisfies Assumption \ref{ass:derivexch} then the DE condition holds for $ f, f^2 $ and their first derivatives on $ S $, by Lemmas \ref{lem:supbound} and \ref{lem:DEcond} (and applying Cauchy-Swartz). Condition $ \ref{ass:derivexch}b $ allows the conditions of the fSLLN to hold for $ f $ and its derivatives and similarly \ref{ass:derivexch}c allows the conditions of the fSLLN to hold for $ f^2 $ and its derivatives. For $ f $ this follows immediately from \cite{Ledoux2013}'s Corollary 7.10 and Section $ \ref{AA:fSLLN} $ provides an explanation of why this and the DE condition holds for $ f^2 $. Note that \ref{ass:derivexch}c implies \ref{ass:derivexch}b but we have stated these assumptions separately as, at times, we will only need to require that \ref{ass:derivexch}b holds. 

In order for Assumption \ref{ass:derivexch} to hold, it is sufficient that the field $ f $ be a $ C^2 $ Gaussian field (see Proposition \ref{prop:Gaussfin} and its proof) or that the field is a convolution field derived from a finite lattice with twice continuously differentiable kernel and finite observation expectation (as shown formally below).


\begin{prop}\label{prop:ass1conv}
	Let $ Y $ be a convolution field, defined as in Definition \ref{defn:cfield}, that is restricted to $ S $. Suppose $  \var(X(l)) < \infty$ for each $ l \in \mathcal{V} $ and that $ K $ is $ C^2 $, then $ Y $ satisfies Assumption \ref{ass:derivexch}.
\end{prop}

\section{Local convergence of the number of peaks}\label{S:LCPOTNOP}
Our first set of results relate to the number of peaks that lie within small regions around each critical point. We show that in the signal plus noise model, defined below, the number of peaks of the observed field, in small regions around peaks of the signal, converges in probability to 1. Furthermore, the number of critical points of the field, away from critical points of the signal, tends to 0 in probability. To do so we introduce Assumption \ref{ass:peakconv} and require that the first two derivatives of the noise converge uniformly in probability to zero as $ N \rightarrow \infty $. These results are important because they show that given a large enough sample size we can be sure that any observed peak is a true peak of the underlying signal, providing peak identifiability.

\subsection{Assumptions on the Signal}
In order to prove identifiability (Proposition \ref{prop:peakconv}) we need to make some assumptions. To do so, consider the general  signal plus noise model $ \hat{\gamma}_N = \gamma + \eta_N $, for $ N \in \mathbb{N} $, where $ \gamma:S \rightarrow \mathbb{R}$ is a fixed function and $ \eta_N $ are $ D $-dimensional random fields on $ S $. Unlike in the standard signal plus noise model, we do not require the $ \eta_N $ to have mean zero. Instead, we will impose specific conditions on $ \gamma $ and $ \eta_N $ and use this more general structure to describe both mean and Cohen's $ d $ fields. We will want $ \gamma $ to be sufficiently nice in terms of smoothness and will want the derivatives of $ \eta_N $ to converge to zero in probability as $ N \rightarrow \infty. $ In particular we will impose the following conditions  on $ \gamma $ to ensure identifiability.
\begin{ass}\label{ass:peakconv}
	\quad
	\begin{enumerate}[label=(\alph*)]
		\item $ \gamma $ is $ C^2 $ and has $ J \in \mathbb{N}$ critical points at locations $ \theta_1, \dots, \theta_J \in S$, such that for $ j = 1, \dots, J $ there exist non-overlapping compact balls $ B_j \subset S $ with radii $\delta_j$ such that $ \theta_j $ lies in the interior of $B_j$. Let $ B_{\text{all}} = \bigcup_j B_j $ and assume that $ 		C := \inf_{t \in S \setminus B_{\text{all}}} \left\lVert \nabla \gamma(t) \right\rVert  > 0. $
		\item Let $ P_{\max}$ and $P_{\min} $ be the subsets of $ \left\lbrace 1, \dots, J \right\rbrace $ corresponding to the non-degenerate local maxima and minima of $ \gamma $, respectively. Define
		\begin{equation*}
		B_{\max} = \bigcup_{j \in P_{\max}} B_j \text{ and } B_{\min} = \bigcup_{j \in P_{\min}} B_j . 
		\end{equation*}
		Assume that $ D_{\max} := -\sup_{t \in B_{\max}}\sup_{\left\lVert x \right\rVert = 1} x^T \nabla^2 \gamma(t) x > 0  $ and that\\
		$ D_{\min} := -\sup_{t \in B_{\min}}\sup_{\left\lVert x \right\rVert = 1} x^T \nabla^2 \gamma(t) x < 0. $
	\end{enumerate}
\end{ass}
On $ S \setminus B $,  Assumption \ref{ass:peakconv}a ensures that $ \gamma $ is not flat, as critical points of $ \hat{\gamma} $ (albeit with low magnitudes) always have a chance of manifesting in regions where the signal $ \gamma $ is flat no matter how low the variance of the noise. Assumption \ref{ass:peakconv}b provides bounds on the eigenvalues of the Hessian of $ \gamma $, within the specified regions, and ensures that, for $ j = 1, \dots, J $, if $ \theta_j $ is a local maximum of $ \gamma $, then $ \nabla^2 \gamma(s) \prec 0 $ for all $ s \in B_j $ and so $ \theta_j = \argmax_{t \in B_j} \gamma(t) $ (and similar uniqueness holds if $ \theta_j $ is a local minimum). These bounds are needed to show that, with high probability, only one peak of $ \hat{\gamma}_N $ is found within each region corresponding to a maximum or a minimum. If $ \gamma $ is $ C^2 $ then the conditions on $ D_{\max} $ and $ D_{\min} $ follow immediately from peak non-degeneracy and by choosing $ B_j $ to be sufficiently small.\\
\noindent The sample mean field described in Section \ref{SS:setup} can be written as $ \hat{\mu}_N = \mu + \frac{\sigma}{N} \sum_{n = 1}^N \epsilon_n $
and so fits into the signal plus noise model very naturally. The Cohen's $ d $ field can also be written in this form simply by taking $ \gamma = \frac{\mu}{\sigma} $ and $ \eta_N = \left( d_N - \frac{\mu}{\sigma} \right) $. In order to prove the results in Section 4, we will require that $ \mu $ and $ \frac{\mu}{\sigma} $ satisfy Assumption \ref{ass:peakconv}. Moreover we show, in Section \ref{SS:MC2}, that $ \frac{\sigma}{N} \sum_{n = 1}^N \epsilon_n $ and $ \left( d_N - \frac{\mu}{\sigma} \right) $ and their derivatives converge to zero uniformly in probability.

\subsection{Identifiability of Peaks}
We will now illustrate what happens to the number of peaks of the random field when it is in the form of the signal plus noise model described in the previous section with signal satisfying Assumption \ref{ass:peakconv}. To do so we will require the following lemma which provides a bound for the probability that the derivative is non-zero in a signal plus noise model.

\begin{lem}\label{lem:derivbound}
	Given $ S' \subset \mathbb{R}^D $, let $\gamma:S' \rightarrow \mathbb{R} $ be differentiable and suppose that $ \hat{\gamma} $ is some estimate of $ \gamma $ on $ S' $ such that $ \hat{\gamma} = \gamma + \eta $  for some random field $ \eta $ on $ S' $.  Then, for any $C' \in \mathbb{R}$ such that $ C' \leq  \inf_{t \in S'} \left\lVert \nabla \gamma(t) \right\rVert $,
	\begin{equation*}
	\mathbb{P}\left( \inf_{t \in S'} \left\lVert \nabla \hat{\gamma}(t)  \right\rVert > 0  \right)
	\geq 1 - \mathbb{P}\left(  \sup_{t \in S'} \left\lVert \nabla \eta(t) \right\rVert > C'  \right).
	\end{equation*}
\end{lem}
Using this lemma we can prove the following proposition on identifiability which shows that the number of local maxima/minima of a non-central random field in a region around the true peak converges to 1 in probability and that the number of critical points outside of $ B_{\text{all}} $ converges to zero in probability. The proof adapts that of \cite{Cheng2017a}'s Lemma A.1.
\begin{prop}\label{prop:peakconv}
	Suppose that $ (\hat{\gamma}_N)_{N \in \mathbb{N}} $ is a sequence of random fields on $ S $ such that $ \hat{\gamma}_N = \gamma + \eta_N $  for some sequence of random fields $( \eta_N)_{N \in \mathbb{N}} $, such that $ \nabla\eta_N \convup 0$, and differentiable $\gamma:S \rightarrow \mathbb{R}$ which satisfies Assumption \ref{ass:peakconv}a. Suppose that for each $ N $, $ \eta_N $ is a.s. differentiable, then as $ N \longrightarrow \infty$,
	\begin{equation*}
	\mathbb{P}(\#\left\lbrace t \in S\setminus B_{\text{all}}: \nabla \hat{\gamma}_N(t) = 0 \right\rbrace = 0) \longrightarrow 1.
	\end{equation*} 
	Additionally assume that all the conditions of Assumption \ref{ass:peakconv} apply to $ \gamma $, and that $ \eta_N $ is a.s. $ C^2 $ with $ \nabla^2 \eta_N \convup 0 $, and let $ 	M_N = \left\lbrace t \in S: \nabla \hat{\gamma}_N(t) = 0 \text{ and } \nabla^2\hat{\gamma}_N(t) \prec 0 \right\rbrace $
	be the set of non-degenerate local maxima of $ \hat{\gamma}_N. $ Then, as $ N \longrightarrow \infty $, for each $ B_j $ containing a non-degenerate local maximum of $ \gamma $:
	\begin{equation*}
	\mathbb{P}(\#\left\lbrace t \in M_N \cap B_j\right\rbrace  = 1) \longrightarrow 1.
	\end{equation*}
\end{prop}
\noindent By symmetry an analogous result holds for the non-degenerate local minima. 
\begin{rmk}
	The condition that $ \eta_N $ is a.s. $ C^2 $ holds in a number of situations. In particular it will hold for any convolution field with a $ C^2 $ kernel derived from a finite lattice. Alternatively, if $ \eta_N $ is Gaussian then the conditions for this to hold are well studied: see for instance \cite{Adler1981}.
\end{rmk}

The requirement that $ \nabla\eta_N, \nabla^2\eta_N  \convup 0$ can be shown to hold in a number of reasonable settings. In Section \ref{S:peakidentifiability} we will show that it holds for mean and $ t$-statistic fields and in the context of the linear model. 

\section{Confidence Regions}\label{CRs}
In this section we will provide asymptotic confidence regions for the locations of peaks of the mean and Cohen's $ d $ of i.i.d random fields. To do so we will prove CLTs for peak location using the extremum estimator framework of \cite{Amemiya1985}, (see also \cite{Xiaoxia}). These results will allow us to build asymptotic confidence regions for peak location. When our random fields are stationary we will show that these can be improved to provide better coverage in the finite sample by taking advantage of the joint distribution between the first and second derivatives.\footnote{Note that in the results that follow the assumptions we state are typically made on the first field in the sequence. This of course means that they hold for all of the fields because of the i.i.d assumption.}

\subsection{Mean}\label{SS41}
Extending the extremum estimator framework to multiple peaks, we obtain the following theorem for the joint distribution of the peak location.
\begin{thm}\label{thm:meanclt}
	Let $ Y_1 $ satisfy Assumption \ref{ass:derivexch}a,b on $ S $ and have mean $ \mu $ which satisfies Assumption \ref{ass:peakconv}. For each $ j = 1, \dots, J$ corresponding to a non-degenerate local maximum of $ \mu $, let $ \hat{\theta}_{j,N} = \argmax_{t \in B_j} \hat{\mu}_N(t) $ (and for the minima let $ \hat{\theta}_{j,N} = \argmin_{t \in B_j} \hat{\mu}_N(t) $) and let 
	$ \boldsymbol{\hat{\theta}_{N}} := (\hat{\theta}_{1,N}^T, \dots, \hat{\theta}_{J,N}^T)^T $ and $ \bf{\theta} := (\theta_{1}^T, \dots, \theta_{J}^T)^T. $
	Then
	\begin{equation*}
	\sqrt{N}(\bf{\hat{\theta}_N} - \bf{\theta}) \convd \mathcal{N}(0, \bf{A}\bf{\Lambda}\bf{A}^T)
	\end{equation*} 
	as $ N \conv \infty $. Here $ \bf{A} $ $ \in \mathbb{R}^{DJ\times DJ}$ is a block diagonal matrix with $ J $ diagonal $ D\times D $ blocks such that, for $ j \in 1, \dots, J$, the $ j $th block is $ \left( \nabla^2 \mu(\theta_j) \right)^{-1}$ and $ \bf{\Lambda} \in \mathbb{R}^{DJ\times DJ}$ is a matrix such that for $ i,j = 1,\dots, J$ the $ ij $th $ D \times D$ block of $ \bf{\Lambda} $ is $ \cov(\nabla^T Y_1(\theta_i), \nabla^T Y_1(\theta_j)) $.
\end{thm}
\begin{proof}
	For each $ j = 1, \dots, J, $ Taylor expansion about $ \theta_j $ gives,
	\begin{equation}\label{eq:Taylormu}
	0 = \nabla \hat{\mu}_N(\hat \theta_{j,N}) = \nabla \hat{\mu}_N( \theta_j)  + (\hat \theta_{j,N} - \theta_j)^T\nabla^2 \hat\mu_N(\theta^*_{j,N})
	\end{equation}
	for some $ \theta^*_{j,N} \in B_{\left\lVert  \theta_j - \hat \theta_{j,N}\right\rVert}(\theta_j) $. As such
	\begin{equation}\label{eq:cltsort}
	\sqrt{N}(\hat{\theta}_{j,N} - \theta_j) = -\left( \nabla^2 \hat\mu_N(\theta^*_{j,N}) \right)^{-1}\left( \sqrt{N} \nabla^T \hat{\mu}_N(\theta_j) \right), 
	\end{equation}
	and in particular,
	\begin{equation*}
	\sqrt{N}\pmatrix{\hat\theta_{1,N} - \theta_1; \vdots; \hat\theta_{J,N} - \theta_J} = -\bf{A_N}\sqrt{N}\pmatrix{\nabla^T \hat{\mu}_N(\theta_1); \vdots; \nabla^T \hat{\mu}_N(\theta_J)}.
	\end{equation*}
	Here $ \bf{A_N} \in \mathbb{R}^{DJ\times DJ}$ is a block diagonal matrix with $ J $ diagonal blocks such that, for $ j \in 1, \dots, J$, the $ j $th block is $ \left( \nabla^2 \hat\mu_N(\theta^*_{j,N}) \right)^{-1}. $ As $ N \rightarrow \infty $, $\bf{A_N} $ converges in probability to $ \bf{A} $ and as such applying the multivariate CLT to the vector of gradients: $ \sqrt{N}\pmatrix{\nabla \hat{\mu}_N(\theta_1), \cdots, \nabla \hat{\mu}_N(\theta_J)}^T $ (which can be written as a sum) and Slutsky yields the result. See Section \ref{SS:meanthm} for the technical details.
\end{proof}
\noindent Considering the $ j $th peak, under the assumptions of Theorem \ref{thm:meanclt}, as $ N \rightarrow \infty, $
\begin{equation*}
\sqrt{N}(\hat{\theta}_{j,N} - \theta_j) \convd \mathcal{N}(0, (\nabla^2 \mu(\theta_j))^{-1}\Lambda(\theta_j)(\nabla^2 \mu(\theta_j))^{-1}).
\end{equation*}
We can obtain a similar asymptotic result for the coefficients in the linear model, see Theorem \ref{thm:lmmean}. 

\subsection{Cohen's $ d $}\label{SS:CD}
Obtaining an analogous result to Theorem \ref{thm:meanclt} for Cohen's $ d $ is a little more complicated. Proving this can be broken down into two main steps: proving a pointwise CLT for the distribution of the derivative of a $ t $-statistic field (see below) and proving convergence of the Hessian in probability (see Proposition \ref{prop:CDconv}).

\subsubsection{Distribution of the derivative of a $ T $-field}\label{SS:Tderivdist}
If our component random fields are Gaussian then we can obtain finite sample distributions for the derivatives of $ Y_n $. To do so we extend \cite{Worsley1994}'s Lemma 5.1a to non-central and non-stationary $ t $-fields to derive the distribution the gradient of the $ t $-statistic field $ T_N $ defined in equation \eqref{eq:Tdef}. We start by assuming that the variance is constant, which simplifies the expressions, and then use this to prove the general result.
\begin{lem}\label{lem:Tstatderiv}
	Let $ Y_1 $ be a unit variance Gaussian random field with differentiable mean $ \mu $. Assume that  $ \Lambda(s) = \cov(\nabla^T Y_1(s)) $ is finite for all $ s \in S $. Then for each $ s \in S $,
	\begin{equation*}
	\nabla T_N(s) | T_N(s), \hat{\sigma}_N(s)  =_d 
	\frac{1}{\hat{\sigma}_N(s)}\mathcal{N}\left(\sqrt{N} \nabla \mu(s), \left(1+\frac{T_N(s)^2}{N-1}\right)\Lambda(s)\right).
	\end{equation*}
\end{lem}
\noindent In particular, if the fields are Gaussian and have unit variance, it follows that we have the following pointwise CLT for Cohen's $ d $:
\begin{equation*}
\sqrt{N}\left(\nabla d_N -\frac{\nabla \mu}{\hat{\sigma}_N}\right)= \nabla T_N - \frac{\nabla \mu\sqrt{N}}{\hat{\sigma}_N}\underset{N\rightarrow \infty}{\convd} \mathcal{N}\left(0, \left( 1+ \mu^2 \right)\Lambda\right).
\end{equation*} 
This follows since $ \hat{\sigma}_N \convas \sigma $ and $ \frac{T_N^2}{N-1}\conv \mu^2 $ as $ N \rightarrow \infty $. Let us now drop the constant variance condition and write
\begin{align}\label{eq:modT}
T_N &= 
\frac{\frac{1}{\sqrt{N}}\sum_{n = 1}^N Y_n}{\left( \frac{1}{N-1}\sum_{n = 1}^N (Y_n - \frac{1}{N}\sum Y_n)^2 \right)^{1/2}}\\ 
&= \frac{\frac{1}{\sqrt{N}}\sum_{n = 1}^N Y_n/\sigma}{\left( \frac{1}{N-1}\sum_{n = 1}^N (Y_n/\sigma - \frac{1}{N}\sum Y_n/\sigma)^2 \right)^{1/2}}
\end{align}
which is the $ t $-statistic derived from component Gaussian random fields $ Y_n' = Y_n/\sigma $, $ n = 1, \dots, N, $ which are i.i.d and have constant variance 1. We can thus apply the constant variance result to yield the following corollary.
\begin{cor}\label{cor:Tstatderiv}
	Let $ Y_1$ be a Gaussian random field and suppose that $ \sigma $ and $ \mu $ are differentiable. Assume that  $ \Lambda(s) $ is finite for all $ s \in S $, then for each $ s \in S $,
	\begin{equation*}
	\nabla T_N(s) | T_N(s), \hat{\sigma}_N(s) =_d 
	\frac{\sigma(s)}{\hat{\sigma}_N(s)}\mathcal{N}\left(\sqrt{N} \nabla d(s), \left(1+\frac{T_N(s)^2}{N-1}\right)\Lambda'(s)\right),
	\end{equation*} 
	where 
	\begin{equation*}
	\Lambda' := \cov\left(\nabla^T \frac{Y_1}{\sigma}\right) = \frac{\Lambda}{\sigma^2} - \frac{\nabla^T \sigma^2 \Gamma}{\sigma^4} + \frac{\nabla^T \sigma^2(\nabla \sigma^2)}{4\sigma^4}.
	\end{equation*}
\end{cor}
\noindent As such for all $ s \in S $, as $ N \rightarrow \infty $
\begin{equation}\label{eq:GaussCDconv}
\sqrt{N}\left(\nabla d_N(s) - \frac{\sigma(s)\nabla (d(s))}{\hat{\sigma}_N(s)} \right) \convd \mathcal{N}\left(0, \left(1+ d(s)^2\right)\Lambda'(s)\right).
\end{equation} 
Examining this expression we see that when the mean of the component fields is zero, then in particular $ d = 0 $ and so the limiting variance does not depend on $ \sigma $, as we would expect, as neither does the $ t $-statistic. When the mean is non-zero, the dependence on $ \sigma $ is captured via Cohen's $ d $ as $ \Lambda' $ does not depend on $ \sigma. $ Note that if $ \sigma^2 = 1 $ everywhere then we recover the unit-variance expression.

\subsubsection{CLTs for the derivative of Cohen's $ d $ and peak location}
In practice the component fields may not be Gaussian. In this case there is no easy closed form for the finite sample distribution of the derivative of the $ t $-statistic, however, it is still possible to derive an asymptotic limit for its distribution. 

\begin{thm}\label{thm:derivclt}
	Suppose that $ Y_1$ has differentiable mean and variance and that\\ $ \cov\left( \nabla^T \frac{Y_1(s)}{\sigma(s)} \right) < \infty$ for all $ s \in S $. Then for all $ J \in \mathbb{N} $ and $ s_1, \dots, s_J \in S, $
	\begin{equation*}
	\sqrt{N}\left( \nabla d_N(s_1) - \frac{\sigma(s_J)\nabla d(s_1)}{\hat{\sigma}_N(s_1)}, \dots, \nabla d_N(s_J) - \frac{\sigma(s_J)\nabla d(s_J)}{\hat{\sigma}_N(s_J)} \right)^T
	\end{equation*}
	satisfies a CLT as $ N \rightarrow \infty. $
\end{thm}
\noindent We can recover the Gaussian case \eqref{eq:GaussCDconv} as follows.
\begin{rmk}\label{rem:rem}
	Suppose that the assumptions of Theorem \ref{thm:derivclt} hold and that $ (Y_n)_{n \in \mathbb{N}} $ are Gaussian. Then, letting $ \Lambda' = \cov(Y_1/\sigma) $, Theorem \ref{thm:derivclt} implies that
	\begin{equation*}
	\sqrt{N}\left( \displaystyle \nabla d_N - \frac{\sigma\nabla d}{\hat{\sigma}_N}\right) \convd \mathcal{N}\left( 0, \left( 1 + d^2 \right) \Lambda'\right)
	\end{equation*}
	pointwise as $ N \rightarrow \infty $, which matches \eqref{eq:GaussCDconv}. See Section \ref{A:remark} for a proof of this.
\end{rmk}


\noindent Putting the pieces together, as with the mean, gives us the following theorem.
\begin{thm}\label{thm:CDlocclt}
	Suppose that $ Y_1 $ satisfies Assumption \ref{ass:derivexch} with Cohen's $ d $: $ d = \frac{\mu}{\sigma} $ satisfying Assumption \ref{ass:peakconv}. For each $ j = 1, \dots, J$ corresponding to a maximum of $ d $, let $ \hat{\theta}_{j,N} = \argmax_{t \in B_j} d_N(t)  $,  (and for the minima let $ \hat{\theta}_{j,N} = \argmin_{t \in B_j} d_N(t) $) and let 
	$ \boldsymbol{\hat{\theta}_{N}} := (\hat{\theta}_{1,N}^T, \dots, \hat{\theta}_{J,N}^T)^T $ and $ \bf{\theta} := (\theta_{1}^T, \dots, \theta_{J}^T)^T. $ Then
	\begin{equation*}
	\sqrt{N}(\bf{\hat{\theta}_N} - \bf{\theta}) \convd \mathcal{N}(0, \bf{A}\mathbf{\Lambda}'\bf{A}^T)
	\end{equation*} 
	as $ N \conv \infty $. Here $ \bf{A} $ $ \in \mathbb{R}^{DJ\times DJ}$ is a block diagonal matrix with $ J $ diagonal $ D\times D $ blocks such that, for $ j \in 1, \dots, J$, the $ j $th block is $ \left( \nabla^2 d(\theta_j)\right)^{-1}$ and $ \mathbf{\Lambda}' \in \mathbb{R}^{DJ\times DJ} $ is the limiting covariance from Theorem \ref{thm:derivclt}.
\end{thm}
\subsection{Asymptotic Confidence Regions}\label{SS:CR}
Given these results we can obtain confidence regions for peak location which have the correct asymptotic coverage. For the mean, for each $ 1 \leq j \leq J $, letting
\begin{equation*}
\Sigma_j =  (\nabla^2 \mu(\theta_j))^{-1}\cov(\nabla^T Y_1(\theta_j))(\nabla^2 \mu(\theta_j))^{-1}
\end{equation*}
and applying Theorem \ref{thm:meanclt} we have
\begin{equation*}
\sqrt{N}\Sigma_j^{-1/2}(\hat\theta_{j,N} - \theta_j) \sim \mathcal{N}(0,I_D) \implies N(\hat\theta_{j,N} - \theta_j)^T\Sigma_j^{-1}(\hat\theta_{j,N} - \theta_j) \sim \chi^2_D.
\end{equation*}
Thus for $ \alpha \in (0,1) $, letting $ \chi^2_{D, 1 - \alpha} $ be the $ 1- \alpha $ quantile of the $ \chi^2_D$ distribution,
\begin{equation*}
\left\lbrace \theta: N(\hat\theta_{j,N} - \theta)^T\Sigma_j^{-1}(\hat\theta_{j,N} - \theta) < \chi^2_{D, 1 - \alpha} \right\rbrace
\end{equation*}
is a $ (1-\alpha)\% $ asymptotic confidence region for $ \theta_j $. In practice $ \Sigma_j $ is unknown however taking $ \hat{\Lambda}(\hat\theta_{j,N})  $ to be the sample covariance of $ \nabla Y_1(\theta_j), \dots, \nabla Y_N(\theta_j) $ and estimating the Hessian of the mean at $ \theta_j $ by $ \nabla^2 \hat{\mu}_N(\hat\theta_j) $ we obtain an asymptotic $ (1-\alpha)\%  $ confidence region for $ \theta_j $ as:
\begin{equation}\label{eq:aymCR}
\left\lbrace \theta: N(\hat\theta_{j,N} - \theta)^T\hat{\Sigma}_j^{-1}(\hat\theta_{j,N} - \theta) < \chi^2_{D, 1 - \alpha} \right\rbrace
\end{equation}
where $ \hat{\Sigma}_j = (\nabla^2 \hat{\mu}(\hat{\theta}_j))^{-1}\hat{\Lambda}(\hat{\theta}_j)(\nabla^2 \hat{\mu}(\hat{\theta}_j))^{-1}$. For Cohen's $ d $, we can derive analogous confidence regions for peak location using Theorem \ref{thm:CDlocclt}. In this case, taking
\begin{equation*}
\hat\Sigma_j = \left( 1+d_N(\hat\theta_{j,N})^2 \right) \left( \nabla^2 d_N(\hat\theta_{j,N}) \right)^{-1} \hat\Lambda'(\hat\theta_{j,N})\left( \nabla^2 d_N(\hat\theta_{j,N}) \right)^{-1}
\end{equation*}
we obtain a $ (1-\alpha)\% $ asymptotic confidence region for each peak location via equation (\ref{eq:aymCR}). When the fields are stationary a better estimate of $ \Sigma_j $ can be obtained in both cases since a better estimate of $ \Lambda $ or $ \Lambda' $ can be obtained by averaging over space rather than using the estimate from a single point. 

So far we have described confidence regions that are asymptotically valid at each peak. It may instead be desirable to obtain ones which are simultaneously valid over all peaks. To do so, we apply a Bonferroni correction, resulting in regions
\begin{equation}\label{eq:aymbonf}
\left\lbrace \theta: N(\hat\theta_{j,N} - \theta)^T\hat{\Sigma}_j^{-1}(\hat\theta_{j,N} - \theta) < \chi^2_{D, 1 - \alpha/J} \right\rbrace
\end{equation}
about each peak. This will control the familywise coverage rate over peaks. Assuming that the correlation function decays and the peaks are reasonably well separated, Theorems \ref{thm:meanclt} and \ref{thm:CDlocclt} show that peak locations are approximately independent so little will be lost by applying a Bonferroni correction.

\subsection{Monte Carlo Confidence Regions}\label{SS:MC}
The above confidence regions perform well asymptotically however typically give undercoverage in the finite sample (see Section \ref{S:Sims}). This is because, amongst other factors, they do not account for the extra variability that occurs because $  \nabla^2 \hat\mu_N(\theta^*_{j,N}) $, defined in equation \eqref{eq:cltsort} has not converged. Under the additional assumptions of stationarity and Gaussianity it is in fact possible to obtain better finite sample coverage by taking account of the joint distribution between $ \nabla^T \hat\mu_N $ and $ \nabla^2 \hat\mu_N $. To do so observe that using an extra term in the Taylor expansion (\ref{eq:Taylormu}), for $ 1 \leq j \leq J $ and $ N \in \mathbb{N} $, it follows that
\begin{align}\label{eq:relation}
\hat{\theta}_{j,N} - \theta_j &= -\left( \nabla^2 \hat\mu_N(\theta^*_{j,N}) \right)^{-1}\nabla^T \hat{\mu}_N(\theta_j) \\
& =  -\left( \nabla^2 \hat{\mu}_N(\theta_j) + \frac{1}{2}(\hat{\theta}_{j,N} - \theta_j)^T\nabla^3 \hat{\mu}_N(\tilde{\theta}_{j,N}) \right)^{-1}\nabla^T\hat{\mu}_N(\theta_j)
\end{align}
where $ \nabla^3 \hat{\mu}_N $ is a 3-dimensional array corresponding to the third derivative and $ \theta^*_{j,N}, \tilde\theta_{j,N} \in B_{\left\lVert  \theta_j - \hat \theta_{j,N}\right\rVert}(\theta_j)$.

As $ N \rightarrow \infty $, the third derivative, $ \nabla^3 \hat{\mu}_N $ converges uniformly to $ \nabla^3 \mu $ and so for large $ N, \frac{1}{2}(\hat{\theta}_{j,N} - \theta_j)^T\nabla^3 \hat{\mu}_N(\tilde{\theta}_{j,N})$ is small relative to $ \nabla^2\hat{\mu}_N(\theta_j) $. This motivates using the distribution of $ -(\nabla^2 \hat{\mu}_N(\theta_j))^{-1}\nabla\hat{\mu}_N(\theta_j) $ to infer on the distribution of $ \hat{\theta}_{j,N} - \theta_j. $ In practice $ \theta_j $ is unknown but we can use the fact that the numerator and denominator are independent Gaussian random variables, estimate their covariance matrices (at the peak $ \hat{\theta}_N $) and simulate from a distribution that is approximately the same as the one above. 

Letting \textbf{vech} be the operation which sends $ D $-dimensional symmetric matrices to $ \mathbb{R}^{D(D+1)/2} $, 
	\begin{equation*}
	\pmatrix{\nabla^T \hat{\mu}_N(\theta_j);  \textbf{vech}(\nabla^2 \hat\mu_N(\theta_j)) } \sim \mathcal{N}\left( \pmatrix{0; \textbf{vech}(\nabla^2 \mu_N(\theta_j)) }, \frac{1}{N}\pmatrix{\Lambda. 0; 0. \Omega}  \right)
	\end{equation*}
	where for each $ s \in S $, $ \Omega = \cov(\textbf{vech}(\nabla^2 Y_1(s))) $. As the fields are assumed to be stationary we can take advantage of the fact that, $ \Lambda $ and $ \Omega $, are constant over $ S $ to obtain very good estimates $ \hat{\Lambda} $ and $ \hat{\Omega} $ of these quantities based on data from the whole image (rather than from just around the peak). The peak location $ \theta_j $ is unknown so we estimate $ \nabla^2 \mu_N(\theta_j) $ using $ \nabla^2 \hat\mu_N(\hat{\theta}_{j,N})$. 
%
	
	To generate samples from this distribution we choose a (reasonably large) number of simulations $ K \in \mathbb{N} $ and generate independent i.i.d sequences of random vectors $ (A_{k,N})_{1 \leq k \leq K} \sim \mathcal{N}(0, \hat{\Lambda}/N)  $ and $ (B_{k,N})_{1 \leq k \leq K} \sim \mathcal{N}(\textbf{vech}(\nabla^2 \mu_N(\hat{\theta}_{j,N})), \hat{\Omega}/N) $. Then our Monte Carlo draws from the estimated peak location distribution are 
	\begin{equation}\label{eq:deltak}
		\delta_{k,N} = (\textbf{vech}^{-1}(B_{k,N}))^{-1}A_{k,N}
	\end{equation}
	for $ 1 \leq k \leq K $ and where $ \textbf{vech}^{-1} $ is the inverse of the \textbf{vech} operation\footnote{Note that in small sample sizes truncation is occasionally necessary for stability i.e. to ensure that $ \det(\textbf{vech}^{-1}(B_{k,N})) $ is kept away from zero, see Section \ref{A:ratioexplan} for details on how this is done in practice.}. 
	
	We use this to provide Monte Carlo confidence regions which have the same shape, up to parameter estimation, as those used for the asymptotic distribution but adjusting their size to ensure $ (1-\alpha)\% $ coverage based on the Monte Carlo draws. Let $\hat{\Sigma}'_j =  (\nabla^2 \hat{\mu}_N(\hat{\theta}_j))^{-1}\hat{\Lambda}(\nabla^2 \hat{\mu}_N(\hat{\theta}_j))^{-1} $ and for $ 0 < \alpha < 1 $, choose $ \lambda_{\alpha} $ such that 
	\begin{equation*}
	\frac{1}{K} \sum_{k = 1}^K 1\left[ N(\hat\delta_{k,N}^T(\hat{\Sigma}'_j)^{-1}\delta_{k,N}) > \lambda_{\alpha}\right] = \frac{\lfloor \alpha K \rfloor}{K}.
	\end{equation*}
	Note that $ \lambda_{\alpha} $ exists as there are no ties with probability 1. Given this we define the Monte Carlo confidence region to be 
	\begin{equation*}
	\left\lbrace \theta: N(\hat\theta_{j,N} - \theta)^T(\hat{\Sigma}'_j)^{-1}(\hat\theta_{j,N} - \theta) < \lambda_{\alpha} \right\rbrace.
	\end{equation*}
	As $ N \rightarrow \infty, $ $ B_{k,N} \convas \textbf{vech}(\nabla^2 \mu(\theta_j)) $ and $ \sqrt{N}A_{k,N} \convd \mathcal{N}(0, \Lambda) $ so applying Theorem \ref{thm:meanclt} it follows that the probability that the $ j $th Monte Carlo confidence region contains $ \theta_j $ converges to $ 1-\alpha $. In fact, since they have been constructed to have the same shape, the Monte Carlo confidence regions will essentially be the same as the asymptotic confidence regions given large enough $ N. $ The advantage of the Monte Carlo confidence regions is that they take into account the variability in the second derivative and thus have a finite sample performance that is substantially improved (see Section \ref{S:Sims}). As with the asymptotic regions, simultaneous coverage can be obtained by using the $ \alpha/J $ quantiles for each region.
	
	Unlike for the mean, Monte Carlo confidence regions for the locations of peaks of Cohen's $ d $ are not easy to obtain, see the discussion for further details.

\section{Simulations and Data Application}\label{S:Sims}
We conduct simulations to evaluate the coverage of the confidence regions in practice. We demonstrate their validity in 1D and 2D as the sample size increases and investigate their performance relative to the shape of the signal and the smoothness of the noise. We verify that our methods have the correct asymptotic coverage in these settings and illustrate how they can be applied in practice to provide confidence regions for peak location. We present results for inferring on the location of peaks of the mean below - results for the 1D simulations can be found in Section \ref{A:fs}. To illustrate the theory in practice, we apply it to provide confidence intervals for the locations of peaks in a 1D MEG power spectrum and for 2D slices of fMRI data obtained from the UK Biobank. For each application of the Monte Carlo method we used $ K = 10^5 $ draws from the distribution in \eqref{eq:deltak}.

\subsection{Coverage}\label{SS:coverage}
Before we present out results we must formally define what we mean by coverage. In each of our simulations we will generate noise about a fixed mean function with a given number of peaks (varying the smoothness of the noise, the number of samples $ N $ and the shape of the peak). For each setting, given $ \alpha > 0 $, we run $ n_{\text{sim}} \in \mathbb{N} $ simulations. For each simulation $ i \in \left\lbrace 1, \dots, n_{\text{sim}} \right\rbrace$ we calculate a $( 1-\alpha)\% $ confidence region $ R^{\alpha}_{i,j}$ for the true location $ \theta_j $ of a given peak of the mean/Cohen's $ d $, as discussed in Section \ref{SS:CR} (for $ 1 \leq j \leq J $).  Given this we define the \textbf{true coverage} about the $ j $th peak to be $ \mathbb{P}\left( \theta_j \in R^{\alpha}_{1,j} \right) $. We can approximate this by the \textbf{empirical coverage} $ \frac{1}{n_{\text{sim}}}\sum_{i = 1}^{n_{\text{sim}}} 1\left[ \theta_j \in R^{\alpha}_{i,j} \right], $ where $ 1\left[ \cdot \right] $ denotes the indicator function. This converges to the true coverage by the SLLN as $ n_{\text{sim}} \rightarrow \infty $. The $ R^{\alpha}_{i,j} $ are asymptotic $ (1-\alpha)\% $ confidence regions so the true coverage converges to $ 1-\alpha $ as $ N \rightarrow \infty$ and the empirical coverage should thus converges to $ 1-\alpha $ as $ n_{\text{sim}}, N \rightarrow \infty$. 

In the presence of multiple peaks we will summarize the performance of the peak location confidence regions by considering the \textbf{average empirical coverage}, which is defined as $ \frac{1}{Jn_{\text{sim}}}\sum_{j = 1}^J\sum_{i = 1}^{n_{\text{sim}}} 1\left[ \theta_j \in R^{\alpha}_{i,j} \right]. $ We will also consider the \textbf{true joint coverage} which is defined as $ \mathbb{P}\left( \theta_j \in R_{1, j}^{\alpha/J}  \text{ for } 1 \leq j \leq J\right) $ and will approximate this using the \textbf{empirical joint coverage}: $ \frac{1}{n_{\text{sim}}}\sum_{i = 1}^{n_{\text{sim}}} 1\left[ \theta_j \in R^{\alpha/J}_{i,j} \text{ for } 1 \leq j \leq J\right]. $ Based on the theory derived in the previous section these should also converge to $ 1 - \alpha $ in the limit as $ n_{\text{sim}}, N $ tend to infinity.
\subsection{Simulations: peaks of the mean}\label{SS:meansim}
Our first set of simulations consists of 2D stationary Gaussian noise added to peaked signals. The noise was generated by smoothing Gaussian white noise with a Gaussian kernel with given FWHM (full width at half maximum) to generate 2D convolution fields (which were truncated to account for the edge effect - see e.g. \cite{Davenport2022Ravi}). We add these fields to the mean function to obtain simulations with $ N \in \left\lbrace 20,40,\dots 200 \right\rbrace $ and a range of applied smoothing levels. We repeated this 5000 times in each setting in order to calculate the empirical coverage. The results are shown in Figures \ref{fig:2Disotropic} and \ref{fig:2Disotropic19}. The dotted lines, in these and all other corresponding figures, give $ 95\% $ confidence bands and are obtained using the normal approximation to the binomial distribution. The figures show that as the number of subjects increase the asymptotic and Monte Carlo approaches have a coverage that converges to 0.95 as the number of subjects increases. However, the rate of convergence of the Monte Carlo method is substantially faster, meaning that reasonable coverage is obtained even for low sample sizes.

The lower the FWHM of the noise relative to the shape of the peak, the larger the number of subjects that is needed to obtain a good level of coverage. In many settings of interest high smoothness relative to the shape of the peak is a reasonable assumption, allowing us to obtain good coverage given available sample sizes. Here stationarity allows us to obtain better estimates of quantities involved (for both the asymptotic and Monte Carlo methods), such as $ \Lambda $, as we can average over the whole image. However, asymptotic coverage is achieved regardless of stationarity for both the asymptotic and Monte Carlo methods. The results from simulations about different 1D peaks (including results where the underlying noise distribution is non-Gaussian) are available in Section \ref{A:fsmean}.

\begin{figure}[h!]
	\begin{subfigure}[t]{0.26\textwidth}
		\centering
		\includegraphics[width=\textwidth]{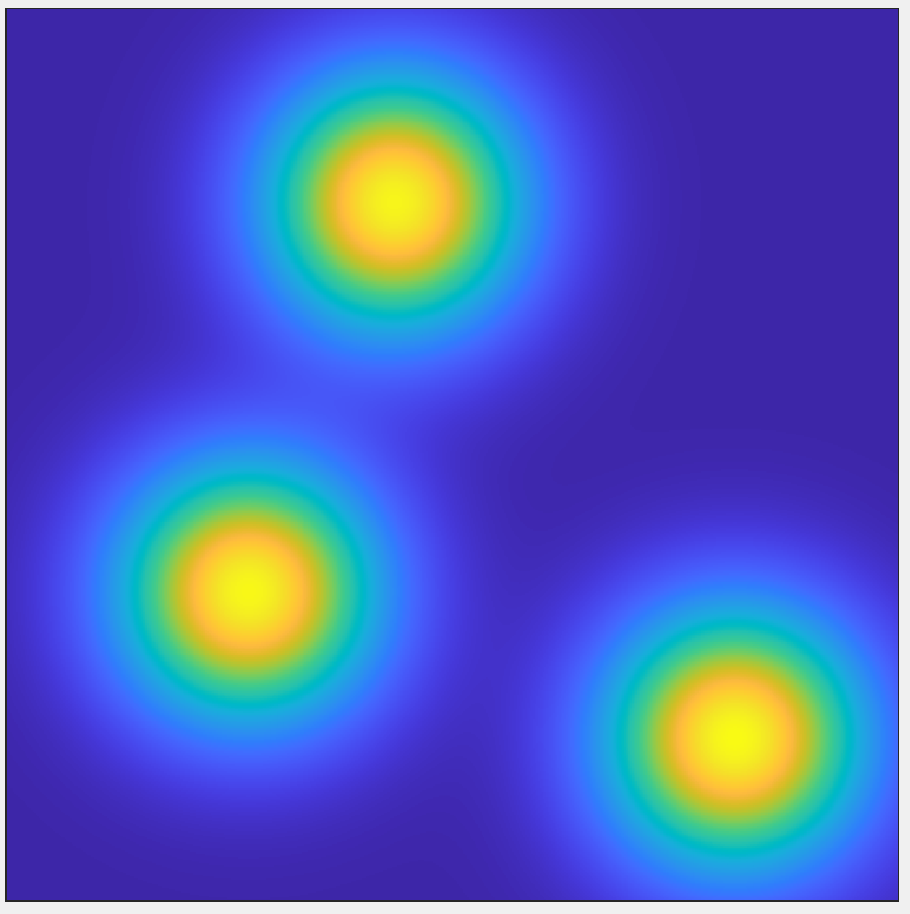}
	\end{subfigure}
	\hfill
	\begin{subfigure}[t]{0.32\textwidth}  
		\centering 
		\includegraphics[width=\textwidth]{smo10_quant_95_dobonf_0_peakno_all_asym.png}
	\end{subfigure}
	\hfill
	\begin{subfigure}[t]{0.32\textwidth}  
		\centering 
		\includegraphics[width=\textwidth]{smo10_quant_95_dobonf_0_peakno_all_ratio.png}
	\end{subfigure}
	\vskip\baselineskip
\begin{subfigure}[t]{0.26\textwidth}
	\centering
	\includegraphics[width=\textwidth]{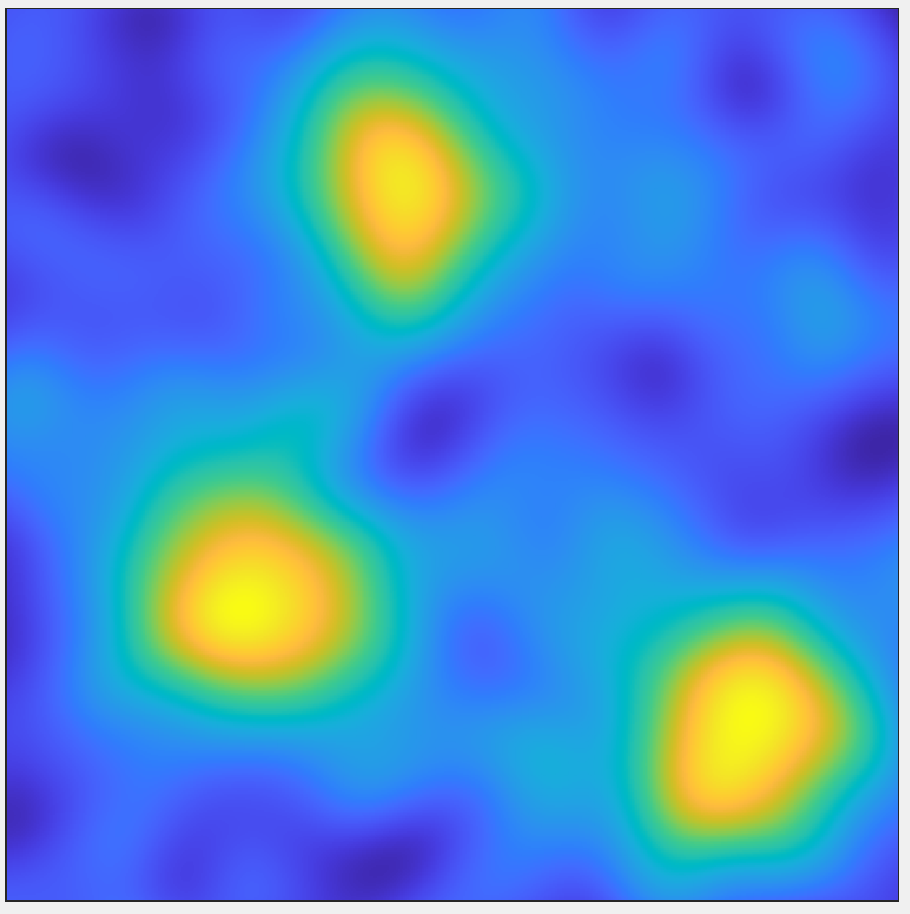}
\end{subfigure}
\hfill
\begin{subfigure}[t]{0.32\textwidth}  
	\centering 
	\includegraphics[width=\textwidth]{smo10_quant_95_dobonf_1_asym.png}
\end{subfigure}
\hfill
\begin{subfigure}[t]{0.32\textwidth}  
	\centering 
	\includegraphics[width=\textwidth]{smo10_quant_95_dobonf_1_ratio.png}
\end{subfigure}
	\caption{Coverage in 2D for the narrower peaks of the mean on a 50 by 50 image. Top Left: a surface plot giving the mean intensity of the signal. Bottom Left: a realisation of a single (FWHM = 5) random field. Centre: the coverage obtained from using the asymptotic distribution. Right: the coverage obtained from using the Monte Carlo distribution. The coverage for the asymptotic and Monte Carlo methods converges to 0.95 as the number of subjects increases. However the rate of convergence of the Monte Carlo approach is substantially faster.}\label{fig:2Disotropic}
\end{figure}

\begin{figure}[h!]
	\begin{subfigure}[t]{0.26\textwidth}
		\centering
		\includegraphics[width=\textwidth]{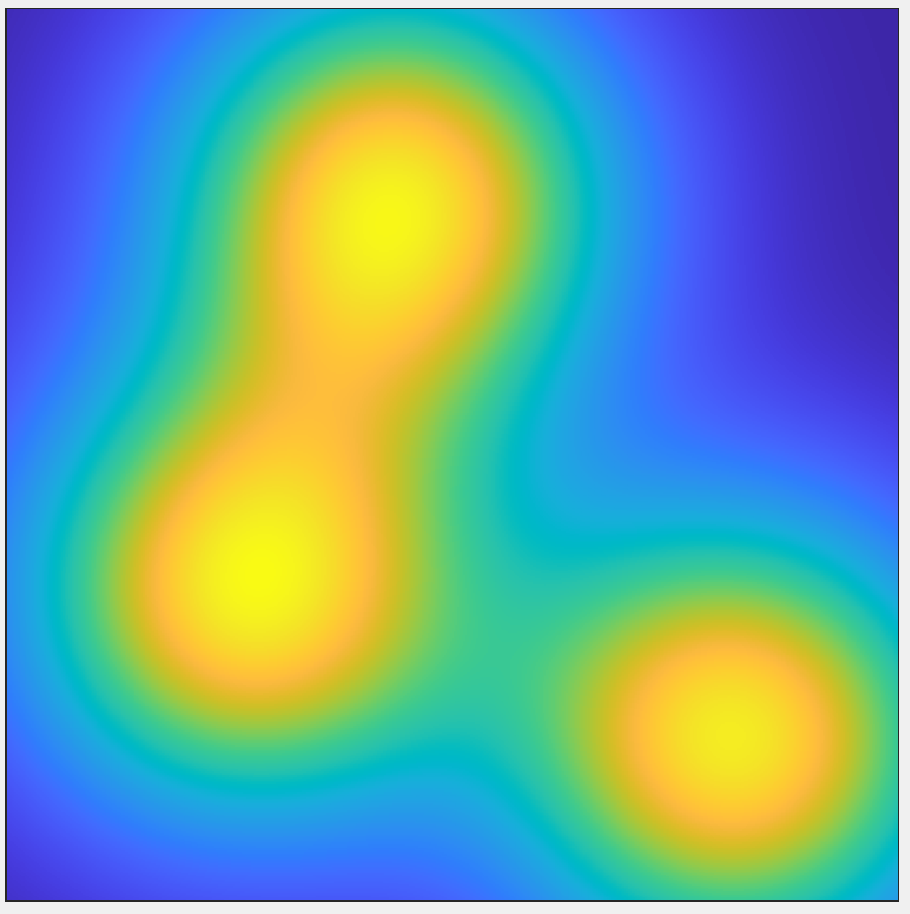}
	\end{subfigure}
	\hfill
	\begin{subfigure}[t]{0.32\textwidth}  
		\centering 
		\includegraphics[width=\textwidth]{smo19_quant_95_dobonf_0_peakno_all_asym.png}
	\end{subfigure}
	\hfill
	\begin{subfigure}[t]{0.32\textwidth}  
		\centering 
		\includegraphics[width=\textwidth]{smo19_quant_95_dobonf_0_peakno_all_ratio.png}
	\end{subfigure}
	\vskip\baselineskip
	\begin{subfigure}[t]{0.26\textwidth}
		\centering
		\includegraphics[width=\textwidth]{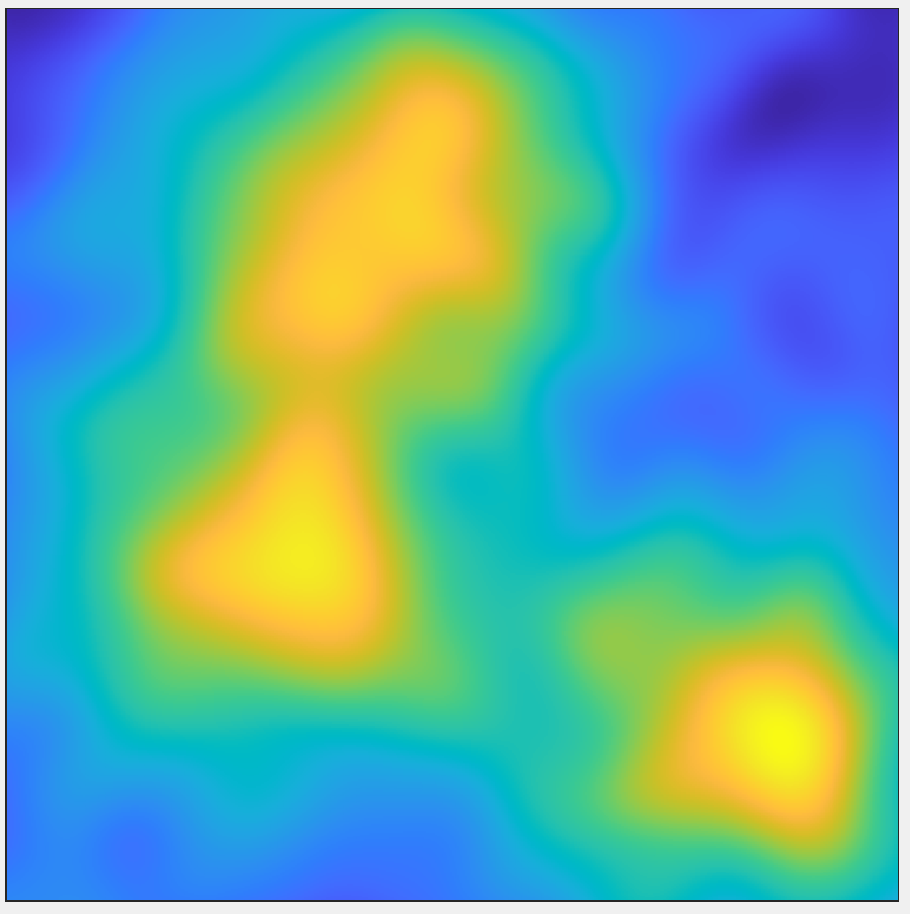}
	\end{subfigure}
	\hfill
	\begin{subfigure}[t]{0.32\textwidth}  
		\centering 
		\includegraphics[width=\textwidth]{smo19_quant_95_dobonf_1_asym.png}
	\end{subfigure}
	\hfill
	\begin{subfigure}[t]{0.32\textwidth}  
		\centering 
		\includegraphics[width=\textwidth]{smo19_quant_95_dobonf_1_ratio.png}
	\end{subfigure}
	\caption{Coverage in 2D for the wider peaks of the mean on a 50 by 50 image. The layout of the figures is the same as in Figure \ref{fig:2Disotropic}. In this more difficult setting the convergence of the asymptotic method is substantially slower than in the case of the narrower peaks. The Monte Carlo method maintains the nominal coverage even in small sample sizes however it is slightly conservative when considering the joint coverage.}\label{fig:2Disotropic19}
\end{figure}

The setting of Figure \ref{fig:2Disotropic19} is more difficult because the peaks are wider and the true maximum is more difficult to localise. In this setting the asymptotic method is very slow to converge. In contrast, the Monte Carlo method performs very well, attaining the nominal coverage even in small sample sizes. However the joint coverage is slightly conservative in this setting.





\subsection{Simulations: peaks of Cohen's $ d $}\label{SS:cdsim}
To illustrate the performance of our confidence regions for local maxima of the $ t $-statistic we use the same simulation settings as above however we infer on peaks of Cohen's $ d $ rather than the mean. For the $ t $-statistic, changing the variance across the image is equivalent to changing the mean so without loss of generality we take the variance to be constant across the image. The results are shown in Figures \ref{fig:2Dtstatiso} and \ref{fig:2Dtstat19} and from these we see that the correct coverage is obtained given sufficiently many subjects. Convergence is slower than for the mean and as before improves the higher the level of applied smoothness. Note that no Monte Carlo approach is available for Cohen's $ d. $ The results for the coverage about 1D peaks  (including results where the underlying noise distribution is non-Gaussian) are available in Section \ref{A:fscd} and show a similar trend.

\begin{figure}[h!]
	\begin{subfigure}[t]{0.26\textwidth}
		\centering
		\includegraphics[width=\textwidth]{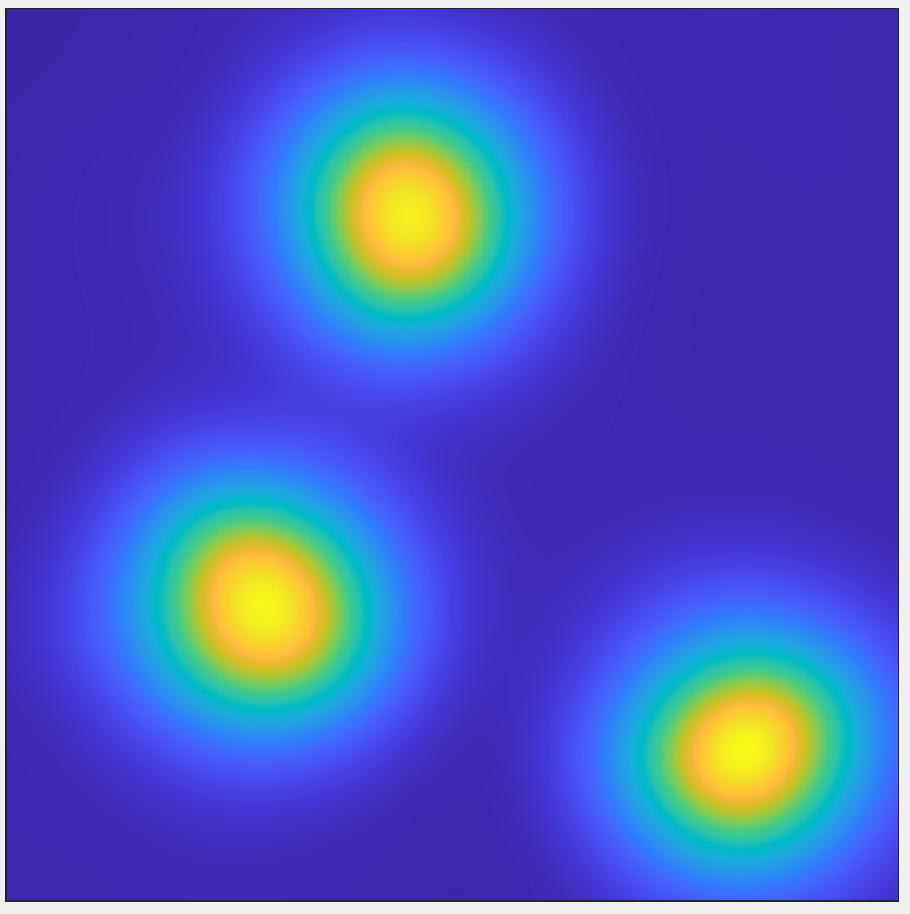}
	\end{subfigure}
	\hfill
	\begin{subfigure}[t]{0.32\textwidth}  
		\centering 
		\includegraphics[width=\textwidth]{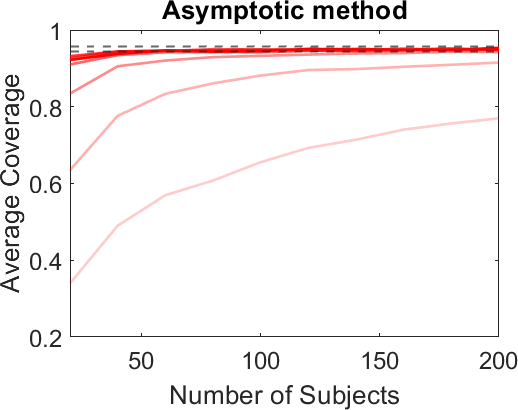}
	\end{subfigure}
	\hfill
	\begin{subfigure}[t]{0.32\textwidth}  
		\centering 
		\includegraphics[width=\textwidth]{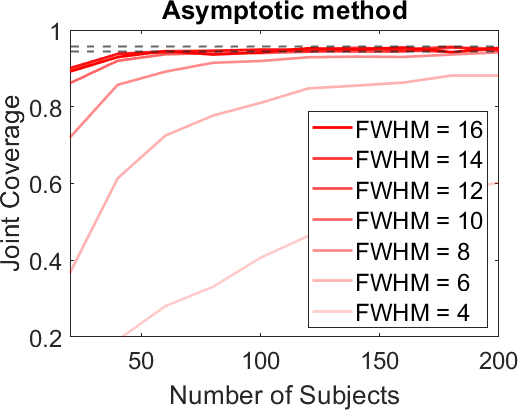}
	\end{subfigure}
	\caption{Coverage in 2D for the narrower peaks of Cohen's $ d $ on a 50 by 50 image. Left: a surface plot giving the observed Cohen's $ d $ for a random sample of 100 subjects whose noise is smoothed with a FWHM of 10. Centre: the coverage obtained from using the asymptotic distribution. Right: the coverage obtained from using the Monte Carlo distribution.}\label{fig:2Dtstatiso}
\end{figure}

\begin{figure}[h!]
	\begin{subfigure}[t]{0.26\textwidth}
		\centering
		\includegraphics[width=\textwidth]{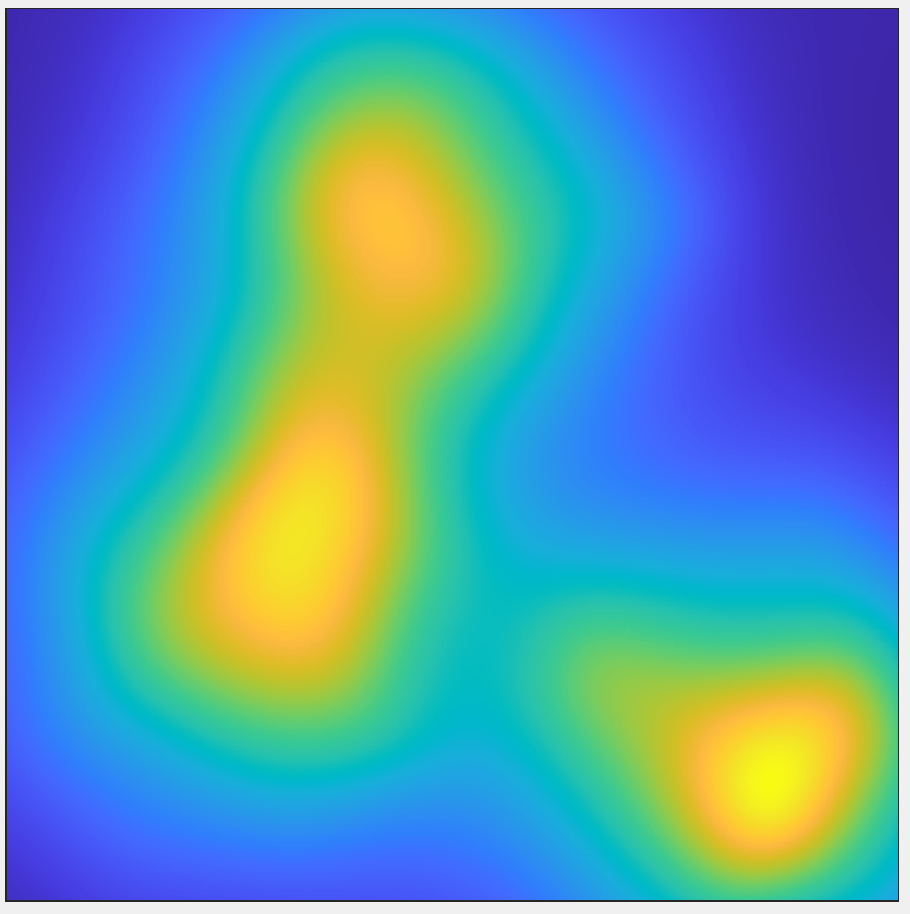}
	\end{subfigure}
	\hfill
	\begin{subfigure}[t]{0.32\textwidth}  
		\centering 
		\includegraphics[width=\textwidth]{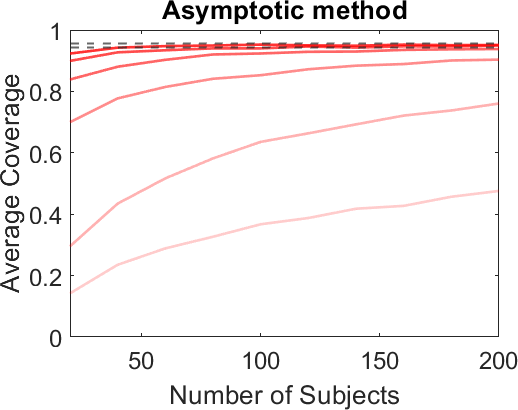}
	\end{subfigure}
	\hfill
	\begin{subfigure}[t]{0.32\textwidth}  
		\centering 
		\includegraphics[width=\textwidth]{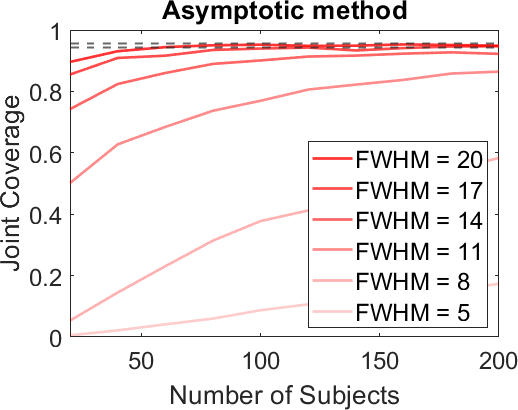}
	\end{subfigure}
	\caption{Coverage in 2D for the wider peaks of Cohen's $ d $ on a 50 by 50 image. The plots are laid out as in Figure \ref{fig:2Dtstatiso}. A higher level of smoothness is required for quick convergence than for the narrow peak setting.}\label{fig:2Dtstat19}
\end{figure}

\subsection{Application: MEG power spectra}\label{SS:MEGmeg}
We have 1D MEG data from 79 subjects (from a single MEG node) and for each subject have around 6 minutes of time series data sampled at a rate of 240Hz (see \cite{Quinn2019} for details on the data and how it was collected). From this data we derive power spectrum random fields using \cite{Welch1967}'s method (as described in Section \ref{SS:MEGspectra}), samples of which are shown in Figure \ref{fig:MEGmean}. Applying our asymptotic confidence regions approach with the Bonferroni correction, we obtain joint $ 95\% $ confidence intervals of (0.876,0.910) Hz and (2.276,2.314) Hz for the location of the highest two peaks of the mean. In this example the uncertainty for both peaks is the same as that provided by the stationary Monte Carlo method. This indicates that the methods have converged and so we expect the confidence interval to provide good coverage in this setting.

We can also infer on peaks of Cohen's $ d $ obtained from the log power spectra with respect to the average across the frequencies from 0 to 60 Hz. Using our approach, we calculate joint 95\% confidence intervals (0.870,0.916) Hz and (2.266,2.323) Hz, for the locations of the top two peaks which lie in the delta frequency band, this is illustrated in Figure \ref{fig:MEGtstat}. Since the noise is smooth relative to the shape of the peaks we expect these confidence intervals to give good coverage for the true peak locations. Notably the standard error for both confidence intervals is similar, this occurs because the peaks have a similar shape and the smoothness of the noise is similar around each peak.

\subsection{Application: fMRI}\label{SS:fmri}
We apply the method to estimate confidence regions for 2D coronal slices of fMRI data of 125 subjects from the UK biobank \citep{Miller2016}. The subjects have undergone a perceptual matching task, consisting of blocks of geometric shapes and emotional faces; the effect of interest is the difference in fMRI signal between the face and shape blocks \citep{Hariri2002}. A first level regression is performed, as described in \cite{Mumford2009}, resulting in a 2D contrast map for each subject, for the contrast of the ``faces" condition minus the "shapes" condition. Smoothing of 8mm (which corresponds to 4 voxels as the voxels are 2mm by 2mm) was performed to generate 2D convolution fields $ (Y_n) $ for each subject. In the resulting dataset the top two peaks of the mean lie within the fusiform face areas which are the regions of the brain that are responsible for identifying faces \citep{Kanwisher2006}. These peaks and the 95\% simultaneous confidence regions (derived using the Monte Carlo distribution combined with Bonferroni adjustment) for their location are displayed in Figure \ref{fig:fMRI}. The use of the Monte Carlo method requires stationarity, an assumption that is common in brain imaging \citep{Worsley1992}. From these plots we see that the regions obtained from the Monte Carlo method are substantially larger than the ones obtained via the asymptotic method. This occurs because, as illustrated in Section \ref{SS:meansim}, the asymptotic method underestimates the variance and thus provides undercoverage. Confidence regions for the local maxima of Cohen's $ d $, obtained using the asymptotic method are displayed in Figure \ref{fig:cdfMRI}.

\begin{figure}[h]
	\begin{center}
		\includegraphics[width=0.32\textwidth]{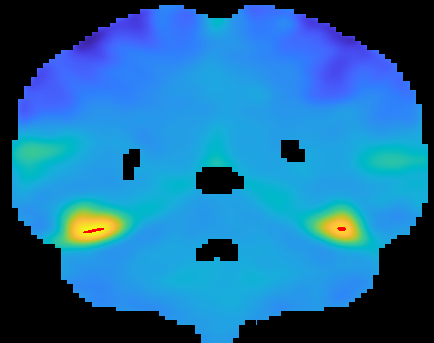}
		\includegraphics[width=0.32\textwidth]{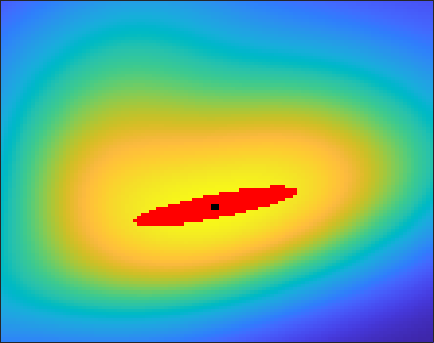}
		\includegraphics[width=0.32\textwidth]{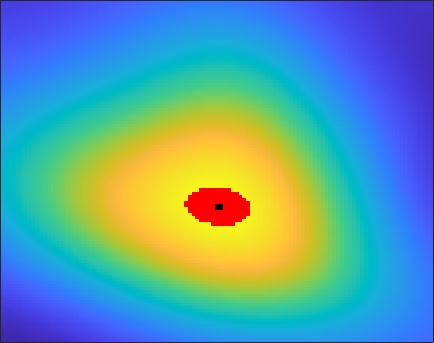}\\
		\vspace{0.25cm}
		\includegraphics[width=0.32\textwidth]{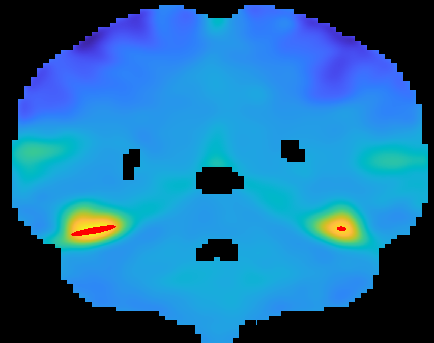}
		\includegraphics[width=0.32\textwidth]{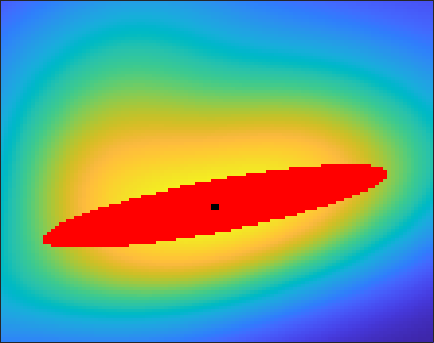}
		\includegraphics[width=0.32\textwidth]{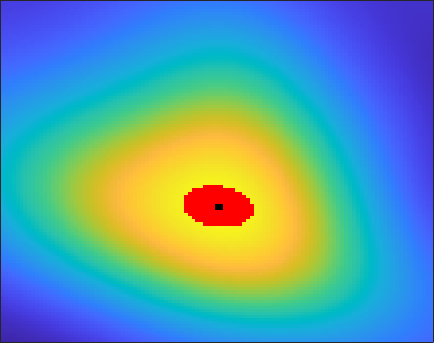}
	\end{center}
\caption{Simultaneous confidence regions for peaks of the mean of 125 subjects from the UK biobank for the faces-shapes contrast. Top: asymptotic method, bottom: Monte Carlo method. The $ 95\% $ confidence regions, corrected to allow for joint coverage over the two peaks, are shown in red: displayed over the mean of the images. Both sets of plots have the following interpretation: Left: whole brain slice. Middle and right: zoomed in sections around each peak with a black dot indicating the location of each maximum. The Monte Carlo regions are bigger than the asymptotic regions for both peaks, however this is more pronounced for the leftmost peak.}\label{fig:fMRI}
\end{figure}

\section{Discussion}
In this paper we have provided asymptotic and Monte Carlo confidence regions for the location of peaks of the mean and Cohen's $ d $ functions and have tested their coverage in a variety of settings as well as illustrating how they can be applied in practice. We have demonstrated that under stationarity, for peaks of the mean, using the Monte Carlo method can significantly improve the rate at which the empirical coverage converges.

The asymptotic results show that the limiting covariance matrix (for the mean) about the $ j $th peak (for $ 1 \leq j\leq J $) is $ (\nabla^2 \mu(\theta_j))^{-1}\Lambda(\theta_j)(\nabla^2 \mu(\theta_j))^{-1} $. This means that, for fixed sample size, the size of the confidence regions increases as the determinant of the Hessian $ \nabla^2\mu(\theta_j) $ goes to zero and as the smoothness of the noise (which is inversely proportional to $ \Lambda(\theta_j) $) decreases. In order provide some intuition regarding this, consider the following illustrative 1D example. Assume that the shape of the peak is that of the pdf of the $ \mathcal{N}(0, \tau^2) $ distribution (for some $\tau > 0$) scaled to have height $ h $ and that the noise is obtained by smoothing Gaussian white noise with a Gaussian kernel with FWHM $ f $ and scaled to have variance $ \sigma^2 $. In this case the limiting variance is proportional to $ \frac{\tau^4 \sigma^2}{h^2f^2} $. This clearly decreases as the FWHM of the noise and the height of the signal increase and increases with the variance of the noise and the width of the peak. 

We also found that wider peaks (relative to the noise) and rougher noise typically require a larger sample size $ N $ before the correct coverage is obtained. This is because the wider the peak and the rougher the noise, the more the location of the maximum will be driven by peaks in the noise rather than peaks in the signal. In general the shape of the peak and the smoothness of the noise have a large effect on both the limiting covariance and the rate at which convergence of the coverage occurs.

Using the asymptotic distribution to create confidence regions does not take account of the extra variance that arises from the inverse of the Hessian in equation \eqref{eq:cltsort}. As such using the asymptotic distribution leads to undercoverage in the finite sample. Under stationarity, for peaks of the mean, we have shown that the coverage obtained can be significantly improved by using the Monte Carlo confidence regions. These are created by simulating from the joint distribution between the first and second derivatives in order to account for the variance of the Hessian and enable a good level of coverage to be obtained even in small sample sizes. Under stationarity the first derivative and the Hessian are independent and as such we should expect the Monte Carlo approach to lead to an increase in the coverage (see Theorem \ref{thm:AoverB} and the discussion following it in Appendix \ref{A:ratioexplan}). The improvement is exemplified at lower smoothness levels where the convergence for the asymptotic method is particularly slow. 

When the noise is non-stationary the parameters of the Monte Carlo distribution become more difficult to estimate and may not be the same at the location of the empirical peak and at the true peak. Moreover the derivative and the Hessian are no longer guaranteed to be independent meaning that the conditions needed to prove results such as Theorem \ref{thm:AoverB}  do not hold. Nevertheless it would be interesting to determine scenarios where the Monte Carlo distribution or a variant could lead to improvements even when the noise is non-stationary, such as under some sort of local stationarity. Further improvements could potentially be achieved by considering higher terms in the Taylor expansion and taking advantage of their joint distribution which would also be Gaussian (at least asymptotically). Such non-stationary improvements could be useful in the context of maximum likelihood estimation in the finite sample. 

The rate of convergence of the coverage for the asymptotic method, when inferring on peaks of Cohen's $ d $, is slower than when inferring on peaks of the mean. This is because the variability in the location of observed peaks of the $ t $-statistic about their true location is substantially higher than for peaks of the mean \citep{Taylor1993}. This means that the estimate of parameters at the true peak is in general less accurate and this has a knock on effect on the coverage rates. In the case that the fields are Gaussian, assuming they are also unit-variance for simplicity, we can quantify this approximately by considering the difference between the variance of the derivatives. To see this note that when the variance equals 1, $ \mu = d $ everywhere, and so $ \nabla^2 \mu = \nabla ^2 d $. As such  changes in the asymptotic variance when comparing convergence of the mean versus convergence of Cohen's $ d $ (i.e. in Theorems \ref{thm:meanclt} and \ref{thm:CDlocclt}) are due to changes in the variance of the derivatives of the noise. At a peak of the mean the variance of $ \sqrt{N}\nabla \hat{\mu}_N $ is $ \Lambda $ whereas, using Lemma \ref{lem:Tstatderiv}, at a peak of Cohen's $ d $ the variance of $ \sqrt{N}\nabla d_N$ is $ \Lambda\left( 1 + \frac{N}{N-1}d_N^2 \right) $. In our simulations, $ d_N \approx 1$ at the peaks, so the variance will approximately be increased by a factor of $ 2 $. In simulations where the noise is given by Gaussian noise convolved with a Gaussian kernel, $  \Lambda $ is inversely proportional to the FWHM squared. As such, in our simulations, we might expect the coverage rates for Cohen's $ d $ (for convolution fields generated by smoothing Gaussian noise with a Gaussian kernel with FWHM $ f $) to be roughly comparable to those of the mean with FWHM $ f/\sqrt{2} $. Comparing the Cohen's $ d $ and mean simulations in Section \ref{S:Sims} this is not quite correct, as there are a number of other factors involved, but this calculation nevertheless provides a justification for the level of the discrepancy.

Developing a corresponding Monte-Carlo distribution for the $ t $-statistic seems feasible in principle. In practice however this is difficult. It requires an understanding of the joint distribution between the first and second derivatives of a $ t $-statistic field. This is complicated and independence between these different derivatives does not hold as both contain terms involving the derivatives of the fields themselves, so it would not be possible to extend Theorem 1 to this setting (see Section 4 of the Supplementary Material for theoretical expansions of the first and second derivatives of Cohen's $ d $). Extending to this setting is thus beyond our scope and is left to future work.

It may also be of interest to develop non-parametric bootstrap style confidence regions. Consistency results for these have been developed in the context of M-estimation (see e.g.\cite{Wellner1996}), and it would be interesting to extend these results to multiple peaks and to $ t $-statistic fields. Future work could also investigate applying these techniques in larger dimensions and in other settings such as for fMRI data. It should also be possible to develop confidence regions for the locations of peaks of other random fields, such as $R^2$-fields, using similar techniques.

Our methods rely on identifiability of the peaks and we have shown that this occurs under reasonable assumptions, given large enough sample sizes. In noisy scenarios there might be more than one observed peak about a true peak of the signal or even none at all. Moreover, if two peaks were found near to each other then it might be difficult to distinguish them. In small sample sizes it may be difficult to know whether identifiability can be assumed to hold. One heuristic that seems reasonable (in 1D) is to assume that identifiability has occurred if the $ 95\% $ confidence interval (about a given peak) lies within the inflection points of the peak in the observed mean/Cohen's $ d $. Algorithms such as STEM \citep{Cheng2017a} can also be used in order to distinguish peaks of the noise from peaks of the signal. 

One interesting application of our results could be to improve coordinate based meta-analysis (see \cite{Eickhoff2009}, \cite{Salimi-Khorshidi2009}). These types of meta-analyses typically make use of a confidence region around the peaks reported across studies that represents a combination of within study and between study variation in the peak location. The within study variation is typically approximate and not theoretically justified. Our work enables confidence regions to be generated that have asymptotic theoretical coverage guarantees. Moreover we have shown that, for $ 0 < \alpha < 1 $, the volume of the $ (1-\alpha)\% $ confidence region shrinks at a rate of $ N^{D/2} $ as the sample size $ N $ increases. Using our results should thus allow practitioners to better account for the change in the size of the within study uncertainty as the sample size changes as well as enabling them to make more precise confidence statements and thus perform more exact meta-analyses.

\section*{Acknowledgments}
\if\blind0
We would like to thank Dr. Fabian Telschow at Humboldt University of Berlin for helpful discussions on Random Field Theory and convolution fields. We would also like to thank Dr. Dan Cheng at the University of Arizona for his help in understanding the proofs of \cite{Cheng2017a}. We are also grateful to Dr. Andrew Quinn at the University of Oxford for his help in understanding the the MEG datasets and for providing us with MEG data.

TEN is supported by the Wellcome Trust, 100309/Z/12/Z and SD was partially funded by the EPSRC. SD, AS and TEN were partially supported by NIH grant R01EB026859. The research was carried out under UK BioBank application \#34077, with bulk image data shared within Oxford (with UK BioBank permission) from application 8107.
\fi

\if\jasa1
\spacingset{1}
\fi

\bibliographystyle{abbrvnat}
\bibliography{../../../../../Citations/TomsPapers,../../../../../Citations/Kernels,../../../../../Citations/fMRI,../../../../../Citations/MachineLearning,../../../../../Citations/Theory,../../../../../Citations/Notes,../../../../../Citations/Optimization,../../../../../Citations/Statistics,../../../../../Citations/RFT,../../../../../Citations/extras}

\begin{thebibliography}{38}
\providecommand{\natexlab}[1]{#1}
\providecommand{\url}[1]{\texttt{#1}}
\expandafter\ifx\csname urlstyle\endcsname\relax
  \providecommand{\doi}[1]{doi: #1}\else
  \providecommand{\doi}{doi: \begingroup \urlstyle{rm}\Url}\fi

\bibitem[Adler(1981)]{Adler1981}
R.~J. Adler.
\newblock \emph{{The Geometry of Random Fields}}.
\newblock 1981.

\bibitem[Adler and Taylor(2007)]{Adler2007}
R.~J. Adler and J.~E. Taylor.
\newblock \emph{Random fields and geometry}.
\newblock Springer Science \& Business Media, 2007.

\bibitem[Adler et~al.(2010)Adler, Taylor, and Worsley]{Adler2010}
R.~J. Adler, J.~E. Taylor, and K.~J. Worsley.
\newblock \emph{Applications of random fields and geometry: Foundations and
  case studies}.
\newblock 2010.

\bibitem[Amemiya(1985)]{Amemiya1985}
T.~Amemiya.
\newblock \emph{Advanced Econometrics}.
\newblock Harvard University Press, 1985.

\bibitem[Bowring et~al.(2019)Bowring, Telschow, Schwartzman, and
  Nichols]{Bowring2019}
A.~Bowring, F.~Telschow, A.~Schwartzman, and T.~E. Nichols.
\newblock Spatial confidence sets for raw effect size images.
\newblock \emph{NeuroImage}, 203:\penalty0 116187, 2019.

\bibitem[Bowring et~al.(2020)Bowring, Telschow, Schwartzman, and
  Nichols]{Bowring2020}
A.~Bowring, F.~Telschow, A.~Schwartzman, and T.~E. Nichols.
\newblock Confidence sets for cohen’s d effect size images.
\newblock \emph{NeuroImage}, 2020.

\bibitem[Braunstein(1992)]{Braunstein1992}
S.~L. Braunstein.
\newblock How large a sample is needed for the maximum likelihood estimator to
  be approximately gaussian?
\newblock \emph{Journal of Physics A: Mathematical and General}, 25\penalty0
  (13):\penalty0 3813, 1992.

\bibitem[Cheng and Schwartzman(2015)]{Cheng2015a}
D.~Cheng and A.~Schwartzman.
\newblock Distribution of the height of local maxima of gaussian random fields.
\newblock \emph{Extremes}, 18\penalty0 (2):\penalty0 213--240, 2015.

\bibitem[Cheng et~al.(2017)Cheng, Schwartzman, et~al.]{Cheng2017a}
D.~Cheng, A.~Schwartzman, et~al.
\newblock Multiple testing of local maxima for detection of peaks in random
  fields.
\newblock \emph{The Annals of Statistics}, 45\penalty0 (2):\penalty0 529--556,
  2017.

\bibitem[Chumbley and Friston(2009)]{Chumbley2009b}
J.~Chumbley and K.~Friston.
\newblock False discovery rate revisited: {FDR} and topological inference using
  {G}aussian random fields.
\newblock \emph{Neuroimage}, 44\penalty0 (1):\penalty0 62--70, 2009.

\bibitem[Chumbley et~al.(2010)Chumbley, Worsley, Flandin, and
  Friston]{Chumbley2010}
J.~Chumbley, K.~Worsley, G.~Flandin, and K.~Friston.
\newblock Topological {FDR} for neuroimaging.
\newblock \emph{Neuroimage}, 49\penalty0 (4):\penalty0 3057--3064, 2010.

\bibitem[Davenport(2021)]{DavenportThesis}
S.~Davenport.
\newblock \emph{Statistical {I}nference in fMRI using {R}andom {F}ield {T}heory
  and {R}esampling {M}ethods}.
\newblock Phd Thesis from Oxford University, 2021.

\bibitem[Davenport and Nichols(2020)]{Davenport2020}
S.~Davenport and T.~E. Nichols.
\newblock Selective peak inference: Unbiased estimation of raw and standardized
  effect size at local maxima.
\newblock \emph{NeuroImage}, 2020.

\bibitem[Davenport and Nichols(2022)]{Davenport2022Ravi}
S.~Davenport and T.~E. Nichols.
\newblock The expected behaviour of random fields in high dimensions:
  contradictions in the results of [1].
\newblock \emph{Magnetic Resonance Imaging}, 2022.

\bibitem[Eickhoff et~al.(2012)Eickhoff, Bzdok, Laird, Kurth, and
  Fox]{Eickhoff2012}
S.~B. Eickhoff, D.~Bzdok, A.~R. Laird, F.~Kurth, and P.~T. Fox.
\newblock {Activation likelihood estimation revisited}.
\newblock \emph{NeuroImage}, 59\penalty0 (3):\penalty0 2349--2361, 2012.

\bibitem[Eickhoff et~al.(2009)]{Eickhoff2009}
S.~B. Eickhoff et~al.
\newblock {Coordinate-based ALE meta-analysis of neuroimaging data: A
  random-effects approach based on empirical estimates of spatial uncertainty}.
\newblock 30\penalty0 (9):\penalty0 2907--2926, 2009.

\bibitem[Hariri et~al.(2002)Hariri, Tessitore, Mattay, Fera, and
  Weinberger]{Hariri2002}
A.~R. Hariri, A.~Tessitore, V.~S. Mattay, F.~Fera, and D.~R. Weinberger.
\newblock The amygdala response to emotional stimuli: a comparison of faces and
  scenes.
\newblock \emph{Neuroimage}, 17\penalty0 (1):\penalty0 317--323, 2002.

\bibitem[Hayashi(2000)]{Hayashi2000}
F.~Hayashi.
\newblock \emph{Econometrics}.
\newblock Princeton University Press, 2000.

\bibitem[Kanwisher and Yovel(2006)]{Kanwisher2006}
N.~Kanwisher and G.~Yovel.
\newblock The fusiform face area: a cortical region specialized for the
  perception of faces.
\newblock \emph{Philosophical Transactions of the Royal Society B: Biological
  Sciences}, 361\penalty0 (1476):\penalty0 2109--2128, 2006.

\bibitem[Landau and Shepp(1970)]{Landau1970}
H.~Landau and L.~A. Shepp.
\newblock On the supremum of a gaussian process.
\newblock \emph{Sankhy{\=a}: The Indian Journal of Statistics, Series A}, pages
  369--378, 1970.

\bibitem[Ledoux and Talagrand(2013)]{Ledoux2013}
M.~Ledoux and M.~Talagrand.
\newblock \emph{Probability in Banach Spaces: isoperimetry and processes}.
\newblock Springer Science \& Business Media, 2013.

\bibitem[Liebl and Reimherr(2019)]{Liebl2019}
D.~Liebl and M.~Reimherr.
\newblock Fast and fair simultaneous confidence bands for functional
  parameters.
\newblock \emph{arXiv preprint arXiv:1910.00131}, 2019.

\bibitem[Miller et~al.(2016)]{Miller2016}
K.~L. Miller et~al.
\newblock Multimodal population brain imaging in the uk biobank prospective
  epidemiological study.
\newblock \emph{Nature neuroscience}, 19\penalty0 (11):\penalty0 1523, 2016.

\bibitem[Mumford and Nichols(2009)]{Mumford2009}
J.~A. Mumford and T.~Nichols.
\newblock Simple group fmri modeling and inference.
\newblock \emph{Neuroimage}, 47\penalty0 (4):\penalty0 1469--1475, 2009.

\bibitem[Quinn et~al.(2019)Quinn, van Ede, Brookes, Heideman, Nowak, Seedat,
  Vidaurre, Zich, Nobre, and Woolrich]{Quinn2019}
A.~J. Quinn, F.~van Ede, M.~J. Brookes, S.~G. Heideman, M.~Nowak, Z.~A. Seedat,
  D.~Vidaurre, C.~Zich, A.~C. Nobre, and M.~W. Woolrich.
\newblock {U}npacking transient event dynamics in electrophysiological power
  {S}pectra.
\newblock \emph{Brain topography}, pages 1--15, 2019.

\bibitem[Radua et~al.(2012)]{Radua2012}
J.~Radua et~al.
\newblock A new meta-analytic method for neuroimaging studies that combines
  reported peak coordinates and statistical parametric maps.
\newblock \emph{European psychiatry}, 27\penalty0 (8):\penalty0 605--611, 2012.

\bibitem[Salimi-Khorshidi~et al.(2009)]{Salimi-Khorshidi2009}
G.~Salimi-Khorshidi~et al.
\newblock {Meta-analysis of neuroimaging data: A comparison of image-based and
  coordinate-based pooling of studies}.
\newblock \emph{NeuroImage}, 45\penalty0 (3):\penalty0 810--823, 2009.

\bibitem[Schwartzman et~al.(2011)Schwartzman, Gavrilov, and
  Adler]{Schwartzman2011}
A.~Schwartzman, Y.~Gavrilov, and R.~J. Adler.
\newblock Multiple testing of local maxima for detection of peaks in 1{D}.
\newblock \emph{Annals of statistics}, 39\penalty0 (6):\penalty0 3290, 2011.

\bibitem[Shi(2011)]{Xiaoxia}
X.~Shi.
\newblock Lecture notes: Asymptotic normality of extremum estimators.
\newblock \emph{Statistics Technical Report}, 2011.

\bibitem[Sommerfeld et~al.(2018)Sommerfeld, Sain, and
  Schwartzman]{Sommerfield2018}
M.~Sommerfeld, S.~Sain, and A.~Schwartzman.
\newblock {Confidence regions for spatial excursion sets from repeated random
  field observations, with an application to climate}.
\newblock \emph{Journal of the American Statistical Association},
  1459:\penalty0 0--0, 2018.

\bibitem[Taylor et~al.(1993)Taylor, Minoshima, and Koeppe]{Taylor1993}
S.~F. Taylor, S.~Minoshima, and R.~A. Koeppe.
\newblock Instability of localization of cerebral blood flow activation foci
  with parametric maps.
\newblock \emph{Journal of Cerebral Blood Flow \& Metabolism}, 13\penalty0
  (6):\penalty0 1040--1041, 1993.

\bibitem[Telschow and Schwartzman(2020)]{TelschowSCB}
F.~Telschow and A.~Schwartzman.
\newblock {On Simultaneous Confidence Statements and Inference for Functional
  Data Using the Gaussian Kinematic Formula}.
\newblock \penalty0 (April):\penalty0 1--27, 2020.

\bibitem[Telschow et~al.(2020{\natexlab{a}})Telschow, Davenport, and
  Schwartzman]{Telschow2020Delta}
F.~Telschow, S.~Davenport, and A.~Schwartzman.
\newblock Functional delta residuals.
\newblock \emph{Journal of Multivariate Analysis}, 2020{\natexlab{a}}.

\bibitem[Telschow et~al.(2020{\natexlab{b}})Telschow, Davenport, and
  Schwartzman]{TelschowFWER}
F.~Telschow, S.~Davenport, and A.~Schwartzman.
\newblock From discrete to continuous land: Using the continuous gaussian
  kinematic formula on a discrete lattice.
\newblock \emph{Preprint}, 2020{\natexlab{b}}.

\bibitem[Welch(1967)]{Welch1967}
P.~Welch.
\newblock The use of fast fourier transform for the estimation of power
  spectra: a method based on time averaging over short, modified periodograms.
\newblock \emph{IEEE Transactions on audio and electroacoustics}, 15\penalty0
  (2):\penalty0 70--73, 1967.

\bibitem[Wellner and Zhan(1996)]{Wellner1996}
J.~A. Wellner and Y.~Zhan.
\newblock Bootstrapping z-estimators.
\newblock \emph{University of Washington Department of Statistics Technical
  Report}, 308:\penalty0 5, 1996.

\bibitem[Worsley(1994)]{Worsley1994}
K.~J. Worsley.
\newblock {Local Maxima and the Expected Euler Characteristic of Excursion Sets
  of $\chi$ 2, F and t Fields}.
\newblock \emph{Advances in Applied Probability}, 26\penalty0 (1):\penalty0
  13--42, 1994.

\bibitem[Worsley et~al.(1992)Worsley, Evans, Marrett, and Neelin]{Worsley1992}
K.~J. Worsley, A.~C. Evans, S.~Marrett, and P.~Neelin.
\newblock {A three-dimensional statistical analysis for CBF activation studies
  in human brain.}
\newblock \emph{Journal of cerebral blood flow and metabolism.}, 12\penalty0
  (6):\penalty0 900--18, 1992.
\newblock ISSN 0271-678X.
\newblock \doi{10.1038/jcbfm.1992.127}.

\end{thebibliography}


\begin{thebibliography}{12}
\providecommand{\natexlab}[1]{#1}
\providecommand{\url}[1]{\texttt{#1}}
\expandafter\ifx\csname urlstyle\endcsname\relax
  \providecommand{\doi}[1]{doi: #1}\else
  \providecommand{\doi}{doi: \begingroup \urlstyle{rm}\Url}\fi

\bibitem[Amemiya(1985)]{Amemiya1985}
Takeshi Amemiya.
\newblock \emph{Advanced Econometrics}.
\newblock Harvard University Press, 1985.

\bibitem[Cheng et~al.(2017)Cheng, Schwartzman, et~al.]{Cheng2017a}
Dan Cheng, Armin Schwartzman, et~al.
\newblock Multiple testing of local maxima for detection of peaks in random
  fields.
\newblock \emph{The Annals of Statistics}, 45\penalty0 (2):\penalty0 529--556,
  2017.

\bibitem[Landau and Shepp(1970)]{Landau1970}
HJ~Landau and Lawrence~A Shepp.
\newblock On the supremum of a gaussian process.
\newblock \emph{Sankhy{\=a}: The Indian Journal of Statistics, Series A}, pages
  369--378, 1970.

\bibitem[Ledoux and Talagrand(2013)]{Ledoux2013}
Michel Ledoux and Michel Talagrand.
\newblock \emph{Probability in Banach Spaces: isoperimetry and processes}.
\newblock Springer Science \& Business Media, 2013.

\bibitem[Quinn et~al.(2019)Quinn, van Ede, Brookes, Heideman, Nowak, Seedat,
  Vidaurre, Zich, Nobre, and Woolrich]{Quinn2019}
Andrew~J Quinn, Freek van Ede, Matthew~J Brookes, Simone~G Heideman, Magdalena
  Nowak, Zelekha~A Seedat, Diego Vidaurre, Catharina Zich, Anna~C Nobre, and
  Mark~W Woolrich.
\newblock {U}npacking transient event dynamics in electrophysiological power
  {S}pectra.
\newblock \emph{Brain topography}, pages 1--15, 2019.

\bibitem[Shi(2011)]{Xiaoxia}
Xiaoxia Shi.
\newblock Lecture notes: Asymptotic normality of extremum estimators.
\newblock \emph{Statistics Technical Report}, 2011.

\bibitem[Solomon~Jr(1991)]{Solomon1991}
Otis~M Solomon~Jr.
\newblock {PSD} computations using welch's method.[power spectral density
  (psd)].
\newblock Technical report, Sandia National Labs., Albuquerque, NM (United
  States), 1991.

\bibitem[Sommerfeld et~al.(2018)Sommerfeld, Sain, and
  Schwartzman]{Sommerfield2018}
Max Sommerfeld, Stephan Sain, and Armin Schwartzman.
\newblock {Confidence regions for spatial excursion sets from repeated random
  field observations, with an application to climate}.
\newblock \emph{Journal of the American Statistical Association},
  1459:\penalty0 0--0, 2018.

\bibitem[{Telschow} et~al.(2019){Telschow}, {Schwartzman}, {Cheng}, and
  {Pranav}]{TelschowHPE}
Fabian {Telschow}, Armin {Schwartzman}, Dan {Cheng}, and Pratyush {Pranav}.
\newblock {Estimation of Expected {E}uler Characteristic Curves of
  Nonstationary Smooth Gaussian Random Fields}.
\newblock \emph{arXiv e-prints}, art. arXiv:1908.02493, August 2019.

\bibitem[Telschow et~al.(2020)Telschow, Davenport, and
  Schwartzman]{Telschow2020Delta}
Fabian Telschow, Samuel Davenport, and Armin Schwartzman.
\newblock Functional delta residuals.
\newblock \emph{Journal of Multivariate Analysis}, 2020.

\bibitem[Welch(1967)]{Welch1967}
Peter Welch.
\newblock The use of fast fourier transform for the estimation of power
  spectra: a method based on time averaging over short, modified periodograms.
\newblock \emph{IEEE Transactions on audio and electroacoustics}, 15\penalty0
  (2):\penalty0 70--73, 1967.

\bibitem[Worsley(1994)]{Worsley1994}
K.~J. Worsley.
\newblock {Local Maxima and the Expected Euler Characteristic of Excursion Sets
  of $\chi$ 2, F and t Fields}.
\newblock \emph{Advances in Applied Probability}, 26\penalty0 (1):\penalty0
  13--42, 1994.

\end{thebibliography}

\appendix
\section{Proofs for Sections 2 and 3}\label{S:proofs}
\subsection{Proofs for Section 2}
\subsubsection{Proof of Lemma \ref{lem:DEcond}}\label{A:DE}
Let $ e_i, i = 1 ,\dots, D $, be the standard basis vectors in $ \mathbb{R}^D $ and let $ L, B(s)$ be the Lipschitz constant and the ball around $ s $  on which the Lipschitz property holds. Then for $ 1 \leq i \leq D $ and $ 1 \leq j \leq D' $,
\begin{equation*}
\mathbb{E}\left[ (\nabla f)_{ij} \right] = \mathbb{E}\left[ \frac{\partial f_j(s)}{\partial s_i}\right]= \mathbb{E}\left[ \lim_{h \rightarrow 0 } \frac{f_j(s+he_i) - f_j(s)}{h} \right].
\end{equation*}
We can thus apply the Dominated Convergence Theorem to obtain the result, using $ L $ as a dominating function, since for $ h $ small enough such that $ s+he_i\in B(s), $
\begin{equation*}
\left| \frac{f_j(s+he_i) - f_j(s)}{h} \right| \leq \frac{\left\lVert f(s+he_i) - f(s) \right\rVert}{\left| h \right|}  \leq \frac{L\left\lVert he_i \right\rVert}{\left| h \right|} = L.
\end{equation*}

\subsubsection{Proof of Proposition \ref{prop:Gaussfin}}\label{A:Gaussfin}
$ S $ is a subset of $ \mathbb{R}^D $ and is therefore separable, moreover for each $ 1 \leq i \leq D, $ $ f_i $ is a continuous Gaussian random field. So it follows that for all $ m \in \mathbb{N} $, $ \mathbb{E}\left[ \sup_{s \in S}\left| f_i(s) \right|^{m} \right] < \infty $, see \cite{Landau1970}. Expanding the norm we see that:
\begin{align*}
\sup_{s \in S}  \left\lVert \nabla f(s) \right\rVert^2 = \sup_{s \in S} \sum_{i = 1}^D f_i(s)^2
= \sum_{i = 1}^D  \sup_{s \in S} f_i(s)^2
\end{align*}
so in particular its expectation is finite. Moreover, given $ k \in \mathbb{N}, $
\begin{equation*}
\sup_{s \in S}\left\lVert \nabla f^k(t) \right\rVert = \sup_{s \in S}\left\lVert kf(t)^{k-1}\nabla f(t) \right\rVert \leq k \sup_{s \in S} \left| f(t) \right|^{k-1}\sup_{s \in S}\left\lVert \nabla f(t) \right\rVert.
\end{equation*}
As such
\begin{equation*}
\mathbb{E}\left[ \sup_{s \in S}\left\lVert \nabla f^k(t) \right\rVert  \right] \leq k \mathbb{E} \left[ \sup_{s \in S} \left| f(t) \right|^{2(k-1)} \right]^{1/2}\mathbb{E}\left[\sup_{s \in S} \left\lVert \nabla f(t) \right\rVert^{2} \right]^{1/2}
\end{equation*}
which is finite. The result thus follows by applying Lemmas \ref{lem:supbound} and \ref{lem:DEcond}.
\subsubsection{Proof of Proposition \ref{prop:convfield}}\label{AA:convfieldproof}
Let $ c $ be the Lipshitz constant of $ K $. Then, for $ s, t \in S, $
\begin{align*}
\left| Y(s) - Y(t) \right| & \leq \sum_{l \in \mathcal{V}} \left| K(s-l) - K(t-l) \right|\left| X(l) \right|\\
& \leq c\sum_{l \in \mathcal{V}} \left| X(l) \right|\left\lVert s-t \right\rVert = \left( c\sum_{l \in \mathcal{V}} \left| X(l) \right| \right) \left\lVert s-t \right\rVert.
\end{align*}
The result follows by taking  $  c\sum_{l \in \mathcal{V}} \left| X(l) \right| $ as the Lipschitz constant and applying Lemma \ref{lem:DEcond}.

\subsubsection{Proof that fSLLNs hold for $ f^2 $}\label{AA:fSLLN}
Assumption $ \ref{ass:derivexch} $(c) implies that conditions for fSLLNs to hold are satisfied by $ f^2 $ and its derivatives. This will follows fairly quickly by \cite{Ledoux2013}'s Corollary 7.10. For $ f^2 $ it follows immediately from Assumption $ \ref{ass:derivexch} c(i) $. Differentiating $ f^2 $, we see that for $ 1 \leq j \leq D, $
\begin{equation*}
\frac{\partial f^2}{\partial s_j} = 2ff_j \implies \sup_{s \in S} \frac{\partial f(s)^2}{\partial s_j} \leq 2 \sup_{s \in S}f(s)\sup_{s \in S}f_j(s)
\end{equation*}
and so 
\begin{equation*}
\mathbb{E}\left[ \sup_{s \in S} \frac{\partial f(s)^2}{\partial s_j}  \right] \leq 2\mathbb{E}\left[ \sup_{s \in S}f(s)^2 \right]^{1/2}\mathbb{E}\left[ \sup_{s \in S}f_j(s)^2 \right]^{1/2}
\end{equation*}
so the results follows for the first derivative under  $ \ref{ass:derivexch}c (i)$ and $ (ii) $. Differentiating once more and applying Cauchy-Swartz the result follows for the second derivative as well (under Assumption \ref{ass:derivexch}c).

%
\subsubsection{Proof of Proposition \ref{prop:ass1conv}}\label{AA:ass1conv}
\ref{ass:derivexch}a follows easily as the lattice is finite, the variance is finite at each point  and $ K $ is $ C^2 $. For \ref{ass:derivexch}b,c note that, as $ K $ is $ C^2 $, we may assume that the absolute value of $ K(s-l) $ and its derivatives are bounded over all $ s \in S $ and $ l \in \mathcal{V} $ by compactness, as $ S $ is bounded and $ \mathcal{V}  $ is finite and $ K $ is defined everywhere. Let $ K^* $ be an upper bound on $\left| K \right| $, then for $ s \in S $,
\begin{equation*}
\left| Y(s) \right| = \left|\sum_{l \in \mathcal{V}} K(s-l)X(l)\right| \leq K^* \sum_{l \in \mathcal{V}}\left| X(l) \right|
\end{equation*}
and
\begin{align*}
\left| Y(s)^2 \right| &= \left| \sum_{l,l' \in \mathcal{V}} K(s-l)X(l)K(s-l')X(l')  \right|\\
&\leq \sum_{l,l' \in \mathcal{V}} \left| K(s-l)K(s-l') \right| \left| X(l)X(l') \right| \leq (K^*)^2 \sum_{l,l' \in \mathcal{V}}\left| X(l)X(l') \right|.
\end{align*}
so \ref{ass:derivexch}b,c follow for $ Y $. Similar arguments hold for its derivatives.

\subsection{Proofs for Section 3}
\subsubsection{Proof of Proposition \ref{lem:derivbound}}\label{AA:derivbound}
\begin{equation*}
\left\lVert \nabla \hat{\gamma}(t) \right\rVert = 
\left\lVert \nabla \gamma(t) + \nabla \eta(t) \right\rVert 
\geq \left\lVert \nabla \gamma(t) \right\rVert - \left\lVert  \nabla \eta(t) \right\rVert \geq 
C' +
\inf_{t \in S'}\left( - \left\lVert  \nabla \eta(t) \right\rVert \right)\text{ and so}
\end{equation*}
\begin{equation*}
\mathbb{P}(\inf_{t \in S'} \left\lVert \nabla \hat{\gamma}(t)  \right\rVert > 0 )
\geq \mathbb{P}\left( \sup_{t \in S'}\left\lVert \nabla \eta(t) \right\rVert < C'\right) = 1 - \mathbb{P}\left( \sup_{t \in S'}\left\lVert \nabla \eta(t) \right\rVert > C'\right).
\end{equation*}

\subsubsection{Proof of Proposition \ref{prop:peakconv}}\label{proof:peakconv}
The probability that $ \hat{\gamma}_N $ has no critical points in $ S \setminus B $ is greater than the probablity that $ \inf_{t \in S \setminus B } \left\lVert \nabla \hat{\gamma}_N(t)  \right\rVert > 0 $ and by Lemma \ref{lem:derivbound},
\begin{equation*}
\mathbb{P}\left( \inf_{t \in S \setminus B } \left\lVert \nabla \hat{\gamma}_N(t)  \right\rVert > 0  \right)
\geq 
1- \mathbb{P}\left(  \sup_{t \in S \setminus B }\left\lVert \nabla \eta_N(t) \right\rVert > C  \right).
\end{equation*}
By the continuous mapping theorem, $ \nabla \eta_N \convup 0$ implies that $ \left\lVert  \nabla \eta_N \right\rVert  \convup 0 $. We have
\begin{equation*}
\sup_{t \in S \setminus B } \left\lVert \nabla \eta_N(t) \right\rVert  \leq \sup_{t \in S} \left\lVert \nabla \eta_N(t) \right\rVert 
\end{equation*} 
so in particular, 
\begin{equation*}
\sup_{t \in S \setminus B } \left\lVert \nabla \eta_N(t) \right\rVert \convp 0
\end{equation*} 
from which the first result follows. As in \cite{Cheng2017a}, the probability that $ \hat{\gamma}_N$ has no maxima is less than or equal to 
\begin{equation*}
\mathbb{P}\left(\sup_{t \in S \setminus B }\left\lVert \nabla \eta_N(t)  \right\rVert > \delta_jD_{\max}\right)
\end{equation*} 
which tends to 0 by a similar argument to above. Note that the argument in \cite{Cheng2017a} requires an application of the Fundamental Theorem of Calculus to the second derivative which is why we require that the second derivatives of $ \gamma $ and $ \eta_N $ are continuous. The probability of 2 or more local maxima is
\begin{equation*}
\mathbb{P}\left(\sup_{t \in S \setminus B }\sup_{\left\lVert x \right\rVert = 1} x^T\nabla^2 \eta_N(t) x > D_{\max}\right)
\end{equation*}
which tends to 0, since 

\begin{equation*}
\sup_{t \in S \setminus B }\sup_{\left\lVert x \right\rVert = 1} x^T\nabla^2 \eta_N(t) x \leq 	\sup_{t \in S \setminus B }\left\lVert \nabla^2 \eta_N(t) \right\rVert \convas 0.
\end{equation*}
Combining these last two convergences implies that the probability that there is exactly one maximum converges to 1.

\newpage
\section{Proofs for Section 4}
\subsection{Conditions for convergence}
\begin{lem}\label{lem:pointwisesigma2}
	Given $a_1, \dots, a_N, \mu \in \mathbb{R}$ for some $ N \in \mathbb{N} $ and letting $ \mybar{a} = \frac{1}{N}\sum_{n = 1}^N a_n$,
	\begin{align*}
	\frac{1}{N}\sum_{n = 1}^N\left(a_n- \mybar{a}\right)^2 &= \frac{1}{N}\sum_{n = 1}^N\left(a_n - \mu + \mu - \mybar{a}\right)^2\\
	&= \frac{1}{N}\sum_{n = 1}^N(a_n - \mu)^2 + \frac{2}{N}\sum_{n = 1}^N\left(\mu - \mybar{a}\right)(a_n - \mu)
	+ \frac{1}{N}\sum_{n = 1}^N\left(\mu - \mybar{a}\right)^2\\
	&=\frac{1}{N}\sum_{n = 1}^N(a_n - \mu)^2 + \left(\mu - \mybar{a}\right)\frac{2}{N}\sum_{n = 1}^N(a_n - \mu) + \left(\mu - \mybar{a}\right)^2.
	\end{align*}
\end{lem}

\begin{lem}\label{lem:inverse}
	In the setting of Section \ref{SS:LM} let $ g $ be the density of $ x_1 $ and suppose that $ g $ is bounded above by $ G \in \mathbb{R}$ then
	$ \left( \frac{1}{N} X_N^TX_N \right)^{-1} \convas (\mathbb{E}\left[ x_1x_1^T \right])^{-1} $.
\end{lem}
\begin{proof}
	Define the event $ E_N = \left\lbrace X \in \mathbb{R}^{N \times p}:\det(X^TX) = 0 \right\rbrace $ then, as $ x_1, \dots, x_N $ are independent,
	\begin{equation*}
	\mathbb{P}\left( X_N \in E_N \right) = \int_{E_N} \prod_{n = 1}^N g(x_n) \,dx_n \leq G^N \int_{E_N} \,\prod_{n = 1}^Ndx_n = 0.
	\end{equation*}
	since $ E_N $ traces out a lower dimensional subspace. So by countability (as the countable union of measure zero sets has measure zero) we can almost surely assume that $ \frac{1}{N}X_N^TX_N $ is invertible for all $ N. $ Now, \begin{equation*}
	\frac{1}{N}X_N^TX_N = \frac{1}{N}\sum_{n = 1}^N x_nx_n^T
	\end{equation*} 
	which converges to $ \mathbb{E}\left[ x_1x_1^T \right] $ by the SLLN (note that if $ x_1 $ is mean zero then this equals $ \cov(x_1) $). As such the result follows by applying the continuous mapping theorem.
\end{proof}
%

\subsection{Continuation of the proof of Theorem \ref{thm:meanclt}}\label{SS:meanthm}
\noindent In what follows we verify that Assumptions CF and EE2 of \cite{Xiaoxia} hold in our setting. Our result then follows by applying their Theorem 4.1.

\begin{enumerate}[label=\roman*)]
	\item For each $ 1 \leq j \leq J, $ $ B_j $ is compact and $ \hat{\mu}_N = \frac{1}{N}\sum_{n = 1}^N Y_n $ is continuous, pointwise measurable and converges uniformly almost surely to $ \mu $ as shown in Section \ref{SS:MC2}. 
	Furthermore, $ \argmax_{t \in B_j} \mu(t) = \theta_j $ is the unique maximum  (or minimum if $ \theta_j $ is a minimum), so it follows that $ \hat{\theta}_{j,N} \convp \theta_j$ by Theorem 4.1.1 from \cite{Amemiya1985}. In particular $ \bf{\hat{\theta}_N} \convp \bf{\theta} $.
	\item For $ N \in \mathbb{N} $, $ \hat{\mu}_N $ is a.s. twice continuously differentiable on $ B_j $. Twice differentiable is required throughout and $ C^2 $ is required for Taylor's theorem.
	\item For each $ 1 \leq j \leq J$, $ \cov(\nabla^T Y_1(\theta_j)) < \infty $ by assumption. By derivative exchangeability, which follows from Assumption \ref{ass:derivexch}bii, $ \mathbb{E}\left[ \nabla^TY_1(\theta_j) \right] = \nabla^T \mu(\theta_j) = 0 $. By Cauchy-Swartz, for $ 1\leq i,j\leq J, $ the cross-covariances: $ \cov(\nabla^TY_1(\theta_i),\nabla^TY_1(\theta_j)) $ are also finite.  As such
	\begin{equation*}
	\sqrt{N}\pmatrix{\nabla^T \hat{\mu}_N(\theta_1);\vdots;\nabla^T \hat{\mu}_N(\theta_J)} = \frac{1}{\sqrt{N}}\sum_{n = 1}^N \pmatrix{\nabla^T Y_n(\theta_1);\vdots;\nabla^T Y_n(\theta_J)} \convd \mathcal{N}\left( 0, \bf{\Lambda} \right).
	\end{equation*} 
	by the multivariate CLT.
	\item For each $ 1 \leq j \leq J, $ given a sequence $ \theta^*_{j,N} \convp \theta_j $,
	\begin{align*}
	\nabla^2\hat{\mu}_N(\theta^*_{j,N}) = \nabla^2\hat{\mu}_N(\theta^*_{j,N}) - \nabla^2\mu(\theta^*_{j,N}) + \nabla^2\mu(\theta^*_{j,N}) - \nabla^2\mu(\theta_{j}) + \nabla^2\mu(\theta_{j})
	\end{align*}  
	which converges in probability to $ \nabla^2 \mu(\theta_j)$. This follows as by the fSLLN (which we can apply to the second derivative because of Assumption \ref{ass:derivexch}biii),
	\begin{equation*}
	\nabla^2\hat{\mu}_N(\theta^*_{j,N}) - \nabla^2\mu(\theta^*_{j,N})\convas 0.
	\end{equation*}
	Now, $ \nabla^2 \mu $ is uniformly continuous on $ B_j $ (as $ \mu $ is $ C^2 $) and so $\nabla^2\mu(\theta^*_{j,N}) - \nabla^2\mu(\theta_{j})\convp 0$ (since $ \hat{\theta}_{j,N} \convp \theta_j$ implies that $ \theta^*_{j,N} \convp \theta_j$). In particular $ \bf{A_N} \convp \bf{A} $.
\end{enumerate}
\vspace{0.35cm}
Combining (iii) and (iv) into \eqref{eq:cltsort} the result follows by applying Slutsky's Lemma.

\subsection{Proof of Theorem \ref{thm:lmmean}}\label{AA:lmmean}
This result follows via a similar proof to that of Theorem \ref{thm:meanclt} and requires us to demonstrate uniform convergence of the second derivative and to provide a pointwise CLT for the derivative. Uniform convergence follows from Proposition \ref{prop:lmconv} which implies that
\begin{equation*}
\nabla^2 w^T\hat\beta_N(\theta_j) \underset{N\rightarrow\infty}{\convuas} w^T\nabla^2 \beta(\theta_j)
\end{equation*}
for $ 1\leq j \leq J $. To prove the pointwise CLT, note that for all $ s \in S,$ 
\begin{align*}
&\sqrt{N} \nabla w^T(\hat\beta_N(s) - \beta(s)) = w^T\left( \frac{X_N^TX_N}{N} \right)^{-1} \frac{1}{\sqrt{N}}\sum_{n = 1}^N x_n \nabla \epsilon_n(s)\\
&= w^T\mathbb{E}\left[ x_1x_1^T \right]^{-1} \frac{1}{\sqrt{N}}\sum_{n = 1}^N x_n \nabla \epsilon_n(s) + w^T\left( \left( \frac{X_N^TX_N}{N} \right)^{-1} -  \mathbb{E}\left[ x_1x_1^T \right]^{-1}\right) \frac{1}{\sqrt{N}}\sum_{n = 1}^N x_n \nabla \epsilon_n(s).
\end{align*}
By Lemma \ref{lem:inverse}, $ \left( \frac{1}{N}X_N^TX_N \right)^{-1} $ converges almost surely to $ (\mathbb{E}\left[ x_1x_1^T \right])^{-1} $ and using the CLT it follows that so the second term converges to zero in distribution. Applying the multivariate CLT it follows that,
\begin{equation*}
\sqrt{N}\left( \nabla w^T \hat\beta(\theta_1), \dots, \nabla w^T \hat\beta(\theta_J) \right)^T \convd \mathcal{N}(0, \bf{\Lambda}).
\end{equation*}
Combining these results with the argument used in Theorem \ref{thm:meanclt} completes the proof.

\subsection{Proofs for Section \ref{SS:CD}}

\subsubsection{Proof of Lemma \ref{lem:Tstatderiv}}
\begin{proof}
	Define $ Z_N, V_N $ as in equation (\ref{eq:Tdef}) and note that for ease of notation we will drop the dependence on $ s $ in what follows. Differentiating $ T_N $, we have,
	\begin{align*}
	\nabla T_N &= \left(\frac{N-1}{V_N}\right)^{1/2}\left(\sqrt{N}\nabla \mu + \nabla Z_N\right) - \frac{\sqrt{N}\mu + Z_N}{2V_N^{3/2}/\sqrt{N-1}}\nabla V_N \text{ and so }\\
	\nabla T_N|Z_N, V_N &\overset{d}{=} \left(\frac{N-1}{V_N}\right)^{1/2}\left(\sqrt{N}\nabla \mu + z_X\right) - \frac{\sqrt{N}\mu+ Z_N}{2V_N^{3/2}/\sqrt{N-1}}2V_N^{1/2}z_V
	\end{align*} 
	where $ z_X $ and $ z_V \iid \mathcal{N}(0, \Lambda) $. Note that the equality in distribution follows by Lemma 3.2 of \cite{Worsley1994} which takes advantage of the independence between a constant variance Gaussian random field and its derivative and the fact that the square of the fields satisfy the DE condition (see Proposition \ref{prop:Gaussfin}). See Appendix \ref{A:chi2deriv} for a generalization of Worsley's result to non-stationary $ \chi^2 $ random fields. $ z_X $ and $ z_V $ are independent of $ Z_N $ and $ V_N $ since the former are functions of the derivatives and the later are functions of the component fields. Thus
	\begin{align*}
	\nabla T_N|Z_N, V_N &
	\overset{d}{=}  \left(\frac{N-1}{V_N}\right)^{1/2}\left(\sqrt{N}\nabla \mu + \mathcal{N}(0,\Lambda) + \frac{\sqrt{N}\mu+ Z_N}{\sqrt{V_N}}\mathcal{N}(0, \Lambda)\right) \\
	&\overset{d}{=}  \left(\frac{N-1}{V_N}\right)^{1/2}N\left(\sqrt{N}\nabla \mu , \left(1+\frac{(\sqrt{N}\mu+  Z_N)^2}{V_N}\right)\Lambda\right)\\
	&\overset{d}{=} \left(\frac{N-1}{V_N}\right)^{1/2}N\left(\sqrt{N}\nabla \mu, \left(1+\frac{T_N^2}{N -1}\right)\Lambda\right).
	\end{align*} 
	Since conditioning on $ Z_N $ and $ V_N $ is equivalent to conditioning on $ \hat{\sigma}_N $ and $ T_N $, the result follows.
\end{proof}
\noindent In particular, if the fields are Gaussian and have unit variance, it follows that we have the following pointwise CLT for Cohen's $ d $:
\begin{equation*}
\sqrt{N}\left(\nabla d_N -\frac{\nabla \mu}{\hat{\sigma}_N}\right)= \nabla T_N - \frac{\nabla \mu\sqrt{N}}{\hat{\sigma}_N}\underset{N\rightarrow \infty}{\convd} \mathcal{N}\left(0, \left( 1+ \mu^2 \right)\Lambda\right).
\end{equation*} 
This follows since $ \hat{\sigma}_N \convas \sigma $ and $ \frac{T_N^2}{N-1}\conv \mu^2 $ as $ N \rightarrow \infty $. Let us now drop the constant variance condition and write
\begin{align}\label{eq:modT}
T_N &= 
\frac{\frac{1}{\sqrt{N}}\sum_{n = 1}^N Y_n}{\left( \frac{1}{N-1}\sum_{n = 1}^N (Y_n - \frac{1}{N}\sum Y_n)^2 \right)^{1/2}}\\ 
&= \frac{\frac{1}{\sqrt{N}}\sum_{n = 1}^N Y_n/\sigma}{\left( \frac{1}{N-1}\sum_{n = 1}^N (Y_n/\sigma - \frac{1}{N}\sum Y_n/\sigma)^2 \right)^{1/2}}
\end{align}
which is the $ t $-statistic derived from component Gaussian random fields $ Y_n' = Y_n/\sigma $, $ n = 1, \dots, N, $ which are i.i.d and have constant variance 1. We can thus apply the constant variance result to yield the following corollary.

\subsubsection{Proof of Lemma \ref{cor:Tstatderiv}}\label{AA:Tstatderiv2}
Applying Lemma \ref{lem:Tstatderiv} to the $ t $-statistic from equation \eqref{eq:modT} we obtain the distributional result with $ \Lambda'(s) = \cov\left(\nabla^T \frac{Y_1(s)}{\sigma(s)}\right)  $. Dropping dependence on $ s $ and expanding,
\begin{align*}
&\cov\left(\nabla^T \frac{Y_1}{\sigma}\right) 
= \displaystyle\mathbb{E}\left[\left(\frac{\nabla^T Y_1}{\sigma} - \frac{Y_1}{2\sigma^3}\nabla^T \sigma^2\right)\left(\frac{\nabla^T Y_1}{\sigma} - \frac{Y_1}{2\sigma^3}\nabla\sigma^2\right)^T\right] \\
&\hspace{1cm}=\frac{\mathbb{E}\left[ (\nabla Y_1)^T\nabla Y_1  \right]}{\sigma^2} - 
\frac{\nabla^T \sigma^2 \mathbb{E}\left[ Y_1(\nabla Y_1) \right]}{2\sigma^4} - 
\frac{\mathbb{E}\left[ (\nabla Y_1)^T(Y_1 \nabla \sigma^2)\right]}{2\sigma^4} 
+ \frac{\mathbb{E}\left[ Y_1^2 \right](\nabla \sigma^2)^T\nabla \sigma^2 }{4\sigma^6}\\
&\hspace{1cm}=\frac{\Lambda}{\sigma^2} - \frac{\nabla^T \sigma^2 \Gamma}{\sigma^4} + \frac{(\nabla \sigma^2)^T\nabla \sigma^2}{4\sigma^4}.
\end{align*} 

\subsubsection{Statement and proof of Lemma \ref{lem:derivclt}}\label{AA:derivcltprep}
To do so we first require the following lemma.
\begin{lem}\label{lem:derivclt}
	Assume that $ Y_1 $ is unit variance, satisfies the DE condition and that $ \Lambda(s) < \infty $ for all $ s \in S $. Then for all $ J \in \mathbb{N} $ and $ s_1, \dots, s_J \in S, $
	\begin{align*}
	\left(\nabla^T Z_N(s_1), \nabla^T V_N(s_1)/\sqrt{N},\cdots,
	\nabla^T Z_N(s_J), \nabla^T V_N(s_J)/\sqrt{N}\right)^T
	\end{align*} 
	(where $ Z_N $ and $ V_N $ are defined as in \eqref{eq:Tdef}) satisfies a CLT as $ N \rightarrow \infty. $
\end{lem}
\begin{proof}
	\noindent Using this lemma we can prove the following theorem which generalizes (\ref{eq:GaussCDconv}).
\noindent Differentiating $ \hat{\sigma}^2_N $ (and evaluating all fields pointwise) and letting $ \mybar{\epsilon}_N =   \frac{1}{N}\sum_{j = 1}^N \epsilon_j$, we have
\begin{align*}
\sqrt{N}\nabla \hat{\sigma}_N^2 &= \frac{2\sqrt{N}}{N-1}\sum_{n=1}^N\left( \epsilon_n - \mybar{\epsilon}_N  \right)\left( \nabla \epsilon_n - \frac{1}{N}\sum_{k = 1}^N \nabla \epsilon_k \right)\\
&= \frac{2\sqrt{N}}{N-1}\sum_{n = 1}^N \epsilon_n \nabla \epsilon_n - \frac{2\sqrt{N}}{N-1} \sum_{n = 1}^N \epsilon_n\left(\frac{1}{N}\sum_{k = 1}^N \nabla \epsilon_k\right) \\
&\quad\quad - \frac{2\mybar{\epsilon}_N\sqrt{N}}{N-1}\sum_{n = 1}^N \nabla \epsilon_n + \frac{2\mybar{\epsilon}_N N\sqrt{N}}{N-1}  \left(\frac{1}{N}\sum_{k = 1}^N \nabla \epsilon_k \right).
\end{align*} 
Now $ \frac{2\bar{\epsilon}_N\sqrt{N}}{N-1} $ converges in distribution by the CLT (as $ \var(\epsilon_n) < \infty$) and $ \frac{1}{N}\sum_{k = 1}^N \nabla \epsilon_k \convas 0 $ by the SLLN as $ Y_1 $ satisfies the DE condition and so 
\begin{equation*}
\left(\frac{2\sqrt{N}}{N-1} \sum_{n = 1}^N \epsilon_n\right)\left(\frac{1}{N}\sum_{k = 1}^N \nabla \epsilon_k\right) \convp 0 \text{ as } N \longrightarrow \infty.
\end{equation*} 
Similarly the third and fourth terms converge in probability to zero as $ N \longrightarrow \infty. $ 
\noindent We can thus write
\begin{equation*}
\pmatrix{\nabla^T Z_N(\theta_1); \nabla^T V_N(\theta_J)/\sqrt{N};\vdots;
	\nabla^T Z_N(\theta_1); \nabla^T V_N(\theta_J)/\sqrt{N}} = \sqrt{N}\pmatrix{\frac{1}{N} \sum_{n = 1}^N\nabla^T \epsilon_n(\theta_1); \nabla^T \hat{\sigma}^2_N(\theta_1); \vdots; \frac{1}{N} \sum_{n = 1}^N\nabla^T \epsilon_n(\theta_J); \nabla^T \hat{\sigma}^2_N(\theta_J)}  = 
\sqrt{N}\frac{1}{N}\sum_{n = 1}^N \pmatrix{\nabla^T \epsilon_n(\theta_1); 2\epsilon_n(\theta_1)\nabla^T \epsilon_n(\theta_1); \vdots; \nabla^T \epsilon_n(\theta_1); 2\epsilon_n(\theta_J)\nabla^T \epsilon_n(\theta_J)}
\end{equation*} 
\noindent where the last equality holds up to a term that converges to zero in probability, by Slutsky. The result follows by applying the multivariate CLT.
%
\end{proof}

\subsubsection{Proof of Theorem  \ref{thm:derivclt}}\label{AA:derivclt}
We will take the same approach as we did in the proof of Lemma \ref{lem:Tstatderiv} and Corollary \ref{cor:Tstatderiv}, namely to first prove the result assuming that the variance is constant and then use this to obtain the general result. So assume that $ Y_1$ has variance 1 everywhere. Then,
\begin{equation*}
\nabla T_N = 
\sqrt{N-1}\left(\frac{\sqrt{N}\nabla \mu + \nabla Z_N}{V_N^{1/2}}\right) - 
\sqrt{N-1}\left(\frac{\sqrt{N}\mu+ Z_N}{2V_N^{3/2}}\right)\nabla V_N.
\end{equation*} 
Thus
\begin{align*}
\nabla T_N - \frac{\sqrt{N}\nabla \mu}{\sqrt{V_N/N-1}} &=  
\nabla Z_N + \left(\sqrt{\frac{N-1}{V_N}} - 1\right) \nabla Z_N 
- \frac{N-1}{V_N} \left( \frac{\sqrt{N}\mu + Z_N}{2V_N^{1/2}}\right)\frac{\nabla V_N}{\sqrt{N-1}}\\
&= \nabla Z_N 
- \frac{\mu\nabla V_N}{2\sqrt{N-1}} +
\left(\sqrt{\frac{N-1}{V_N}} - 1\right) \nabla Z_N\\
&\hspace{2cm}+
\left(\frac{\mu}{2}- \frac{N-1}{V_N} \left( \frac{\sqrt{N}\mu + Z_N}{2V_N^{1/2}}\right)\right)\frac{\nabla V_N}{\sqrt{N-1}}. 
\end{align*} 
The last two terms converge to zero in distribution (by the usual arguments involving Slutsky by applying the CLT and using the fact that $ \frac{\sqrt{N}\mu + Z_N}{\sqrt{V_N}} \convas \mu$, see Section \ref{SS:Tderivdist}). Applying Slutsky again and using the joint asymptotic distribution of $ (\nabla Z_N, \nabla V_N/\sqrt{N} ) $, derived in Lemma \ref{lem:derivclt} gives the result in the unit-variance case. Dropping the assumption of constant variance, the general result follows by considering the fields $ Y_n/\sigma $, arguing as in the proof of Corollary \ref{cor:Tstatderiv}.
%

\subsubsection{Proof of Remark \ref{rem:rem}}\label{A:remark}

Assume that the fields are unit-variance. Then since the fields are Gaussian, $ Y_1^2 $ satisfies the DE condition and so $ \nabla Y_1 $ is independent of $ Y_1 $ at each point. As such $ \mathbb{E}\left[ \epsilon_1 \nabla^T \epsilon_1 \nabla \epsilon_1 \right] = 0 $ ($ \epsilon_1$ is mean zero as $Y_1$ satisfies the DE condition) and 
\begin{equation*}
\cov\left(\epsilon_1\nabla^T\epsilon_1\right) =  \var\left(\epsilon_1\right)\cov\left(\nabla^T\epsilon_1\right) = 
\cov\left(\nabla^T\epsilon_1\right) = \Lambda.
\end{equation*} 
So the $ 2 \times 2  $ block diagonal entries of the limiting covariance in Lemma \ref{lem:derivclt} equal $ \pmatrix{\Lambda. 0;0.4\Lambda}  $. As such, using the expansion of $ \nabla T_N - \frac{\sqrt{N}\nabla \mu}{\sqrt{V/N-1}} $ from the proof of Theorem \ref{thm:derivclt} it follows that
\begin{equation*}
\nabla T_N - \frac{\sqrt{N}\nabla\mu}{\sqrt{V_N/N-1}} \convd \mathcal{N}\left( 0, (1+\mu^2)\Lambda \right)
\end{equation*}
as $ N \rightarrow \infty$. The general result (for arbitrary variance) follows by arguing as in the proof of Corollary \ref{cor:Tstatderiv}.


\subsection{Proof of Theorem \ref{thm:CDlocclt}}\label{SS:cdthmproof}
\noindent In what follows we verify that Assumptions CF and EE2 of \cite{Xiaoxia} hold in our setting. Our result then follows by applying their Theorem 4.1.
\begin{proof}
	\begin{enumerate}[label=\roman*)]
		\item In our setting, $ B_j $ is compact and $ d_N $ is continuous, pointwise measurable and converges uniformly in probability to $ \mu $ as shown in Section \ref{SS:MC2}. Furthermore, $ \argmax_{t \in B_j} \frac{\mu(t)}{\sigma(t)} = \theta_j $ is the unique maximum (or minimum if $ \theta_j $ is a minimum) of $ \frac{\mu}{\sigma} $, so it follows that $ \hat{\theta}_{j,N} \convp \theta_j $ by Theorem 4.1.1 from \cite{Amemiya1985}. $ \theta_0 $ also lies on the interior of $ B_j$.  In particular $ \bf{\hat{\theta}_N} \convp \bf{\theta} $.
		\item $ d_N $ is a.s. twice continuously differentiable as $ Y_N $ are and the $ Y_n$ are non-degenerate.
		\item For $ 1 \leq j \leq J, $ $ \nabla \frac{\mu(\theta_j)}{\sigma(\theta_j} = 0 $ and so by Theorem \ref{thm:derivclt}, $ \sqrt{N}\left( \nabla d_N(s_1), \cdots, \nabla d_N(s_J) \right)^T $ satisfies a CLT.
		\item For each $ 1 \leq j\leq J,$  given a sequence $\tilde{\theta}_{j,N} \convp \theta_j$, 
		\begin{align*}
		\nabla^2d_N(\tilde{\theta}_{j,N}) = \nabla^2d_N(\tilde{\theta}_{j,N}) - \nabla^2\frac{\mu(\tilde{\theta}_{j,N})}{\sigma(\tilde{\theta}_{j,N})} + \nabla^2\frac{\mu(\tilde{\theta}_{j,N})}{\sigma(\tilde{\theta}_{j,N})} - \nabla^2\frac{\mu(\theta_{j})}{\sigma(\theta_j)} + \nabla^2\frac{\mu(\theta_{j})}{\sigma(\theta_j)}
		\end{align*}  
		which converges in probability to $ \nabla^2 \frac{\mu(\theta_j)}{\sigma(\theta_j)} $, since by Proposition \ref{prop:CDconv},
		\begin{equation*}
		\nabla^2d_N(\tilde{\theta}_{j,N}) - \nabla^2\frac{\mu(\tilde{\theta}_{j,N})}{\sigma(\tilde{\theta}_{j,N})} \convas 0
		\end{equation*}
		and as $ \nabla^2 \frac{\mu}{\sigma} $ is uniformly continuous $\nabla^2\frac{\mu(\tilde{\theta}_{j,N})}{\sigma(\tilde{\theta}_{j,N})} - \nabla^2\frac{\mu(\theta_{j})}{\sigma(\theta_j)}\convp 0$. In particular we have convergence in probability of the matrix of second derivatives.
	\end{enumerate}
	
\end{proof}

\newpage
\section{Further Results}
\subsection{Understanding the distribution of a symmetric ratio}\label{A:ratioexplan}
\begin{thm}\label{thm:AoverB}
	Suppose that $ A $ and $ B $ are independent real valued random variables with well defined densities $ p_A $ and $ p_B $  which are symmetric about $ \mathbb{E}\left[ A \right] $ and $ \mathbb{E}\left[ B \right] $ respectively. Assume that $ p_A(x) $ is decreasing for $ x> 0 $ and increasing for $ x < 0 $, $ B $ is positive and that $ \mathbb{E}\left[ \left| B \right| \right] < \infty$. Then for all $ x > 0, $  
	\begin{equation*}
	\mathbb{P}\left( \frac{A}{\mathbb{E}\left[ B \right]} > x \right) \leq \mathbb{P}\left( \frac{A}{B} > x \right).
	\end{equation*}
\end{thm}
\begin{proof}
	Let $ k = \mathbb{E}\left[ B \right] $, then 
	\begin{equation*}
	\mathbb{P}\left( \frac{A}{k} > x  \right) = \mathbb{P}\left( A > kx \right) = \int_{b = 0}^{2k}\int_{kx}^{\infty} p_A(a) p_B(b) \,da\,db
	\end{equation*}
	Here we have used the fact that $ p_B $ only has support in $ (0,2k) $ as $ B $ is positive and symmetric. Now, 
	\begin{equation*}
	\mathbb{P}\left( \frac{A}{B} > x  \right) = \mathbb{P}\left( A > Bx \right) = \int_{b = 0}^{2k}\int_{bx}^{\infty} p_A(a) p_B(b) \,da\,db.
	\end{equation*}
	As such,
	\begin{equation*}
	\mathbb{P}\left( \frac{A}{B} > x  \right) - \mathbb{P}\left( \frac{A}{k} > x  \right) = 
	\int_{b = 0}^k\int_{bx}^{kx} p_A(a) p_B(b) \,da \,db - \int_{b = k}^{2k}\int_{kx}^{bx} p_{A}(a)p_B(b) \,da \,db > 0
	\end{equation*}
	since for each $ b \in (0,k),  $ 
	\begin{equation*}
	\int_{bx}^{kx} p_A(a) \,da > \int_{kx}^{(2k-b)x}p_A(a) \, da
	\end{equation*}
	and $ p_B(b) = p_B(2k-b) $ by symmetry. 
\end{proof}

This theorem, applied in our setting, provides a justification for why we should expect the Monte Carlo confidence regions to have a higher level of coverage than their asymptotic counterparts. That is because conditional on the observed data we could obtain the quantiles, used to generate the asymptotic confidence intervals by generating data from a $ \mathcal{N}(0, \Lambda(\theta_j)) $ distribution and multiplying by $ (\nabla^2 \hat{\mu}_N)(\theta_j))^{-1} $ (note that here and for the rest of this discussion we pick an arbitrary $ 1 \leq j \leq J$). On the other hand the Monte Carlo distribution generates data via equation \eqref{eq:deltak}, using the $ A_{k,N}, B_{k,N} $ notation there $ \textbf{vech}^{-1}(B_{k,N}) $ is centred at $ \hat{\mu}_N(\theta_j)) $ and $ A_{k,N} \sim \mathcal{N}(0, \Lambda(\theta_j)) $. So we are in the setting of Theorem \ref{thm:AoverB}. The only difference is that $ \textbf{vech}^{-1}(B_{k,N}) $ is not guaranteed to be positive. Asymptotically this is essentially irrelevant because the distribution of $ \textbf{vech}^{-1}(B_{k,N}) $ is closely concentrated around $ -\nabla^2 \hat{\mu}_N(\theta_j)$ and so is kept away from 0 with very high probability. Nevertheless in the finite sample, in order to fully benefit from the results of the theorem and ensure stability of the Monte Carlo method, we recommend truncating $ \textbf{vech}^{-1}(B_k) $ to ensure that it is positive and to do this symmetrically so that the theorem applies.

\subsection{Peak Identifiability in practice}\label{S:peakidentifiability}
The requirement in Proposition \ref{prop:peakconv} that  $ \nabla\eta_N, \nabla^2\eta_N  \convup 0$ can be shown to hold in a number of reasonable settings. Here we will show that it holds for mean and $ t$-statistic fields and in the context of the linear model. To demonstrate this we will need to able to exchange integration and differentiation and then apply the fSLLN for which we will require Assumption \ref{ass:derivexch}.
\subsubsection{Mean and Cohen's $ d $}\label{SS:MC2}
Assume that the random fields $ (Y_n)_{n \in \mathbb{N}} $ satisfy Assumption \ref{ass:derivexch}a,b. Then we can apply the fSLLN and derivative exchangeability to yield
\begin{equation*}
\hat{\mu}_N - \mu \convuas 0,  \quad \nabla \hat{\mu}_N - \nabla \mu \convuas 0 \text{ and } \nabla^2 \hat{\mu}_N - \nabla^2 \mu \convuas 0.
\end{equation*}
In the same setting but for Cohen's $ d $ we have the following results.
\begin{lem}\label{lem:CDconv}
	Suppose that $ (Y_n)_{n \in \mathbb{N}} $ satisfy Assumption \ref{ass:derivexch}$ b(i) $ and $ c(i) $, then $ \hat{\sigma}_N^2 \convuas \sigma^2 $. Then, recalling that we have assumed that $ \inf_{s \in S} \sigma^2(s) > 0 $, it follows that $ \frac{1}{\hat{\sigma}_N} \convuas \frac{1}{\sigma}  \text{ and so } \displaystyle d_N \convuas d. $
\end{lem}
\begin{proof}
	Applying Lemma \ref{lem:pointwisesigma2} pointwise to expand $ \hat{\sigma}^{2} $, and the fSLLN multiple times to the elements of the resulting expansion, (and scaling by $ \frac{N}{N-1} $) it follows that $  \hat{\sigma}_N^2 \convuas \sigma^2 $. Since $ \inf_{s\in S} \sigma^2(s) > 0 $ the inverse is well-defined and so the final results follow by the continuous mapping theorem and by noting that $ \hat{\mu}_N \convuas \mu $.
\end{proof}
\begin{prop}\label{prop:CDconv}
	Suppose that $ (Y_n)_{n \in \mathbb{N}} $ satisfies Assumption \ref{ass:derivexch} then, $ \nabla\left( \hat{d}_N - d \right)  \convuas 0  $ and 
	$ \nabla^2\left(\hat{d}_N - d \right) \convuas 0. $
\end{prop}
\begin{proof}
	By Lemma \ref{lem:supbound} we can exchange both first and second derivatives of $ \sigma^2\epsilon_1^2 $ with the expectation so that 
	\begin{equation*}
	\nabla \sigma(s)^2 = \nabla  \mathbb{E}\left[ (\sigma(s) \epsilon_1(s))^2 \right] = \mathbb{E}\left[ \nabla (\sigma(s) \epsilon_1(s))^2 \right]
	\end{equation*}
	and
	\begin{equation*}
	\nabla^2 \sigma(s)^2 = \nabla^2 \mathbb{E}\left[ (\sigma(s) \epsilon_1(s))^2 \right] = \mathbb{E}\left[ \nabla^2 (\sigma(s) \epsilon_1(s))^2  \right].
	\end{equation*}
	As such, differentiating the expansion from Lemma \ref{lem:pointwisesigma2} (applied to $ \hat{\sigma}^{2} $), and applying the fSLLN multiple times to the expansion, it follows that 
	$  \nabla\hat{\sigma}^2_N \convuas \nabla \sigma^2  $. As such, applying Lemma \ref{lem:CDconv}, we have
	\begin{equation*}
	\nabla \hat{\sigma}_N = \nabla \left( \hat{\sigma}^2_N \right)^{1/2} = \frac{\nabla \hat{\sigma}^2_N}{2\hat{\sigma}_N} \convuas \frac{\nabla \sigma^2}{2\sigma} = \nabla \sigma.
	\end{equation*}
	Similarly, $ \nabla^2 \hat{\sigma}_N \convuas \nabla^2 \sigma.$ $ \frac{1}{\hat{\sigma}_N} \convuas \frac{1}{\sigma} $ by Lemma \ref{lem:CDconv}, so it follows that
	\begin{equation*}
	\nabla \left( \frac{\hat{\mu}_{N}}{\hat{\sigma}_N} - \frac{\mu}{\sigma}\right) =
	\frac{\nabla \hat{\mu}_{N}}{\hat{\sigma}_N} - \nabla\left( \frac{1}{\hat{\sigma}_N} \right)\hat{\mu} - \frac{\nabla\mu}{\sigma} + \nabla\left( \frac{1}{\sigma} \right)\mu \convuas 0.
	\end{equation*}
	The proof for the second derivative is similar.
\end{proof}

These results mean that Proposition \ref{prop:peakconv} can be applied to mean and $ t $-fields.
\subsubsection{Linear Model}\label{SS:LM}
The linear model falls naturally into the signal plus noise framework and so the identifiability results of Proposition \ref{prop:peakconv} can be shown to apply. To formalize this, let $ p \in \mathbb{N} $ be the number of predictors and let $ \mathcal{X} $ be a multivariate distribution on $ \mathbb{R}^p $ with finite second moments, with density that is bounded above and such that if $ x \sim \mathcal{X} $ then $\cov(x) $ is positive definite. Let $ (x_n)_{n \in \mathbb{N}} $ be a sequence of independent random vectors in $ \mathbb{R}^p $ such that $ x_n \sim \mathcal{X} $ for all $ n$ 
and for each $ N \in \mathbb{N} $ set $ X_N = (x_1 \dots x_N)^T \in \mathbb{R}^{N \times p } $. Define a sequence of random fields $ (Y_n)_{n \in \mathbb{N}} $ on $ S $ such that for $ s \in S $, $ Y_n(s) = x_n^T\beta(s) + \sigma(s)\epsilon_n(s), $
where the $ \epsilon_n$ are i.i.d real-valued mean-zero and variance-one random fields and $ \beta(s) \in \mathbb{R}^p $. Let $ Y^N = [Y_1, \dots, Y_N]^T $ and $ \pmb{\epsilon}_N = [\epsilon_1, \dots, \epsilon_N]^T \in \mathbb{R}^N$. Given some contrast vector $ w \in \mathbb{R}^p $ let $ \gamma = w^T\beta $ and define
\begin{align}\label{eq:GLMfield}
\hat{\gamma}_N &= w^T\hat{\beta}_N = w^T(X_N^TX_N)^{-1}X_N^TY^N = w^T(X_N^TX_N)^{-1}X_N^T(X_N\beta + \pmb{\epsilon}_N)\\
&= w^T\beta + w^T(X_N^TX_N)^{-1}X_N^T \pmb{\epsilon}_N. 
\end{align}
This model thus falls under the signal plus noise framework and we have the following result. (Treating the linear model as a signal plus noise model is relatively common, see e.g. \cite{Sommerfield2018} and \cite{TelschowHPE}.)
\begin{prop}\label{prop:lmconv}
	Suppose that the $ \epsilon_n $ are twice differentiable, are independent of the $ x_n $ and satisfy Assumption \ref{ass:derivexch}b. Then as $ N \rightarrow \infty, $ 
	\begin{equation*}
	(X_N^TX_N)^{-1}X_N^T\pmb{\epsilon}_N \convuas 0, \nabla(X_N^TX_N)^{-1}X_N^T\pmb{\epsilon}_N \convuas 0 \text{ and } \nabla^2(X_N^TX_N)^{-1}X_N^T\pmb{\epsilon}_N \convuas 0.
	\end{equation*}
	In particular for  $w \in \mathbb{R}^p $, $ \hat{\gamma}_N- \gamma \convuas 0, \nabla(\hat{\gamma}_N- \gamma) \convuas 0 \text{ and } \nabla^2(\hat{\gamma}_N- \gamma) \convuas 0. $
\end{prop}
\begin{proof}
	We will prove the result for the first derivative, the other results follow similarly. Now,
	\begin{equation*}
	\frac{1}{N}X_N^T \nabla \pmb{\epsilon}_N = \frac{1}{N}\smatrix{x_1, \cdots, x_N}\nabla \pmb{\epsilon}_N
	\end{equation*}
	and so for $ i = 1, \dots, p $ and $ j = 1, \dots, D $,  
	\begin{equation*}
	\left( \frac{1}{N}X_N^T \nabla\pmb{\epsilon}_N \right)_{i,j} = \frac{1}{N}\sum_{n = 1}^N\left( (X_N)^T_{in}(\nabla\pmb{\epsilon}_N)_{nj} \right) = \frac{1}{N}\sum_{n = 1}^N(x_n)_i \frac{\partial \epsilon_n}{\partial t_j} \convuas \mathbb{E}\left[ (x_1)_i \frac{\partial \epsilon_1}{\partial t_j} \right]
	\end{equation*}
	$\text{as } N\longrightarrow \infty. $ For each $ i $, $(x_n)_i $ for $ n = 1, \dots, N $ are i.i.d as are $ \frac{\partial \epsilon_n}{\partial t_j} $ for each $ j$, so for all $ i,j $, $ (x_n)_i \frac{\partial \epsilon_n}{\partial t_j}$ are i.i.d for $ n = 1, \dots, N $. Additionally by independence and since the noise satisfies Assumption \ref{ass:derivexch},
	\begin{equation*}
	\mathbb{E}\left[ \sup_{s \in S} \left| (x_n)_i \frac{\partial \epsilon_n}{\partial t_j} \right| \right] \leq \mathbb{E}\left| (x_n)_i \right| \mathbb{E}\left[ \sup_{s \in S}\left| \frac{\partial \epsilon_n}{\partial t_j} \right| \right] < \infty
	\end{equation*}	
	so the convergence above occurs by the fSLLN and the limit equals \begin{equation*}
	\mathbb{E}\left[ (x_1)_i \frac{\partial \epsilon_1}{\partial t_j} \right] = \mathbb{E}\left[ (x_1)_i \right]\mathbb{E}\left[ \frac{\partial \epsilon_1}{\partial t_j} \right] = 0.
	\end{equation*} 
	Now $ N(X_N^TX_N)^{-1} = (\frac{1}{N}X_N^TX_N)^{-1} \convas \Sigma^{-1}$ as $N \longrightarrow \infty $ (using Lemma \ref{lem:inverse}) and so
	\begin{equation*}
	\nabla(X_N^TX_N)^{-1}X_N^T\pmb{\epsilon}_N = \left( \frac{1}{N}X_N^TX_N \right)^{-1}\left( \frac{1}{N}X_N^T\nabla\pmb{\epsilon}_N  \right) \convuas 0 \text{ as } N \longrightarrow \infty.
	\end{equation*}
	Since
	\begin{align*}
	\hat{\gamma}_N - \gamma = w^T\hat{\beta}_N - w^T\beta = w^T(X_N^TX_N)^{-1}X_N^T(X\beta + \epsilon) - w^T\beta = w^T(X_N^TX_N)^{-1}X_N^T\pmb{\epsilon}_N,
	\end{align*}
	the second set of results follow immediately.
\end{proof}
Thus, if we assume that $ w^T\beta $ satisfies Assumption \ref{ass:peakconv}, Proposition \ref{prop:peakconv} applies in this linear model setting. The CLT results of Section \ref{SS41} can also be extending to the linear model as the following theorem shows.

\begin{thm}\label{thm:lmmean}
	In the linear model setting (described in Section \ref{SS:LM}) assume that $ \epsilon_1 $ satisfies Assumption \ref{ass:derivexch}a,b and that $ w^T\beta $ satisfies Assumption \ref{ass:peakconv}. Let $ \hat{\theta}_{j,N} = \argmax_{t \in B_j} w^T\hat\beta_N(t)$ and define 
	$ \boldsymbol{\hat{\theta}_{N}}$ and $ \bf{\theta} $ as in Theorem \ref{thm:meanclt}. Then 
	\begin{equation*}
	\sqrt{N}(\bf{\hat{\theta}_N} - \bf{\theta}) \convd \mathcal{N}(0, \bf{A}\bf{\Lambda}\bf{A}^T)
	\end{equation*} 
	as $ N \conv \infty $. Here $ \bf{A} $ $ \in \mathbb{R}^{DJ\times DJ}$ is a block diagonal matrix with $ J $ diagonal $ D\times D $ blocks such that, for $ j \in 1, \dots, J$, the $ j $th block is $ -\left( w^T\nabla^2 \beta(\theta_j)\right)^{-1}$ and $ \bf{\Lambda} \in \mathbb{R}^{DJ\times DJ}$ is a matrix such that for $ i,j = 1,\dots, J$ the $ ij $th $ D \times D$ block of $ \bf{\Lambda} $ is \\
	\begin{equation*}
	\cov(\nabla^T \epsilon_1(\theta_i) x_1^T\mathbb{E}\left[ x_1x_1^T \right]^{-1}w, \nabla^T \epsilon_1(\theta_j) x_1^T\mathbb{E}\left[ x_1x_1^T \right]^{-1}w).
	\end{equation*}
\end{thm}

\subsection{The derivative of a $ \chi^2 $ field}\label{A:chi2deriv}
\begin{lem}\label{lem:chi2deriv}
	Let $Y_1, \dots, Y_N $ be zero mean i.i.d $ D $-dimensional Gaussian random fields on $ S $ (with variance $ \sigma^2$ that is not necessarily constant). Let $ U = \sum_{n = 1}^N Y_n^2$ and $ Y = (Y_1, \dots, Y_N)^T $, then
	\begin{equation*}
	\left.\nabla^T U(s) \right| Y(s) \sim \mathcal{N}\left(\frac{2\Gamma(s) U(s)}{\sigma(s)^2}, 4U(s)(\Lambda(s) - \Gamma(s) \Gamma(s)^T/\sigma^2(s))\right).
	\end{equation*} 
\end{lem}
\begin{proof}
	Evaluating all quantities pointwise, for each $ n = 1, \dots, D $,
	\begin{equation*}
	(Y_n, \nabla Y_n)^T \sim \mathcal{N}\left(0 , \pmatrix{\sigma^2. \Gamma^T; \Gamma. \Lambda}\right)
	\end{equation*} 
	\noindent as such using the formula for the conditional Gaussian distribution,
	\begin{equation*}
	\nabla^T Y_n | Y \sim \mathcal{N}\left(\frac{\Gamma}{\sigma^2}Y_n, \Lambda - \Gamma\Gamma^T/\sigma^2\right).
	\end{equation*} 
	Now differentiating $ U $ we find that
	\begin{align*}
	\nabla^T U | Y = 2\sum_{n = 1}^N Y_n \nabla^T Y_n 
	&\sim \mathcal{N}\left(\frac{2\Gamma}{\sigma^2}\sum_{n = 1}^N Y_n^2, 4\sum_{n = 1}^N Y_n^2(\Lambda - \Gamma\Gamma^T/\sigma^2)\right)\\
	&\sim \mathcal{N}\left(\frac{2\Gamma}{\sigma^2}U, 4U(\Lambda - \Gamma\Gamma^T/\sigma^2)\right).
	\end{align*} 
\end{proof}
\subsection{Expanding the first and second derivatives of Cohen's $ d $}
As mentioned in the Discussion, developing a Monte Carlo method for Cohen's $ d $ is difficult because it requires an understanding of the joint distribution of the first and second derivatives. These expansions are not simple and are shown below for completeness. Differentiating $ d_N $, we have:
\begin{equation*}
\nabla d_N = \left(\frac{\sqrt{N}\nabla \mu + \nabla Z_N}{V_N^{1/2}}\right) - 
\left(\frac{\sqrt{N}\mu+ Z_N}{2V_N^{3/2}}\right)\nabla V_N
\end{equation*}
and so (applying the vector product rule) it follows that
\begin{align*}
\nabla^2 d_N &= (\sqrt{N}\nabla \mu + \nabla Z_N)^T\left( -\frac{1}{2V_N^{3/2}}\nabla V_N\right) + \frac{1}{\sqrt{V_N}}(\sqrt{N}\nabla^2 \mu + \nabla^2 Z_N)\\
&\hspace{1cm}+ (\nabla V_N)^T\left(\frac{3(\sqrt{N}\mu+ Z_N)}{4V_N^{5/2}}\nabla V_N - \frac{\sqrt{N}\nabla \mu + \nabla Z_N}{2V_N^{3/2}}\right) - \frac{\sqrt{N}\mu+ Z_N}{2V_N^{3/2}}\nabla^2 V_N\\
&=  (\nabla \mu + \nabla Z_N/ \sqrt{N})^T\left( -\frac{1}{2(V_N/N)^{3/2}}\frac{\nabla V_N}{N}\right) + \frac{\nabla^2 \mu + \nabla^2 Z_N/\sqrt{N}}{\sqrt{V_N/N}}\\
&\hspace{1cm}+ (\nabla V_N/N)^T\left(\frac{3\mu + 3Z_N/\sqrt{N}}{4(V_N/N)^{5/2}}(\nabla V_N/N)- \frac{\nabla \mu + \nabla Z_N/\sqrt{N}}{2(V_N/N)^{3/2}}\right)\\  
&\hspace{1cm} +\frac{\mu+ Z_N/\sqrt{N}}{2(V_N/N)^{3/2}}\frac{\nabla^2 V_N}{N}. 
\end{align*}

\newpage
\section{Further simulations}\label{A:fs}
\subsection{1D simulations - mean function}\label{A:fsmean}
\subsubsection{Gaussian noise}
Our first set of supplementary simulations consists of 1D stationary Gaussian noise about two different types of peaks. In this setting we can leverage the stationarity of the fields to obtain good estimates for the variance and $ \Lambda $ because they are the same at every point. The peaks we use are sections of the pdfs of the $ \text{Beta}(1.5, 3)$ and $ \text{Beta}(1.5, 2)$ distributions. The first set of peaks are narrow and the second set are wider relative to the smoothness of the noise. The mean signal is calculated by repeating the peaks several times over the domain. The noise is obtained by smoothing white noise with a Gaussian kernel, with a FWHM ranging from $ 3 $ to $ 9 $ FWHM per voxel, to obtain a Gaussian SuRF $ \epsilon $ for each realization. We add these fields to the mean to obtain simulations with $ N \in \left\lbrace 20, 40, \dots, 180, 200 \right\rbrace $ and repeat this 1000 times in each setting.

\begin{figure}[h!]
	\begin{subfigure}[b]{0.32\textwidth}
		\centering
		\includegraphics[width=\textwidth]{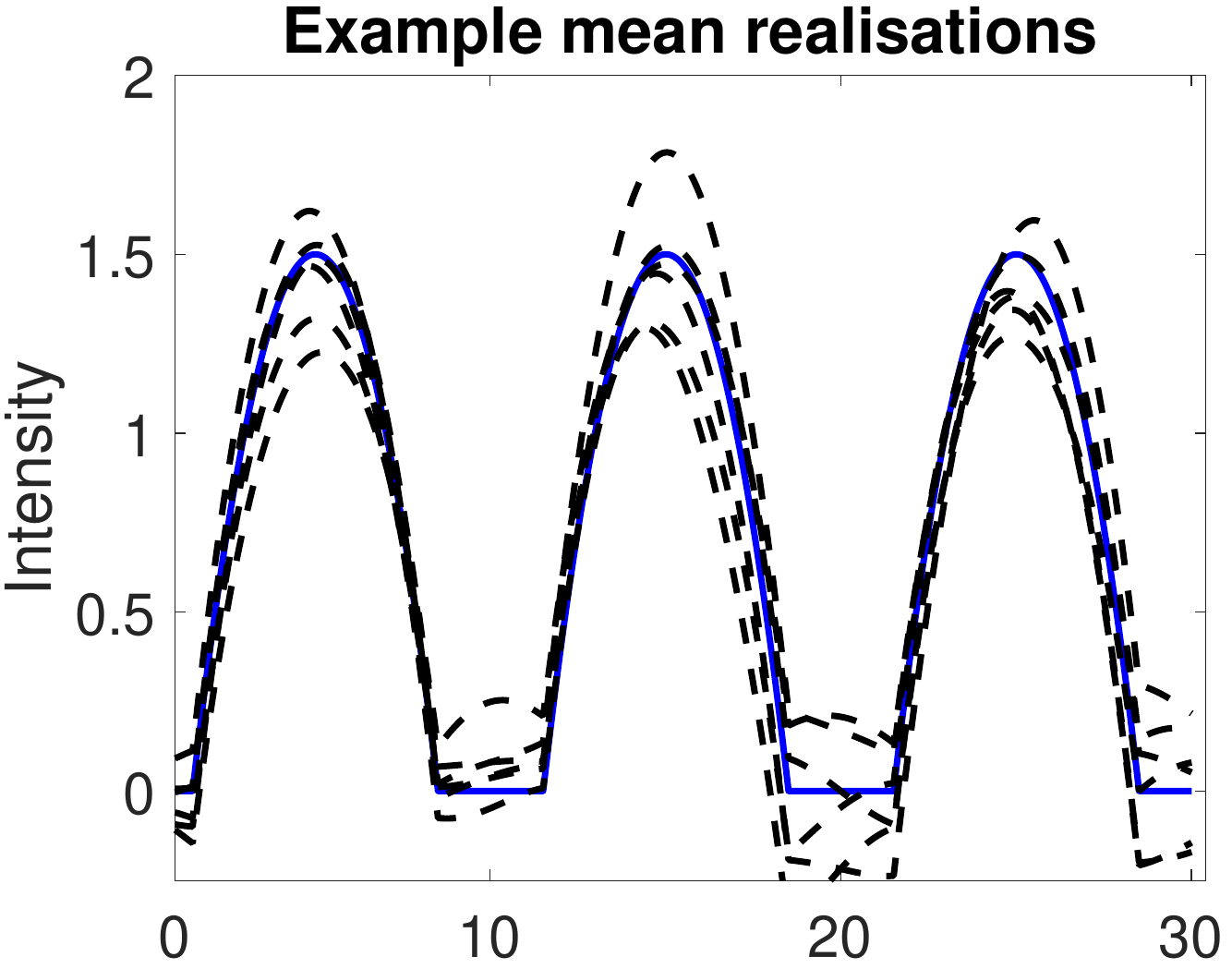} 
	\end{subfigure}
	\hfill
	\begin{subfigure}[b]{0.32\textwidth}  
		\centering 
		\includegraphics[width=\textwidth]{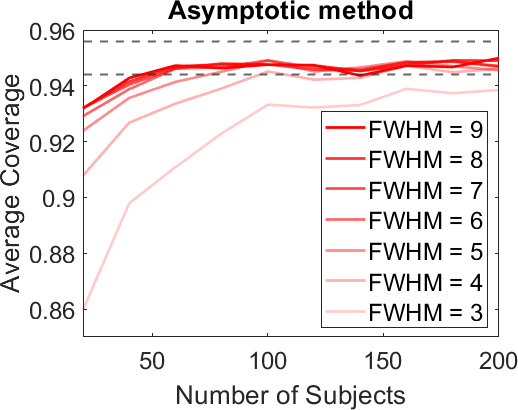} 
	\end{subfigure}
	\hfill
	\begin{subfigure}[b]{0.32\textwidth}  
		\centering 
		\includegraphics[width=\textwidth]{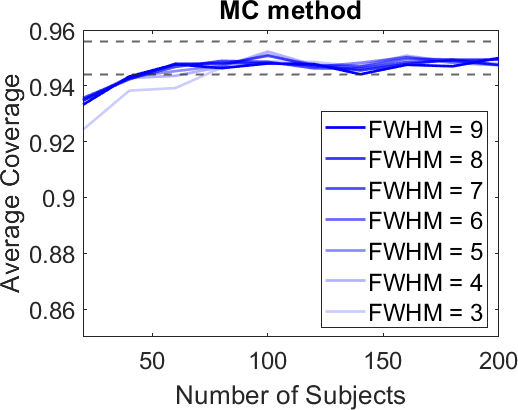} 
	\end{subfigure}
	\caption{Coverage of confidence intervals for the maximum of the mean for stationary, variance-one Gaussian noise added to the wide beta peak. The left panel contains the true mean and single realisations of the processes (which are the true mean plus smooth noise) and the mean of 50 i.i.d realisations (of fields generated with FWHM ) is shown in the lower left panel. The centre and right panels display the coverage of 95$ \% $ confidence intervals obtained using the asymptotic and Monte Carlo methods.
		Reasonable coverage is generally obtained for $ N \geq 60. $ From these graphs we can see that the Monte Carlo confidence regions have an improved finite sample coverage. As can be seen from the plots of the data, the peak of the means lie well within the image and so the simulations are not affected by edge effect issues.
	}\label{fig:statnarrow}
\end{figure}

\begin{figure}[h!]
	\begin{subfigure}[b]{0.32\textwidth}
		\centering
		\includegraphics[width=\textwidth]{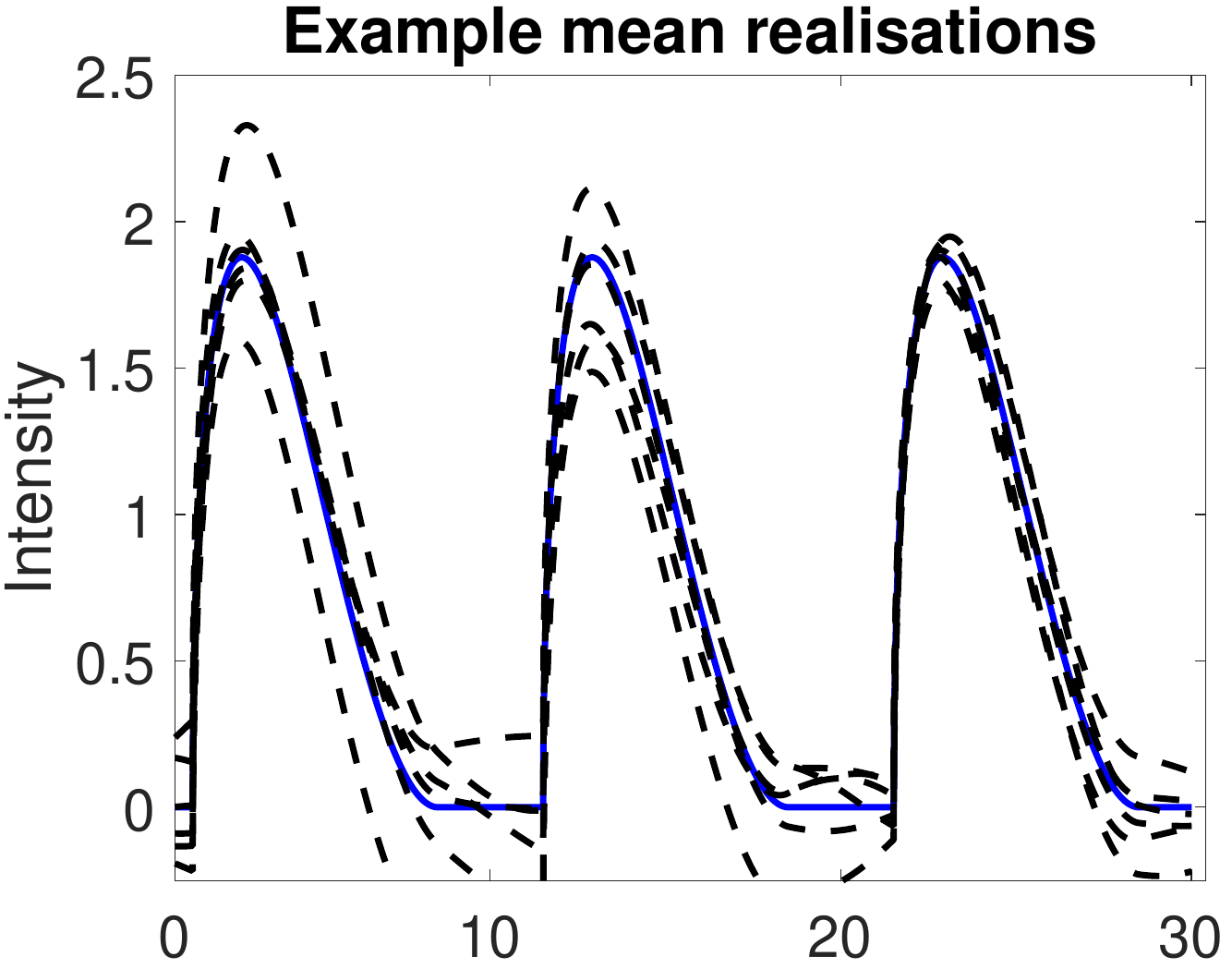} 
	\end{subfigure}
	\hfill
	\begin{subfigure}[b]{0.32\textwidth}  
		\centering 
		\includegraphics[width=\textwidth]{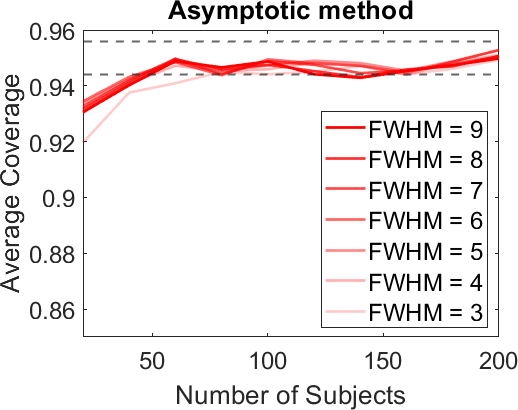} 
	\end{subfigure}
	\hfill
	\begin{subfigure}[b]{0.32\textwidth}  
		\centering 
		\includegraphics[width=\textwidth]{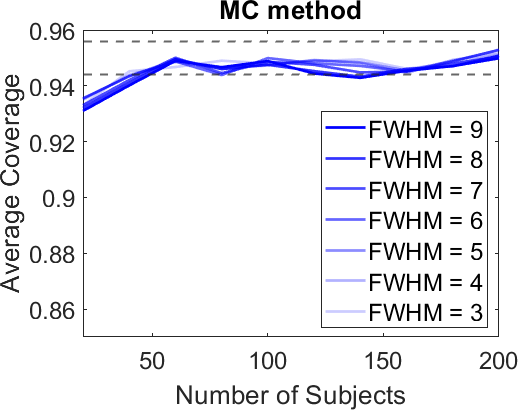} 
	\end{subfigure}
	\caption{Coverage of confidence intervals for the maximum of the mean for stationary, variance 1 Gaussian noise added to the narrower Beta peak, scaled such that Cohen's $ d $ is 1 at the peak. The layout of the plots is the same as in Figure \ref{fig:statnarrow}.}\label{fig:statwide}
\end{figure}

The results (shown for a scenario involving 3 peaks in Figures \ref{fig:statnarrow} and \ref{fig:statwide}) illustrate that, as the number of subjects increases, the coverage converges to the desired level. The dotted lines, in these and all other corresponding figures, give $ 95\% $ confidence bands and are obtained using the normal approximation to the binomial distribution. The coverage of the Monte Carlo confidence regions also converges asymptotically, at a slightly faster rate than that of the asymptotic method. 

The lower the FWHM of the noise relative to the shape of the peak, the larger the number of subjects that is needed to obtain the correct coverage. In many settings of interest high smoothness relative to the shape of the peak is a reasonable assumption, allowing us to obtain good coverage given available sample sizes. Here stationarity allows us to obtain better estimates of quantities involved (for both the asymptotic and Monte Carlo methods), such as $ \Lambda $, as we can average over the whole image. However, asymptotic coverage is achieved regardless of stationarity for both the asymptotic and Monte Carlo methods.

\begin{figure}[h!]
	\centering
	\begin{subfigure}[b]{0.32\textwidth}  
		\centering 
		\includegraphics[width=\textwidth]{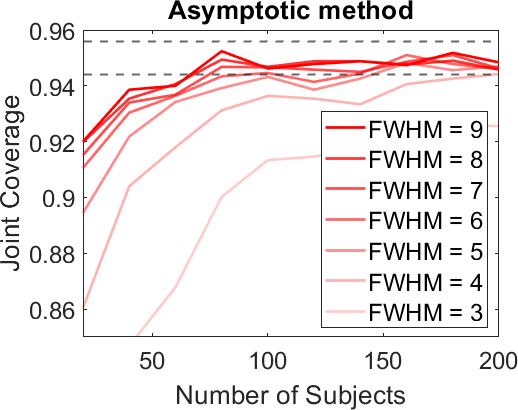} 
	\end{subfigure}
	\begin{subfigure}[b]{0.32\textwidth}  
		\centering 
		\includegraphics[width=\textwidth]{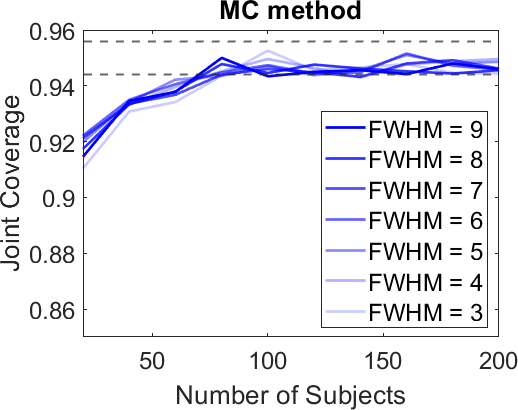} 
	\end{subfigure}
	\caption{Joint coverage of confidence intervals for the maximum of the mean for stationary, variance-one Gaussian noise added to the wide beta peak.
	}\label{fig:statwidejoint}
\end{figure}

\begin{figure}[h!]
	\centering
	\begin{subfigure}[b]{0.32\textwidth}  
		\centering 
		\includegraphics[width=\textwidth]{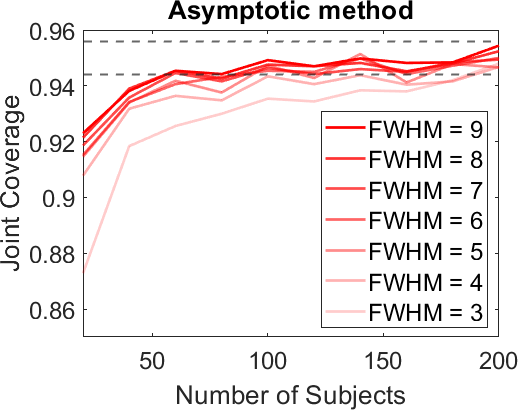} 
	\end{subfigure}
	\begin{subfigure}[b]{0.32\textwidth}  
		\centering 
		\includegraphics[width=\textwidth]{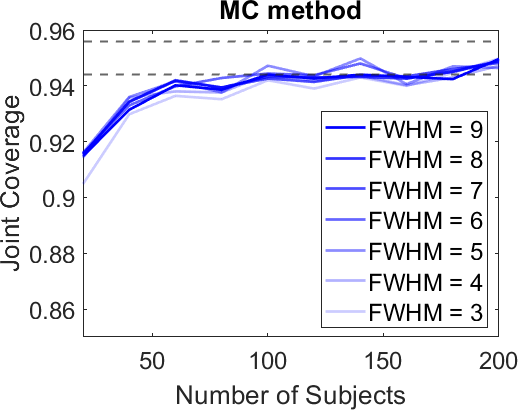} 
	\end{subfigure}
	\caption{Joint coverage of confidence intervals for the maximum of the mean for stationary, variance-one Gaussian noise added to the narrow beta peak.
	}\label{fig:statnarrowjoint}
\end{figure}

\newpage
\subsubsection{$ t $ noise}
Here we discuss a scenario where the noise is non-Gaussian. Instead of smoothing Gaussian white noise to obtain our SuRFs we instead smooth heavy tailed i.i.d noise that has been marginally generated from a $ t $-distribution with $ 3 $ degrees of freedom. After smoothing we add this noise to the same peaks as in the previous section. The results are similar, see Figures \ref{fig:twide} and \ref{fig:tnarrow}, but the methods take slightly longer to converge to the correct coverage as this is a more challenging setting.

\begin{figure}[h!]
	\centering
	\begin{subfigure}[b]{0.32\textwidth}  
		\centering 
		\includegraphics[width=\textwidth]{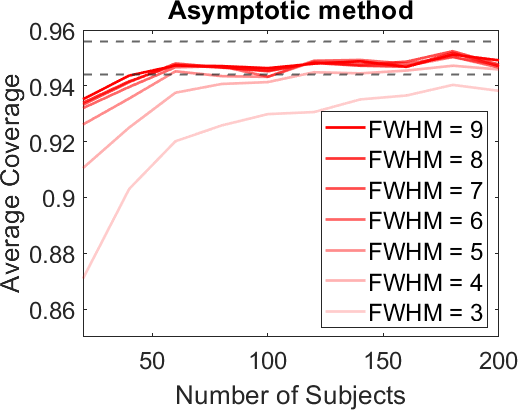} 
	\end{subfigure}
	\begin{subfigure}[b]{0.32\textwidth}  
		\centering 
		\includegraphics[width=\textwidth]{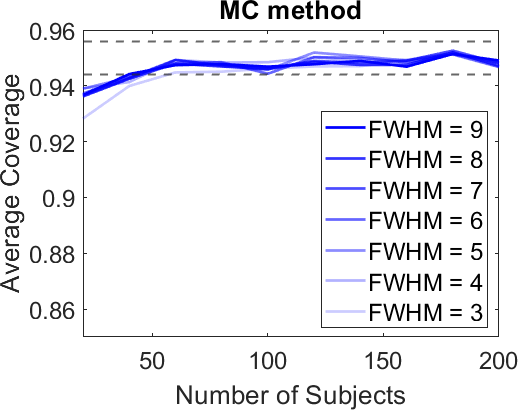} 
	\end{subfigure}\\
\vspace{0.25cm}
	\begin{subfigure}[b]{0.32\textwidth}  
		\centering 
		\includegraphics[width=\textwidth]{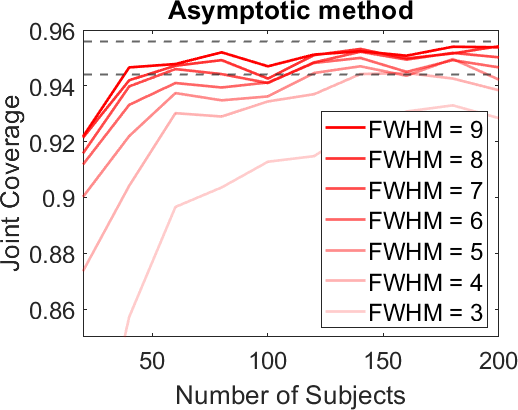} 
	\end{subfigure}
	\begin{subfigure}[b]{0.32\textwidth}  
		\centering 
		\includegraphics[width=\textwidth]{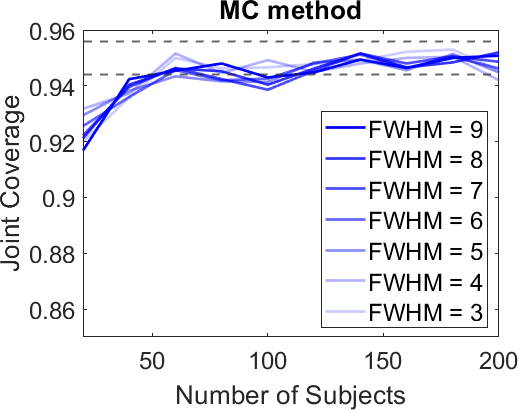} 
	\end{subfigure}
	\caption{Average and joint coverage of confidence intervals for the maximum of the mean for stationary, variance-one $ t_3 $ noise added to the wide beta peak.
	}\label{fig:twide}
\end{figure}

\begin{figure}[h!]
	\centering
	\begin{subfigure}[b]{0.32\textwidth}  
		\centering 
		\includegraphics[width=\textwidth]{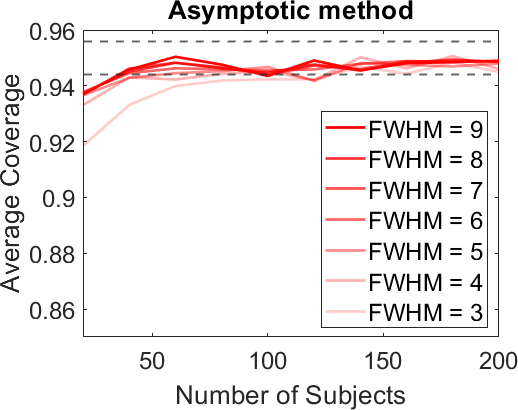} 
	\end{subfigure}
	\begin{subfigure}[b]{0.32\textwidth}  
		\centering 
		\includegraphics[width=\textwidth]{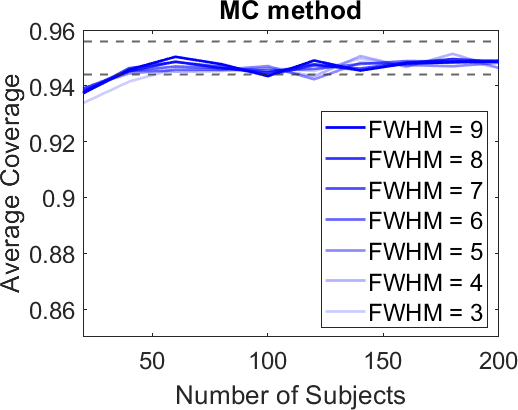} 
	\end{subfigure}\\
	\vspace{0.25cm}
	\begin{subfigure}[b]{0.32\textwidth}  
		\centering 
		\includegraphics[width=\textwidth]{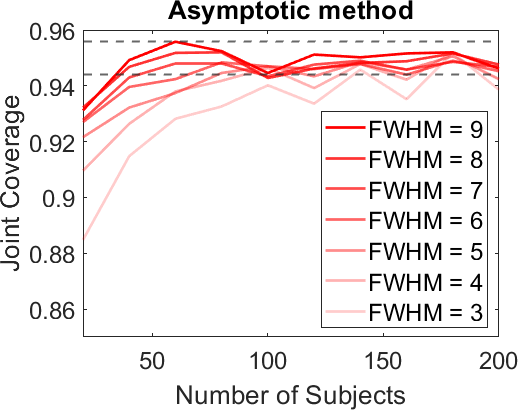} 
	\end{subfigure}
	\begin{subfigure}[b]{0.32\textwidth}  
		\centering 
		\includegraphics[width=\textwidth]{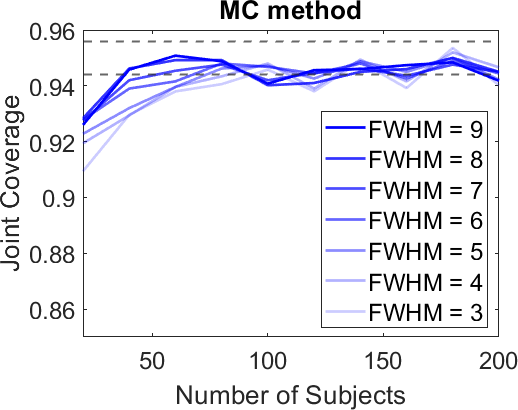} 
	\end{subfigure}
	\caption{Average and joint coverage of confidence intervals for the maximum of the mean for stationary, variance-one $ t_3 $ noise added to the narrow beta peak.
	}\label{fig:tnarrow}
\end{figure}

\newpage
\subsection{1D simulations - Cohen's $ d $}\label{A:fscd}
In this section we consider the same simulations as in the previous section - but infer on peaks of Cohen's $ d $ rather than the mean. The results are shown in Figure \ref{fig:t-Gnoise} for Gaussian noise and in Figure \ref{fig:t-tnoise} for t-noise. The results are similar to the asymptotic results for inferring on the mean (i.e. compare to the red curves in Figures \ref{fig:statnarrow}, \ref{fig:statwide}, \ref{fig:statwidejoint}, \ref{fig:statnarrowjoint}, \ref{fig:twide}, \ref{fig:tnarrow}). No Monte Carlo improvement is available for Cohen's $ d. $
\subsubsection{Gaussian noise}
\begin{figure}[h!]
	\centering
	\begin{subfigure}[b]{0.32\textwidth}  
		\centering 
		\includegraphics[width=\textwidth]{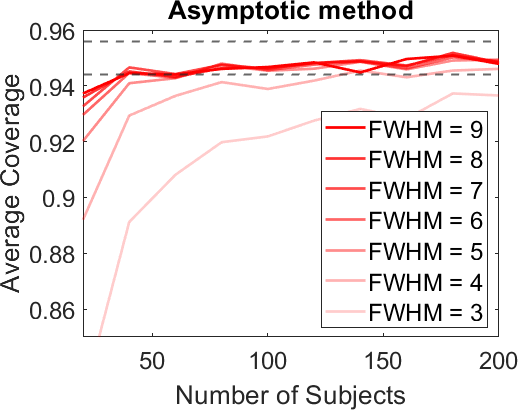} 
	\end{subfigure}
	\begin{subfigure}[b]{0.32\textwidth}  
		\centering 
		\includegraphics[width=\textwidth]{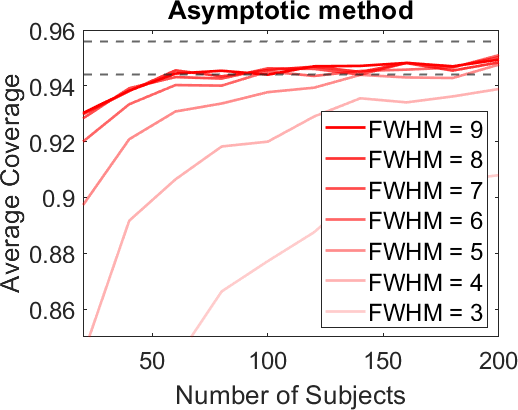} 
	\end{subfigure}\\
	\vspace{0.25cm}
	\begin{subfigure}[b]{0.32\textwidth}  
		\centering 
		\includegraphics[width=\textwidth]{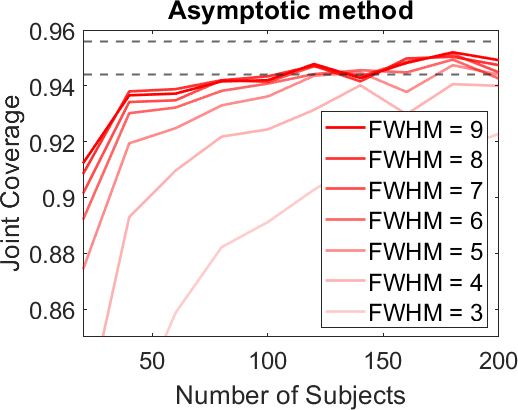} 
	\end{subfigure}
	\begin{subfigure}[b]{0.32\textwidth}  
		\centering 
		\includegraphics[width=\textwidth]{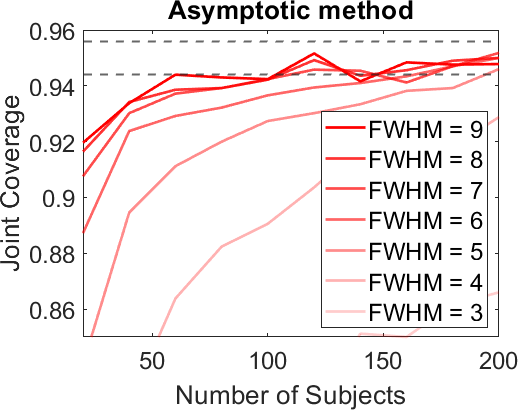} 
	\end{subfigure}
	\caption{Average and joint coverage of confidence intervals for the maximum of Cohen's $ d $ for stationary, variance-one $ t_3 $ noise added to the wide (left plots) and narrow (right plots) beta peak.
	}\label{fig:t-Gnoise}
\end{figure}
\newpage
\subsubsection{$ t $-noise}
\begin{figure}[h!]
	\centering
	\begin{subfigure}[b]{0.32\textwidth}  
		\centering 
		\includegraphics[width=\textwidth]{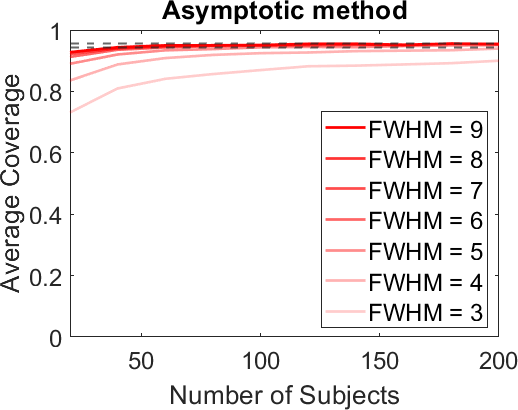} 
	\end{subfigure}
	\begin{subfigure}[b]{0.32\textwidth}  
		\centering 
		\includegraphics[width=\textwidth]{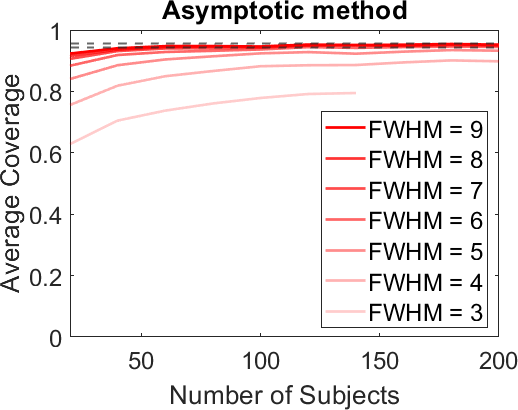} 
	\end{subfigure}\\
	\vspace{0.25cm}
	\begin{subfigure}[b]{0.32\textwidth}  
		\centering 
		\includegraphics[width=\textwidth]{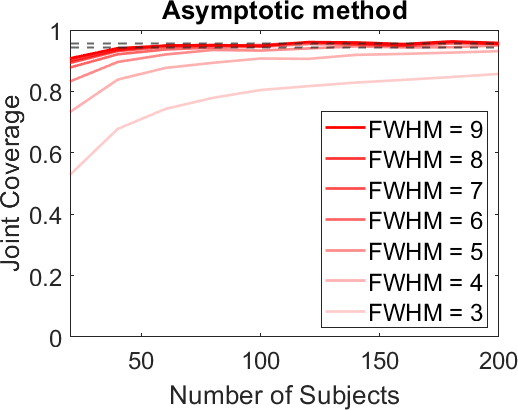} 
	\end{subfigure}
	\begin{subfigure}[b]{0.32\textwidth}  
		\centering 
		\includegraphics[width=\textwidth]{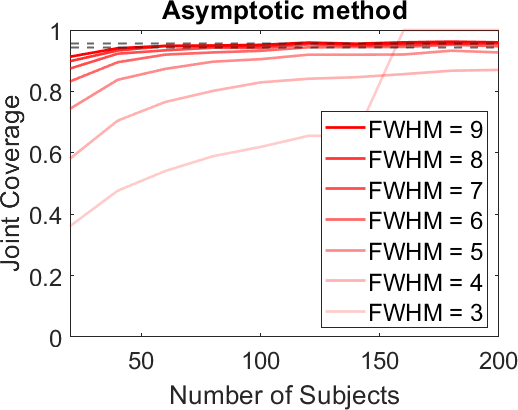} 
	\end{subfigure}
	\caption{Average and joint coverage of confidence intervals for the maximum of the mean for stationary, variance-one $ t_3 $ noise added to the wide (right plots) and narrow (left plots) beta peak.
	}\label{fig:t-tnoise}
\end{figure}
\newpage
\section{MEG data application}\label{S:MEGspectra}
We will now apply our methods in a real data setting. We have 1D MEG data from 79 subjects (from a single MEG node) and for each subject have around 6 minutes of time series data sampled at a rate of 240Hz (see \cite{Quinn2019} for details on the data and how it was collected). 

\subsection{Deriving MEG power spectra}\label{SS:MEGspectra}
We wish to infer on peaks in the power spectrum across subjects. We first derive the power spectra themselves. To do so, in order to infer on frequencies of interest in the data, we turn each time series into a periodogram using Welch's method (\cite{Welch1967}, \cite{Solomon1991}). For each subject $ n = 1, \dots, N $, let $ X_n(t)$ denote its regularly sampled time series. Given a segment length $ a $, we divide $ X_n(t) $ sequentially into segments of length $ a $ such that each overlaps by $ \lfloor \frac{a}{2} \rfloor $ data points (and ignore the final segment if this does not divide evenly). We window each segment using a Gaussian kernel (with parameter chosen so that the kernel tails down to a small value (0.05) at either end of the segment) to eliminate cutoff effects and then take the Fourier transform of these windowed segments. Let $ M_n $ denote the number of segments, and let $ X_{n,m} $ denote the $ m $th segment and let $ W \in \mathbb{R}^a $ be a window of Gaussian weights. Then 
\begin{equation*}
\mathcal{D}(W \cdot X_{n,m} ) = \mathcal{D}(W)\star\mathcal{D}(X_{n,m})
\end{equation*}
where $ \mathcal{D} $ denotes the (periodic) discrete Fourier transform, $ \cdot $ denotes pointwise multiplication and $ \star $ denotes convolution. In particular for $ x \in F := \left\lbrace \frac{240k}{a}: k \in \mathbb{Z} \right\rbrace  $,
\begin{equation}\label{eq:D}
\mathcal{D}(W \cdot X_{n,m})(x) = \sum_{y \in F} K(x-y) \mathcal{D}(X_{n,m})(y)
\end{equation}
where $ K $, is the discrete Fourier transform of $ W $ and is thus a Gaussian kernel. Since the Gaussian kernel is continuous (\ref{eq:D}) has a natural extension as a convolution field:
\begin{equation*}
\mathcal{D}_{c,n,m}(s) = \sum_{y \in F} K(s-y) \mathcal{D}(X_{n,m})(y)
\end{equation*}
defined on $ s \in [-120,120] $ Hz (it is in fact defined on all $ s \in \mathbb{R} $ but is periodic so we restrict to this bounded subset). In particular we can define $ \mathcal{P}_{n,m}(s) = \left\lVert \mathcal{D}_{c,n,m}(s) \right\rVert^2 $. Using \cite{Welch1967}'s approach we obtain the power spectrum field
\begin{equation*}
\mathcal{P}_n = \frac{10}{M_n} \sum_{m = 1}^{M_n} \log_{10}(\mathcal{P}_{n,m})
\end{equation*}
(where addition is performed pointwise) defined on $ s \in [-120,120] $ Hz. 

We apply the methods described above to derive power spectrum fields for each subject, using a segment length of one minute which corresponds to taking $ a = 240. $. A selection of these fields (and their mean) are shown in Figure \ref{fig:MEGmean}. In this setting the time series have varying length but all consist of around 70,000 time points, meaning that the time series are divided into segments of around 600 for each subject. As a result a large amount of averaging goes into calculating the power spectrum fields which are thus effectively Gaussian random fields; by the functional CLT.

\subsection{Power Spectrum Confidence Intervals}
Here we plot confidence intervals for peaks of the power spectra as described in Section \ref{SS:MEGmeg}.
\begin{figure}[h]
	\begin{center}
		\begin{subfigure}[b]{0.4\textwidth}
			\centering
			\includegraphics[width=\textwidth]{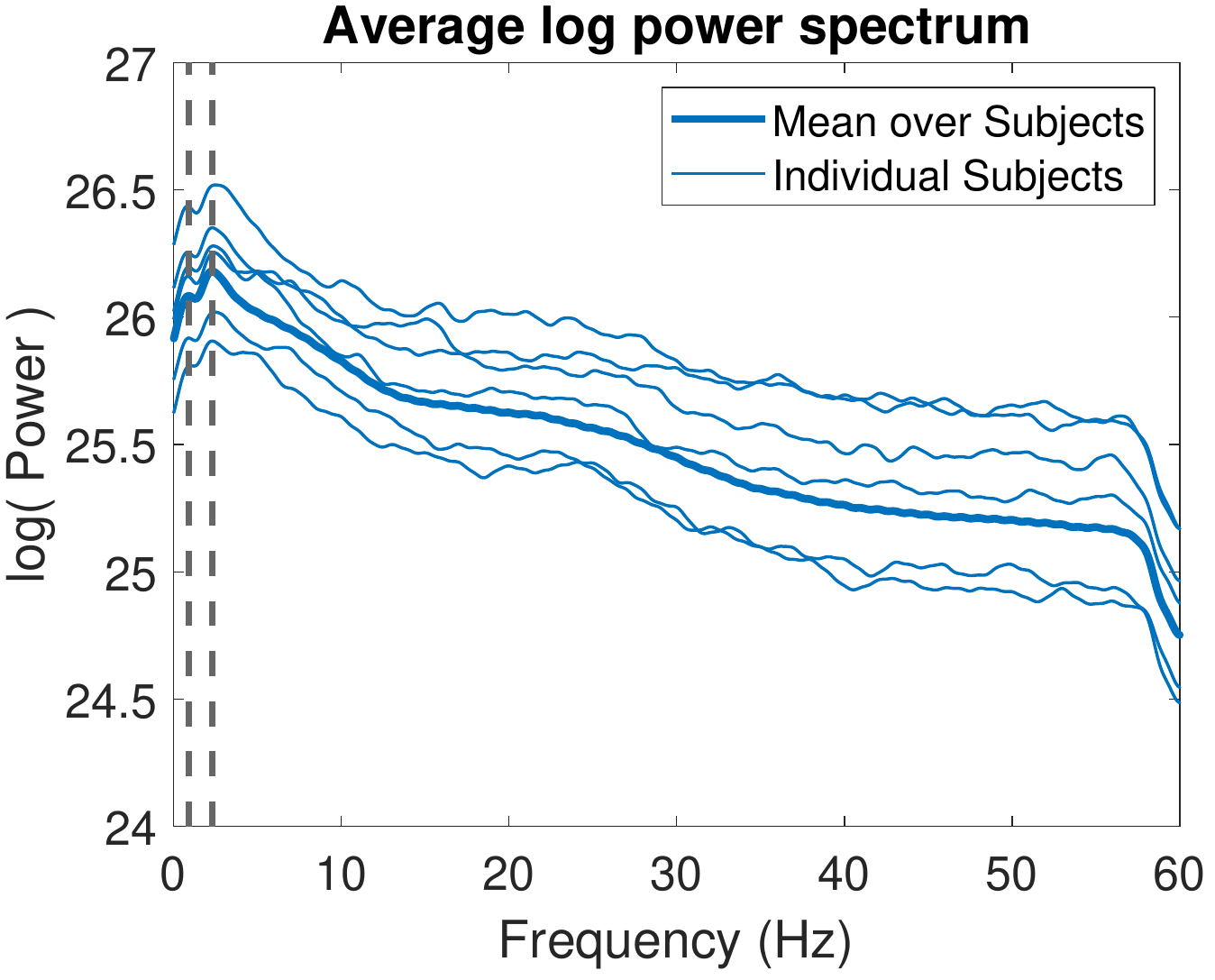}
		\end{subfigure}
		\hspace{1cm}
		\begin{subfigure}[b]{0.4\textwidth}
			\centering 
			\includegraphics[width=\textwidth]{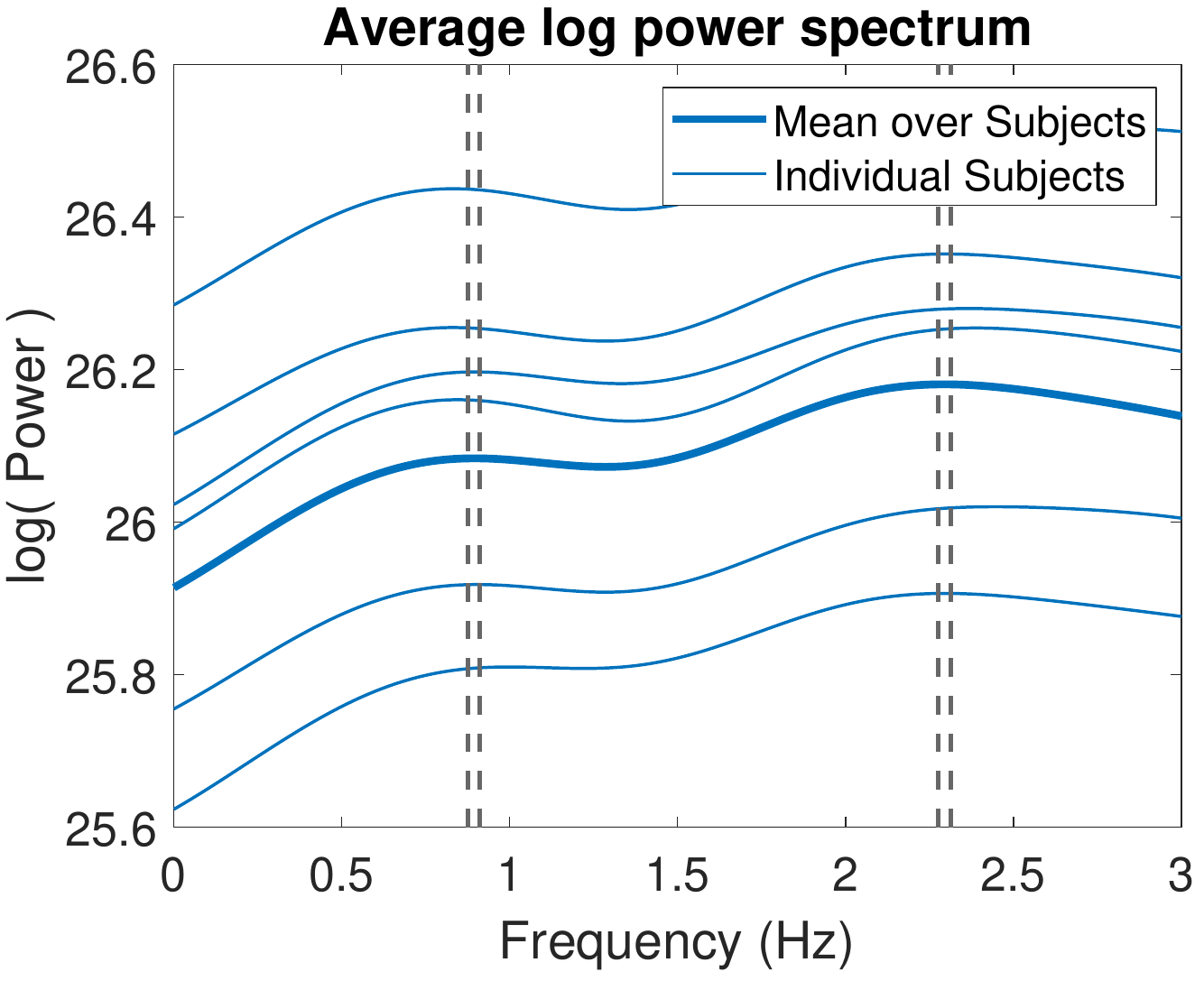}
		\end{subfigure}
	\end{center}
	\caption{Average power spectrum random fields across subjects and individual subject log power spectra are displayed in this figure. This is shown from frequencies of 0 to 60 Hz on the left and from 0 to 3 Hz on the right. The locations of the maxima of the observed mean are shown with dashed lines on the left and 95\% joint confidence intervals for the true peak locations are shown on the right. The top 2 peaks in the mean occur at 0.893 $ \pm \,0.017 $ Hz and 2.295 $ \pm \,0.019 $ Hz. }\label{fig:MEGmean}
\end{figure}
\begin{figure}[h]
	\begin{center}
		\includegraphics[width=0.4\textwidth]{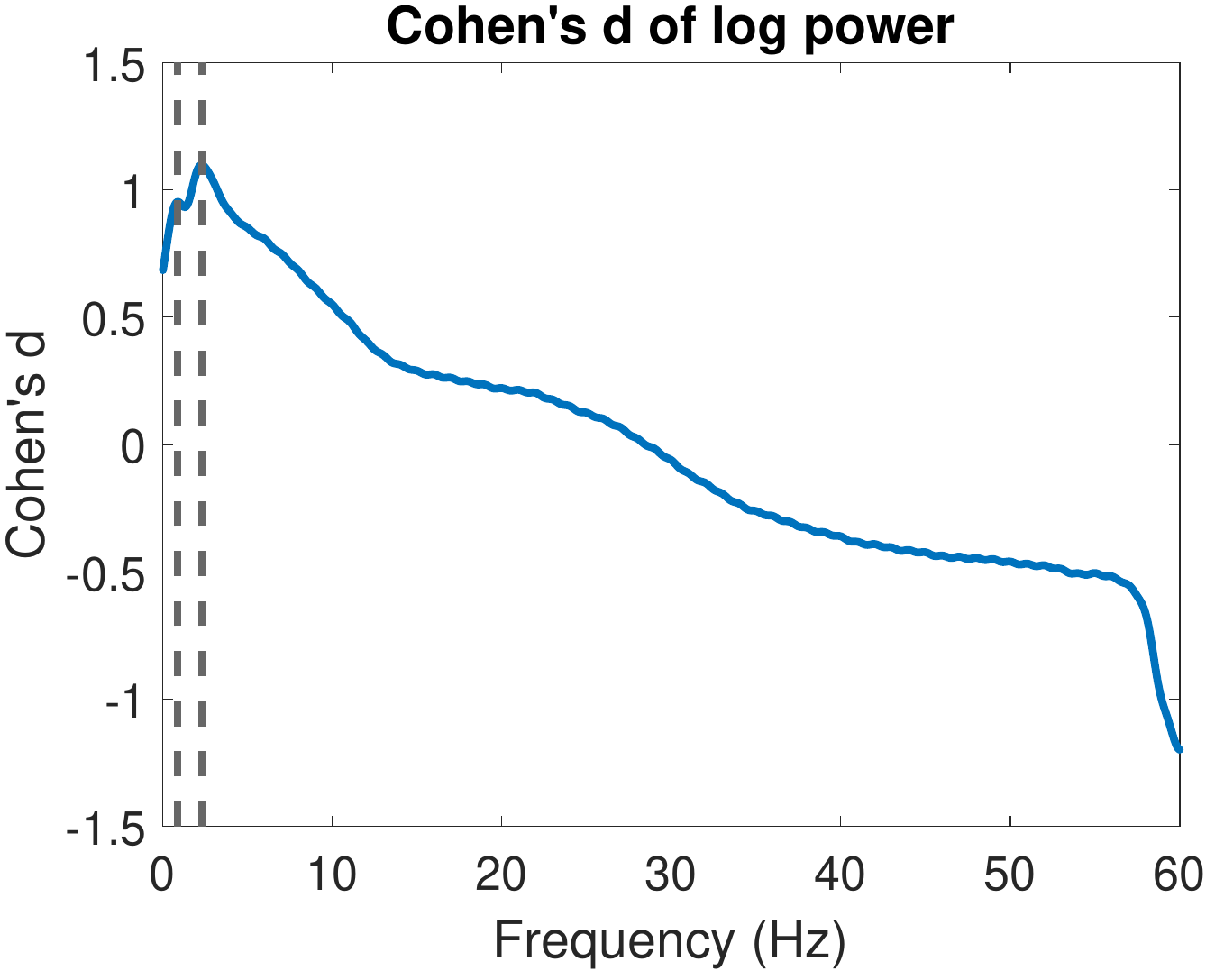}	
		\hspace{1cm}	\includegraphics[width=0.4\textwidth]{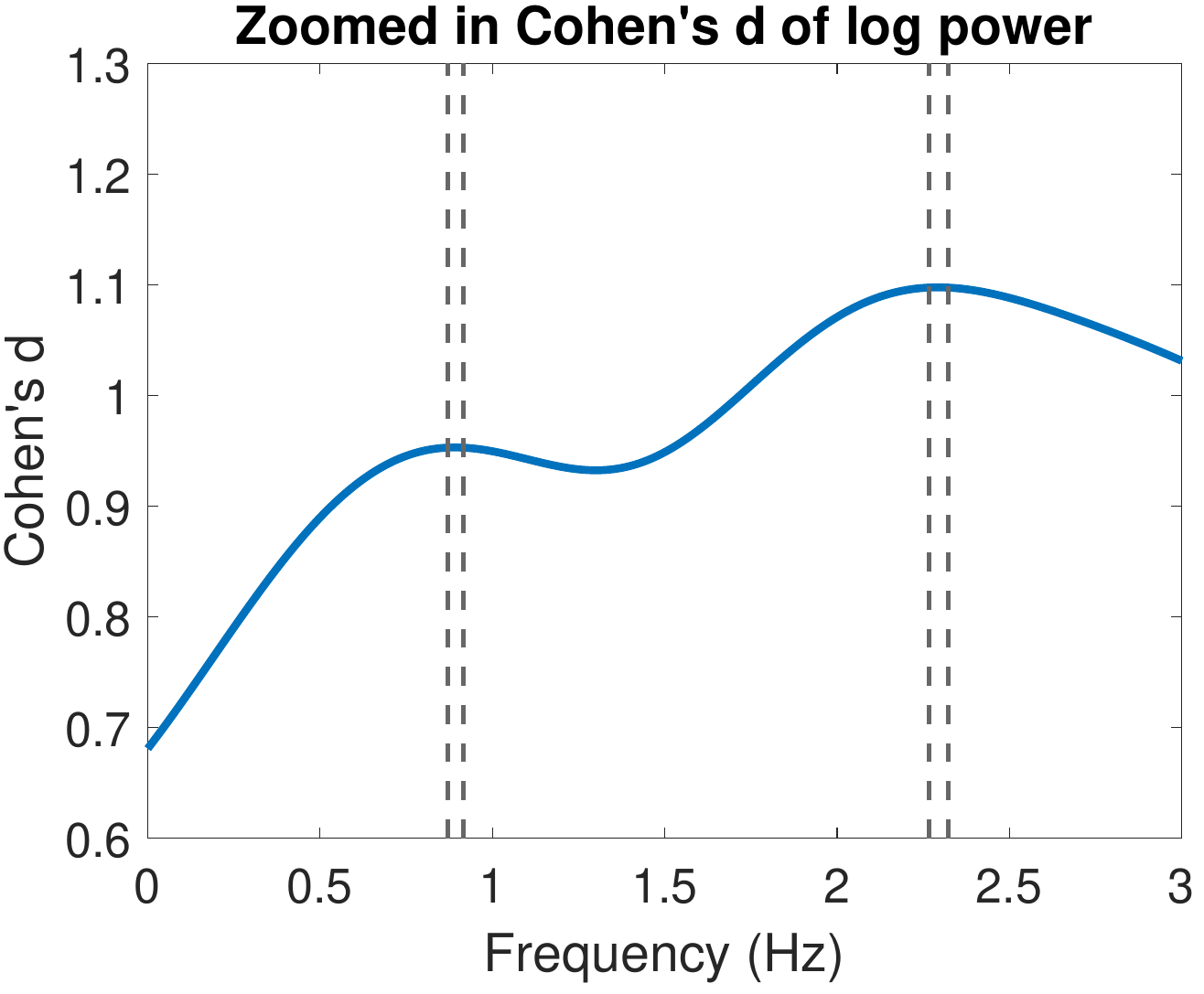}
	\end{center}
	\caption{95\% joint confidence intervals for the top two peaks of Cohen's $ d $ of the log power spectrum relative to its average.  The locations of the maxima of the observed mean are shown with dashed lines on the left and the confidence intervals for the true peak location are shown (in gray) on the right. The peaks occur at 0.893 $ \pm $ 0.023 Hz and 2.270 $ \pm $ 0.028 Hz .}\label{fig:MEGtstat}
\end{figure}
\newpage
\section{Cohen's $ d $ fMRI application}
Joint 95\% confidence regions for the peaks of Cohen's $ d $ in the fMRI data are shown in Figure \ref{fig:fMRI}.

\begin{figure}[H]
	\begin{center}
		\includegraphics[width=0.32\textwidth]{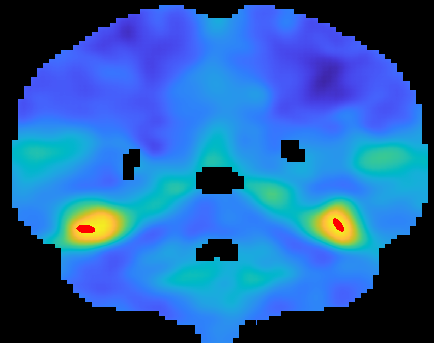}
		\includegraphics[width=0.32\textwidth]{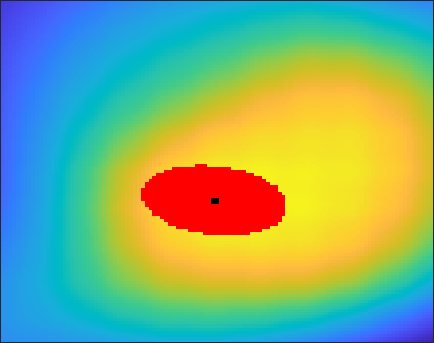}
		\includegraphics[width=0.32\textwidth]{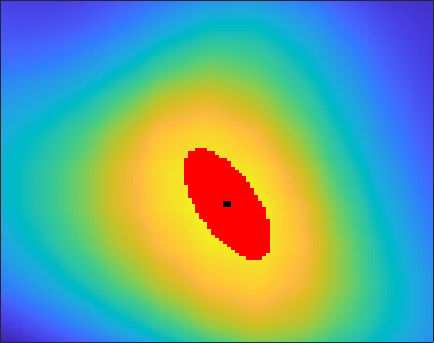}
	\end{center}
	\caption{Simultaneous confidence regions for peaks of Cohen's $ d $ derived from 125 images from the UK biobank. The $ 95\% $ confidence regions, corrected to allow for joint coverage over the two peaks, are shown in red: displayed over the Cohen's $ d $ of the images. The plots have the following interpretation: Left: whole brain slice. Middle and Right: zoomed in sections around each peak with a black dot indicating the location of each maximum. Note that no Monte Carlo approach is available for Cohen's $ d. $}\label{fig:cdfMRI}
\end{figure}

\section{Further Remarks}
\begin{rmk}
	If a random field $ f $ is $ L_1 $-Lipschitz on $ S $ then there exists an integrable  random variable $ L $ such that (given any $ s \in S $)
	\begin{equation*}
	\left| \sup_{t \in S}f(t) - f(s) \right| \leq L\sup_{t \in S}\left\lVert t-s \right\rVert = L{\rm diam}(S).
	\end{equation*}
	In particular, as $ S $ is bounded,
	\begin{equation*}
	\mathbb{E}| \sup_{t \in S}f(t)| \leq \mathbb{E}\left[ L \right]{\rm diam}(S) + \mathbb{E}\left| f(s) \right| <\infty. 
	\end{equation*}
\end{rmk}

\begin{rmk}
	It is in fact also possible to prove a CLT for Cohen's $ d $, see \cite{Telschow2020Delta} for further details.
\end{rmk}


\end{document}